\pdfoutput=1
\RequirePackage{ifpdf}
\ifpdf 
\documentclass[pdftex]{sigma}
\else
\documentclass{sigma}
\fi

\usepackage{verbatim}
\usepackage[all]{xy}

\newtheorem{Theorem}{Theorem}[section]

\newtheorem{Proposition}[Theorem]{Proposition}
\newtheorem{Lemma}[Theorem]{Lemma}
\newtheorem{Corollary}[Theorem]{Corollary}

\newtheorem{Claim}[Theorem]{Claim}

\theoremstyle{definition}

\newtheorem{Definition}[Theorem]{Definition}
\newtheorem{Remark}[Theorem]{Remark}
\newtheorem{Construction}[Theorem]{Construction}
\newtheorem{Example}[Theorem]{Example}

\newcommand{\ZZ}{\mathbf{Z}}

\newcommand{\RR}{\mathbb{R}}
\newcommand{\RRR}{\mathbf{R}}

\newcommand{\PP}{\mathbb{P}}

\newcommand{\NNNN}{\operatorname{N}}

\newcommand{\HH}{\operatorname{\mathcal{HH}}}
\newcommand{\HHC}{\operatorname{\mathcal{HH}}^\bullet}
\newcommand{\HHH}{\operatorname{\mathcal{HH}}_\bullet}

\newcommand{\FF}{\mathcal{F}}

\newcommand{\AAA}{\mathcal{A}}
\newcommand{\BBB}{\mathcal{B}}
\newcommand{\LFun}{\operatorname{Fun}^{\textup{L}}}

\newcommand{\Diskd}{\Disk^\dagger}
\newcommand{\ZZZ}{\mathbb{Z}}

\newcommand{\MM}{\mathcal{M}}

\newcommand{\DDD}{\mathcal{D}}

\newcommand{\CCC}{\mathcal{C}}
\newcommand{\RPerf}{\operatorname{RPerf}}

\newcommand{\PR}{\operatorname{Pr}^{\textup{L}}}

\newcommand{\SSSS}{\mathbb{S}}

\newcommand{\OO}{{\mathcal{O}}}

\newcommand{\MMM}{\mathcal{M}}

\newcommand{\PPP}{\mathcal{P}}

\newcommand{\Hom}{\operatorname{Hom}}

\newcommand{\Comp}{\operatorname{Comp}}

\newcommand{\Perf}{\operatorname{Perf}}

\newcommand{\SP}{\operatorname{Sp}}
\newcommand{\SPS}{\operatorname{Sp}^{\Sigma}}

\newcommand{\Mod}{\operatorname{Mod}}

\newcommand{\SSS}{\mathcal{S}}

\newcommand{\colim}{\operatorname{colim}}
\newcommand{\CAT}{\operatorname{Cat}}
\newcommand{\Cat}{\textup{Cat}_{\infty}}

\newcommand{\Map}{\operatorname{Map}}

\newcommand{\Fun}{\operatorname{Fun}}
\newcommand{\Alg}{\operatorname{Alg}}

\newcommand{\End}{\operatorname{End}}

\newcommand{\sSet}{\operatorname{Set}_{\Delta}}

\newcommand{\FIN}{\operatorname{Fin}_\ast}

\newcommand{\wCat}{\widehat{\textup{Cat}}_{\infty}}

\newcommand{\CAlg}{\operatorname{CAlg}}

\newcommand{\LM}{\operatorname{\mathcal{LM}}}

\newcommand{\Ind}{\operatorname{Ind}}

\newcommand{\AADG}{\mathsf{A}}
\newcommand{\PRST}{\PR_{\textup{St}}}

\newcommand{\assoc}{\operatorname{As}}
\newcommand{\triv}{\operatorname{Triv}}
\newcommand{\cyl}{\mathbf{Cyl}}
\newcommand{\cylm}{\overline{\mathbf{Cyl}}}
\newcommand{\dcylm}{\overline{\mathbf{DCyl}}}
\newcommand{\eone}{\mathbf{E}_1}
\newcommand{\etwo}{\mathbf{E}_2}
\newcommand{\dtwo}{\mathbf{D}}

\newcommand{\dcyl}{\mathbf{DCyl}}

\newcommand{\KS}{\mathbf{KS}}
\newcommand{\CYL}{\operatorname{Cyl}}
\newcommand{\DCYL}{\operatorname{DCyl}}
\newcommand{\CYLM}{\overline{\operatorname{Cyl}}}
\newcommand{\DCYLM}{\overline{\operatorname{DCyl}}}
\newcommand{\mul}{\operatorname{Mult}}
\newcommand{\cbar}{C_M}
\newcommand{\eonem}{\overline{\mathbf{E}}_1}
\newcommand{\LMod}{\operatorname{LMod}}
\newcommand{\RMod}{\operatorname{RMod}}
\newcommand{\ST}{\operatorname{\mathcal{S}t}}
\newcommand{\Mfld}{\mathsf{Mfld}}
\newcommand{\Disk}{\mathsf{Disk}}
\newcommand{\lsr}{\big\langle S^1\big\rangle}
\newcommand{\OP}{\operatorname{Oper}_\infty}
\newcommand{\CATI}{\operatorname{Cat}_\infty}
\newcommand{\AB}{\mathbb{A}}

\newcommand{\Proof}{{\sl Proof.}\quad}

\begin{document}
\allowdisplaybreaks

\newcommand{\arXivNumber}{1904.02359}

\renewcommand{\thefootnote}{}

\renewcommand{\PaperNumber}{097}

\FirstPageHeading

\ShortArticleName{Differential Calculus of Hochschild Pairs for Infinity-Categories}

\ArticleName{Differential Calculus of Hochschild Pairs \\ for Infinity-Categories\footnote{This paper is a~contribution to the Special Issue on Primitive Forms and Related Topics in honor of~Kyoji Saito for his 77th birthday. The full collection is available at \href{https://www.emis.de/journals/SIGMA/Saito.html}{https://www.emis.de/journals/SIGMA/Saito.html}}}

\Author{Isamu IWANARI}

\AuthorNameForHeading{I.~Iwanari}

\Address{Mathematical Institute, Tohoku University, 6-3 Aramakiaza, Sendai, Miyagi, 980-8578, Japan}
\Email{\href{mailto:isamu.iwanari.a2@tohoku.ac.jp}{isamu.iwanari.a2@tohoku.ac.jp}}

\ArticleDates{Received February 25, 2020, in final form September 04, 2020; Published online October 02, 2020}

\Abstract{In~this paper, we provide a~conceptual new construction of~the algebraic structure on the pair of~the Hoch\-schild cohomology spectrum (cochain complex) and Hoch\-schild homology spectrum, which is analogous to the structure of~calculus on a~manifold. This algebraic structure is encoded by~a two-colored operad introduced by~Kontsevich and Soibelman. We~prove that for a~stable idempotent-complete infinity-category, the pair of its Hoch\-schild cohomology and homology spectra naturally admits the structure of~algebra over the operad. Moreover, we prove a~generalization to the equivariant context.}


\Classification{16E40; 18N60; 18M60}

\begin{flushright}
{\it To Kyoji Saito on his 77th birthday}
\end{flushright}

\renewcommand{\thefootnote}{\arabic{footnote}}
\setcounter{footnote}{0}

\section{Introduction}

Let $M$ be a~smooth real manifold.
Let $T_M^\bullet=\bigoplus_{p\ge0}\wedge^{p}T_M$
and $\Omega_{M}^{\bullet}=\bigoplus_{q\ge0}\Omega_M^q$
denote the graded vector space of~multivector fields
and differential forms on $M$, respectively.
By convention, $T_M^p=\wedge^{p}T_M$ has homological degree $-p$
while $\Omega^q_M$ has homological degree $q$:
we adopt the reverse grading.
There are several algebraic structures on the pair $(T_M^\bullet,\Omega_{M}^{\bullet})$.
The graded vector space $T_M^\bullet$ has a~graded commutative
(associative) product given by~$\wedge$. Further,
the~shifted graded vector space $T_M^{\bullet+1}$ inherits the structure of~a~graded Lie algebra defined by~the
Schouten--Nijenhuis bracket $[-,-]$.
On the other hand, $\Omega_{M}^{\bullet}$ has the
de Rham differential ${\rm d}_{\rm DR}$
(we do not consider the obvious graded commutative algebra structure on $\Omega_{M}^{\bullet}$ because it is irrelevant to the noncommutative context).
Since $\Omega_M^{q}$ is the dual vector space of~$T^q_M$, the
contraction morphisms $T^p_M\otimes \Omega_M^q\to \Omega_M^{q-p}$
give rise to a~$(T_M^{\bullet},\wedge)$-module structure on $\Omega_{M}^{\bullet}$:
\begin{gather*}
i\colon\ (T_M^{\bullet},\wedge)\otimes \Omega^{\bullet}_{M}\to \Omega_{M}^{\bullet},
\end{gather*}
where we regard $(T_M^{\bullet},\wedge)$ as the graded algebra determined by~$\wedge$.
The Lie derivative on $M$ defines a~Lie algebra action of~$\big(T_M^{\bullet+1},[-,-]\big)$
 on $\Omega_M^\bullet$:
\begin{gather*}
l\colon\ \big(T_M^{\bullet+1},[-,-]\big)\otimes \Omega_M^\bullet \to \Omega_M^\bullet.
\end{gather*}
The tuple
$(\wedge, [-,-], {\rm d}_{\rm DR}, i,l)$
constitutes fundamental calculus operations on the manifold $M$.
These operations are subject under certain relations such as
${\rm d}_{\rm DR}^2=0$, the
Cartan homotopy/magic formula, the compatibility between the Lie algebra
action $l$ and the de Rham differential
${\rm d}_{\rm DR}$, and so on. If~$X$ is a~smooth algebraic variety over a~field
of~characteristic zero,
the pair $(\mathcal{T}_X^\bullet,\Omega_{X}^\bullet)$ of~sheaves
of~multivector fields (given by~the
exterior products of~the tangent sheaf) and differential forms admit such an~algebraic
structure of~calculus.

Let us shift our interest to noncommutative algebraic geometry
in which stable $\infty$-categories, (pretriangulated)
differential graded (dg) categories or the likes play the roles of~fundamental geometric objects.
From the Hoch\-schild--Kostant--Rosenberg theorem,
an analogue of~$(\mathcal{T}_X^\bullet,\Omega_{X}^\bullet)$
is the Hoch\-schild pair, that is,
the pair of~Hoch\-schild cohomology (cochain complex) and~Hoch\-schild homology (chain complex).
In~\cite[Section~11.2]{KS},
Kontsevich and Soibelman introduced a~two-colored topological operad which we shall refer to as the Kontsevich--Soibelman operad and denote here by~$\KS$.
It~can be used to encode all the structures on
the Hoch\-schild pair, which
are analogous to the structure of~calculus.
The operad $\KS$ generalizes the little $2$-disks (cubes) operad, i.e., the $\etwo$-operad.
It~contains two colors, $D$ and $C_M$, such that
the full suboperad spanned by~the color $D$ is the $\etwo$-operad.
In~\cite{DTT}, it was shown that a~combinatorial~dg operad~$\mathcal{KS}_{{\rm comb}}$,
which is a~certain dg version of~$\KS$, acts
on the pair of~the Hoch\-schild cochain complex and the Hoch\-schild chain complex
of~an associative algebra ($A_\infty$-sense),
and in~\cite{H} an~algebraic structure on the Hoch\-schild pair over $\KS$
was constructed for a~ring spectrum by~means of~the Swiss-cheese operad conjecture \cite{Thomas}.

In~this paper, we provide a~conceptual new construction of~the structure
of~an algebra over $\KS$ that yields the following result (see Theorem~\ref{mainconstruction}):

\begin{Theorem}\label{main1}
Let $R$ be a~commutative ring spectrum.
Let $\CCC$ be a~small $R$-linear stable
ind\-em\-potent-complete $\infty$-category.
Let $\HHC(\CCC)$ be the Hoch\-schild cohomology $R$-module spectrum and
$\HHH(\CCC)$ the Hoch\-schild homology $R$-module spectrum.
Then the pair $(\HHC(\CCC),\HHH(\CCC))$ is~promoted to
an algebra over $\KS$, namely, it is
a $\KS$-algebra in the $\infty$-category of~$R$-module spectra $\Mod_R$.
\end{Theorem}

The structure of~the $\KS$-algebra on
$(\HHC(\CCC),\HHH(\CCC))$ in Theorem~\ref{main1} induces
the action morphism $u\colon\HHC(\CCC)\otimes \HHH(\CCC)\to \HHH(\CCC)$,
which is a~counterpart to $i$ in the aforementioned tuple.
In~the classical differential graded algebraic situation,
we further prove that $u$ can be described by~means of~well-known algebraic constructions.
Thus, $u$ can be considered
a noncommutative contraction morphism, see Section~\ref{actionsec}, Theorem~\ref{maincontraction} and Proposition~\ref{explicitdg}.
In~other words, the underlying morphism given by~actions is an~expected one.
By considering Cartan's homotopy formula built in $\KS$,
an analogue $L\colon\HHC(\CCC)[1]\otimes \HHH(\CCC)\to \HHH(\CCC)$ of~the Lie derivative map $l$ is also an~expected morphism, cf.~Remark~\ref{otheraction}.

In~Section~\ref{strategysec},
we briefly describe the idea and approach of~our
construction, which is based on a~simple observation.
To achieve this,
we prove the following (see Corollary~\ref{algebraequivalence} for details):

\begin{Theorem}\label{main2}Let $\MM^\otimes$ be a~symmetric monoidal $\infty$-category
such that it admits small colimits and the tensor product functor $\otimes\colon\MM\times \MM\to \MM$
preserves small colimits separately in each variable.
$($The typical examples of~$\MM^\otimes$ we should keep in mind are the $\infty$-category of~spectra, and the
derived $\infty$-category of~vector spaces.$)$

Let $\Fun\big(BS^1\!,\MM\big)\!$ be the functor category from the classifying space
$BS^1$ of~the circle~$S^1$~to~$\MM$, which inherits a~pointwise symmetric monoidal structure from the structure on $\MM^\otimes$. Namely,
an object of~$\Fun\big(BS^1,\MM\big)$ can be viewed as an~object $M$ of~$\MM$
equipped with an $S^1$-action.
Let $\Alg_{\KS}(\MM)$ be the $\infty$-category of~$\KS$-algebras
in $\MM$.
Let $\Alg_{\etwo}(\MM)$ be the $\infty$-cate\-gory of~$\etwo$-algebras
$($i.e., algebras over the little $2$-disks operad$)$ in $\MM$.
Similarly, $\Alg_{\assoc}\big(\Fun\big(BS^1,\MM\big)\big)$
denotes the $\infty$-category of~associative algebras $($$\eone$-algebras$)$
in $\Fun\big(BS^1,\MM\big)$.
We~denote by~$\LMod\big(\Fun\big(BS^1,\MM\big)\big)$ the $\infty$-category of~pairs $(A,M)$ such that $A\in \Alg_{\assoc}\big(\Fun\big(BS^1,\MM\big)\big)$
and $M$ is a~left $A$-module object in $\Fun\big(BS^1,\MM\big)$.

Then there exists a~canonical equivalence of~$\infty$-categories{\samepage
\begin{gather*}
\Alg_{\KS}(\MM)\simeq \Alg_{\etwo}(\MM)\times_{\Alg_{\assoc}(\Fun(BS^1,\MM))}\LMod\big(\Fun\big(BS^1,\MM\big)\big).
\end{gather*}
See Section~{\rm \ref{operadsec}} for the definition of~the fiber product on the right-hand side.}
\end{Theorem}

This result means that $\Alg_{\etwo}(\MM)$, $\Alg_{\assoc}\big(\Fun\big(BS^1,\MM\big)\big)$
and $\LMod\big(\Fun\big(BS^1,\MM\big)\big)$ form building blocks for $\KS$-algebras. This allows us to describe the structure of~a~$\KS$-algebra
as a~collection of~more elementary algebraic data involving associative algebras, left modules, circle actions, and $\etwo$-algebras.
As for $\etwo$-algebras, thanks to Dunn additivity theorem
for $\infty$-operads proved by~Lurie \cite{HA}, a~canonical equivalence
$\Alg_{\etwo}(\MM)\simeq \Alg_{\assoc}(\Alg_{\assoc}(\MM))$ exists.
While we make use of~Theorem~\ref{main2} in the construction process, it would be generally useful in the theory of~$\KS$-algebras since
the notion of~$\KS$-algebras is complicated as is.
For example, when $\MM$ is the derived $\infty$-category $\DDD(k)$ of~vector spaces over a~field $k$ of~characteristic zero,
i.e., in the differential graded context, there is a~quite elementary
interpretation.
One may take $\Fun\big(BS^1,\DDD(k)\big)$ to be the $\infty$-category obtained from the category of~mixed complexes in the sense of~Kassel (see, e.g.,~\cite{L})
by~localizing quasi-isomorphisms.
Therefore, an~object of~$\Alg_{\assoc}\big(\Fun\big(BS^1,\DDD(k)\big)\big)$ may be regarded
as an~associative algebra in the monoidal ($\infty$-)ca\-tegory of~mixed complexes.
Objects of~$\LMod\big(\Fun\big(BS^1,\DDD(k)\big)\big)$ can be described in a~similar way.
Moreover, dg $\etwo$-operad is
formal in characteristic zero.

As we will describe in the following section,
our method consists of~only natural procedures.
In~particular, by~contrast with previous work, it does not
involve/use complicated resolutions
of~operads or genuine chain complexes.
Thus, we hope that our proposed approach can be applicable
 to other settings and generalizations
such as $(\infty,n)$-categories.
Indeed, the method allows us to prove
an equivariant generalization of~Theorem~\ref{main1}
(see Theorem~\ref{mainequivariant}):

\begin{Theorem}\label{main3}
Let $G$ be a~group object in the $\infty$-category $\SSS$ of~spaces,
that is, a~\mbox{group-like} \mbox{$\eone$-space}.
Let $\CCC$ be a~small $R$-linear stable
idempotent-complete $\infty$-category.
Suppose that $G$ acts on $\CCC$ $($namely, it gives a~left action$)$.
Then $(\HHC(\CCC),\HHH(\CCC))$ is promoted to
a~\mbox{$\KS$-algebra} in~$\Fun(BG,\Mod_R)$.
Namely, $(\HHC(\CCC),\HHH(\CCC))$ is a~$\KS$-algebra in
$\Mod_R$, which comes equip\-ped with a~left action of~$G$.
\end{Theorem}

We~would like to invite the reader's attention to the noteworthy features of~our method:
\begin{itemize}\itemsep=0pt
\item Our construction of~the structure of~an algebra
on the Hoch\-schild pair over $\KS$
 starts with an~$R$-linear stable $\infty$-category $\mathcal{C}$.
 Consequently,
if we have an~equivalence $\CCC\simeq \CCC'$,
we~have a~canonical equivalence $(\HHC(\CCC),\HHH(\CCC))\simeq (\HHC(\CCC'),\HHH(\CCC'))$ as algebras over $\KS$ (see Remark~\ref{remarkfunctoriality}).
Consider the situation that the associative algebra $A$ and $A'$
in the $\infty$-category of~$R$-modules have the equivalent module
category $\LMod_A$ and $\LMod_{A'}$, that is, $A$ and $A'$
are (derived) Morita equivalent to one another.
Here $\LMod_A$ and $\LMod_{A'}$ denote the $\infty$-categories
of~left $A$-module spectra and left $A'$-module spectra, respectively
(cf.~Section~\ref{NC}).
Then $\LMod_A\simeq \LMod_{A'}$ induces the canonical equivalence
of~Hoch\-schild pairs as algebras over $\KS$.
In~other words,
our method provides a~natural Morita invariant structure.
This invariance has a~fundamental importance in
noncommutative geometry and is reasonable to expect,
 whereas the Morita invariant property
of~algebra
structures in~\cite{DTT} and \cite{H} remains an~open problem.
We~would like to mention a~recent work \cite{AK}
in which the authors prove the Morita invariance of~the calculus structure at the level of~graded vector spaces of~homology.

\item The base ring can be any commutative ring spectrum.
Note that it is not possible to use
the dg operad $\mathcal{KS}_{{\rm comb}}$ in the generalization
to the ``spectral" setting.

\item
As stated Theorem~\ref{main3}, our functorial method allows us to generalize
to equivariant situ\-a\-tions. The $\infty$-categories with group actions
naturally appear in future applications (see~below).

\item
As revealed the outline in Section~\ref{strategysec},
our construction works well not with algebras but with $\infty$-categories.
Thus, even when one is ultimately interested in algebras,
it is important to~consider the $\infty$-category of~their modules.

\end{itemize}

We~would like to view our results from the perspective
of~noncommutative algebraic geometry.
As mentioned above, the notion of~$\KS$-algebra structures
is a~counterpart to the calculus on~manifolds.
Thus, $\KS$-algebras are central objects in ``noncommutative calculus".
We~refer the reader to~\cite{DTTex} and references therein
for this point of~view.

\looseness=1
Recall the algebro-geometric interpretations of~the Hoch\-schild cohomology
$\HHC(\CCC)$ and Hoch\-schild homology $\HHH(\CCC)$ for stable $\infty$-categories $\CCC$
or dg categories (somewhat more precisely, we assume that they are ``linear'' over a~field of~characteristic zero).
The $\etwo$-algebra $\HHC(\CCC)$ governs the deformations theory of~the stable $\infty$-category $\CCC$ in the derived geometric formulation.
The Hoch\-schild homology $\HHH(\CCC)$
(more precisely, the Hoch\-schild chain complex)
inhe\-rits an~$S^1$-action that corresponds to the Connes operator.
Then $\HHH(\CCC)$ with $S^1$-action gives rise to an~analogue of~the Hodge
filtration: the pair of~the negative cyclic homology and the peri\-o\-dic
cyclic homology can be thought of~as such a~structure.
(These algebraic structures are contained in the $\KS$-algebra $(\HHC(\CCC),\HHH(\CCC))$.)
As revealed in~\cite{I} in the case of~associative (dg) algebras $A$,
the action of~$\HHC(A)$ on $\HHH(A)$ encoded by~the $\KS$-algebra
structure at~the operadic level
is a~key algebraic datum that describes
variations of~the (analogue of) Hodge filtration along
noncommutative (curved) deformations. Namely,
the period map for noncommutative deformations (of~an associative
algebra) is controlled by~the $\KS$-algebra of~the Hoch\-schild cohomology and Hoch\-schild homology.
Therefore, the $\KS$-algebra $(\HHC(\CCC),\HHH(\CCC))$ will provide~a crucial algebraic input for the theory of~period maps for deformations of~the
stable $\infty$-category~$\CCC$.
The~significance of~the generalization to the equivariant context
(Theorem~\ref{main3})
will be seen when one comes to consider fruitful examples.
The motivations
partly come from mirror symmetry.
For example, stable $\infty$-categories endowed with $S^1$-actions
or some algebraic actions, that
are interesting from the viewpoint of~$S^1$-equivariant deformation theory,
naturally appear from Landau--Ginzburg models in the context
of~matrix factorizations. Its equivariant deformations together with the
associated Hodge structure should provide
a categorification of~the theory of~Landau--Ginzburg models.
As a~second example, if $X$ is a~sufficiently nice algebraic stack (more generally, a~derived stack),
one can consider the derived free loop space
$LX=\Map\big(S^1,X\big)$ of~$X$ (see, e.g., \cite{BFN}).
The stable $\infty$-category $\Perf(LX)$ of~perfect complexes on $LX$ comes equipped
with the natural $S^1$-action.
Finally, we would also like to mention that main results in this paper form the basis for
our recent work~\cite{IwF}.

\section{Strategy and organization}
\label{strategysec}

The purpose of~this section is to outline the strategy of~a construction of~a~$\KS$-algebra structure on the pair of~Hoch\-schild cohomology
and Hoch\-schild homology and to give the brief organization of~this paper.
This section is something like the second part of~introduction.
We~hope that the following outline will be helpful in understanding
the content of~the sequel. However, this section is independent with the
rest of~this paper so that the reader can skip it.

{\bf 2.1.} We~will give an~outline of~the construction.
Let $\CCC$ be a~small stable $\infty$-category.
If~$\CCC$ is not idempotent-complete,
we replace $\CCC$ by~its idempotent-completion: we assume that~$\CCC$ is idempotent-complete.
While we work with stable $\infty$-categories over a~commutative
ring spectrum $R$ in the paper, for simplicity
 we here work with plain stable idempotent-complete
$\infty$-categories (equivalently, we assume that $R$ is the sphere
spectrum). We~let $\DDD=\Ind(\CCC)$ denote the Ind-category
that is a~compactly generated stable $\infty$-category.
The $\infty$-category $\DDD$ is~also equivalent to
the functor category $\Fun^{\textup{ex}}\big(\CCC^{\rm op},\SP\big)$
of~exact functors, where $\SP$ is the stable \mbox{$\infty$-category} of~spectra.

Let $\Fun^{\textup{L}}(\DDD,\DDD)$ be the functor category from $\DDD$
to itself that
consists of~those functors which preserve small colimits.
This functor category is a~compactly generated stable $\infty$-category
and
inherits an~associative monoidal structure
given by~the composition of~functors.
We~denote by~$\mathcal{E}\textup{nd}(\DDD)^\otimes$ the
presentable stable $\infty$-category $\Fun^{\textup{L}}(\DDD,\DDD)$
endowed with the (associative) monoidal structure.

Let $\Alg_{\assoc}(\SP)$ denote the $\infty$-category of~associative ring spectra.
Given $A\in \Alg_{\assoc}(\SP)$, we define $\RMod_A$ to be
the $\infty$-category of~right $A$-module spectra.
Let us regard $A$ as a~right $A$-module in an~obvious way.
Then there is an~essentially unique colimit-preserving functor $p_A\colon\SP \to \RMod_A$ which sends the sphere spectrum $\mathbb{S}$ in $\SP$ to $A$.
Let $\PRST$ be the $\infty$-category of~presentable stable $\infty$-categories
in which morphisms are colimit-preserving functors.
The~assi\-g\-n\-ment $A\mapsto \{p_A\colon\SP\to \RMod_A\}$
induces $I\colon\Alg_{\assoc}(\SP)\to \big(\PRST\big)_{\SP/}$.
The right adjoint $E\colon\big(\PRST\big)_{\SP/}\to \Alg_{\assoc}(\SP)$
of~$I$ carries $p\colon\SP\to \mathcal{P}$ in $\big(\PRST\big)_{\SP/}$
to the endomorphism (asso\-ci\-a\-tive) algebra $\End_{\mathcal{P}}(P)$,
where $P$ is the image $p(\mathbb{S})$.
Note that $\Alg_{\assoc}(\SP)$ has a~natural symmetric monoidal
structure whose tensor product is induced by~the
tensor product $A\otimes B$ in $\SP$.
The $\infty$-category $\PRST$ also admits an~appropriate symmetric monoidal
structure in which $\SP$ is a~unit object,
and $I\colon\Alg_{\assoc}(\SP)\to \big(\PRST\big)_{\SP/}$
can be promoted to a~symmetric monoidal functor.
Applying $\Alg_{\assoc}(-)$ to
$I\colon\Alg_{\assoc}(\SP)\leftrightarrows \big(\PRST\big)_{\SP/}\,{\colon}\!E$, we obtain
\begin{gather*}
I\colon\ \Alg_{\etwo}(\SP)\simeq \Alg_{\assoc}(\Alg_{\assoc}(\SP)) \rightleftarrows \Alg_{\assoc}\big(\big(\PRST\big)_{\SP/}\big) \simeq \Alg_{\assoc}\big(\PRST\big)\ {\colon}E,
\end{gather*}
where $\Alg_{\etwo}(\SP)$
is the $\infty$-category of~$\etwo$-algebras, and
$\Alg_{\etwo}(\SP)\simeq \Alg_{\assoc}(\Alg_{\assoc}(\SP))$
and the~left equivalence
follows from Dunn additivity theorem. Here we abuse notation by~writing $I$ and $E$ for the induced functors.
Let us regard $\Alg_{\assoc}\big(\PRST\big)$ as the $\infty$-category
of~monoidal presentable stable $\infty$-categories.
The left adjoint sends an~$\etwo$-algebra $A$ to
the associative monoidal $\infty$-category $\RMod_A^\otimes$.
The right adjoint carries a~monoidal presentable stable \mbox{$\infty$-category}
$\mathcal{M}^\otimes$ to the endomorphism spectrum $\End_{\mathcal{M}}(\mathsf{1}_{\mathcal{M}})$ of~the unit object $\mathsf{1}_{\mathcal{M}}$, endowed with an~$\etwo$-algebra structure.

We~define Hoch\-schild cohomology spectrum $\HHC(\CCC)\!=\HHC(\DDD)$
as $E\big(\mathcal{E}\textup{nd}(\DDD)^\otimes\big)\!\in \Alg_{\etwo}(\SP)$.
The underlying associative algebra
$\HHC(\CCC)$ is the endomorphism algebra of~the identity functor
$\DDD\to \DDD$ in $\Fun^{\textup{L}}(\DDD,\DDD)$.

Consider the counit map of~the adjunction:
\begin{gather*}
\RMod_{\HHC(\DDD)}^\otimes \longrightarrow \mathcal{E}\textup{nd}(\DDD)^\otimes,
\end{gather*}
which is a~monoidal functor.
Since $\mathcal{E}\textup{nd}(\DDD)^\otimes$ naturally acts on $\DDD$,
it gives rise to an~action of~$\RMod_{\HHC(\DDD)}^\otimes$ on
$\DDD$:
\begin{gather*}
\mathcal{E}\textup{nd}(\DDD)^\otimes \curvearrowright \DDD \ \ \Rightarrow \ \ \RMod_{\HHC(\DDD)}^\otimes \curvearrowright \DDD.
\end{gather*}
In~other words, $\DDD$ is a~left $\RMod_{\HHC(\DDD)}^\otimes$-module
object in $\PRST$.
Let $\RPerf_{\HHC(\CCC)}^\otimes\subset \RMod_{\HHC(\DDD)}^\otimes$ be the monoidal full subcategory that consists of~compact objects.
By the restrictions,
it gives rise to a~left $\RPerf_{\HHC(\CCC)}^\otimes$-module object $\CCC$:
\begin{gather*}
\RPerf_{\HHC(\CCC)}^\otimes \curvearrowright \CCC,
\end{gather*}
in the $\infty$-category $\ST$ of~small stable idempotent-complete $\infty$-categories in which morphisms are exact functors ($\ST$ also admits
a suitable symmetric monoidal structure).
Informally, we think of~it as a~categorical associative action of~$\HHC(\CCC)$ on $\CCC$.
This is induced by~the adjunction so that it has an~evident universal property.

Construct a~functor $\ST\to \SP$ which carries
$\CCC$ to the Hoch\-schild homology spectrum $\HHH(\CCC)$.
In~the classical differential graded context,
Hoch\-schild chain complex comes equipped with the Connes operator.
In~our general setting, it is natural to encode such structures
by~means of~circle actions: Hoch\-schild homology spectrum $\HHH(\CCC)$
is promoted to a~spectrum with an~\mbox{$S^1$-action}, that is, an~object of~$\Fun\big(BS^1,\SP\big)$. Thus we configure the assignment
$\CCC\mapsto \HHH(\CCC)$ as a~symmetric monoidal functor
\begin{gather*}
\HHH(-)\colon\ \ST^\otimes \longrightarrow \Fun\big(BS^1,\SP\big)^\otimes,
\end{gather*}
where $\Fun\big(BS^1,\SP\big)$ inherits a~pointwise symmetric monoidal
structure from the structure on~$\SP$.

Applying the symmetric monoidal functor $\HHH(-)$ to
the left $\RPerf_{\HHC(\CCC)}^\otimes$-module object $\CCC$,
we obtain a~left $\HHH\big(\!\RPerf_{\HHC(\CCC)}\!\big)$-module $\HHH(\CCC)$ in $\Fun\big(\!BS^1\!,\SP\!\big)$. Note that $ \HHH(\HHC(\CCC))$ $\simeq \HHH\big(\RPerf_{\HHC(\CCC)}\big)$
is an~associative algebra object in $\Fun\big(BS^1,\SP\big)$. In~other words,
it is an~associative ring spectrum equipped with an~$S^1$-action.

There is a~topological operad ($\infty$-operad) $\cyl$ (defined in a~geometric way). We~have a~cano\-ni\-cal equivalence $\Alg_{\cyl}(\SP)\simeq \Alg_{\assoc}\big(\Fun\big(BS^1,\SP\big)\big)$ between $\infty$-categories of~algebras,
where $\Alg_{\cyl}(\SP)$ is the $\infty$-category of~algebras over $\cyl$.
There is another two-colored topological operad ($\infty$-operad)
$\dcyl$ such that operads $\etwo^\otimes$ and $\cyl$ are contained
in $\dcyl$ as full suboperads: $\etwo^\otimes \subset \dcyl \supset \cyl$.
Let $i\colon\etwo^\otimes\hookrightarrow \dcyl$
denote the inclusion, and let $i_!\colon\Alg_{\etwo}(\SP)\to \Alg_{\dcyl}(\SP)$
be the left adjoint of~the forgetful functor $i^*\colon\Alg_{\dcyl}(\SP)\to \Alg_{\etwo}(\SP)$.
Consider the sequence
\begin{gather*}
\Alg_{\etwo}(\SP)\stackrel{i_!}{\longrightarrow} \Alg_{\dcyl}(\SP) \stackrel{\textup{forget}}{\longrightarrow} \Alg_{\cyl}(\SP)\simeq \Alg_{\assoc}\big(\Fun\big(BS^1,\SP\big)\big)
\end{gather*}
and denote by~$i_!(\HHC(\CCC))_C$ the image of~$\HHC(\CCC)$
under the composite.
We~construct a~canonical equivalence $i_!(\HHC(\CCC))_C\simeq \HHH(\HHC(\CCC))$
in $\Alg_{\assoc}\big(\Fun\big(BS^1,\SP\big)\big)$.

Assembling the constructions, we obtain a~triple
\begin{gather*}
(\HHC(\CCC), i_!(\HHC(\CCC))_C\simeq \HHH(\HHC(\CCC)), \HHH(\CCC)),
\end{gather*}
where $\HHC(\CCC)$ is the $\etwo^\otimes$-algebra, $\HHH(\CCC)$
is the left $\HHH(\HHC(\CCC))$-module (in $\Fun\big(BS^1,\SP\big)$).
As mentioned in Theorem~\ref{main2}, we prove that the triple exactly amounts to a~$\KS$-algebra
$(\HHC(\CCC),\HHH(\CCC))$, that is, the structure of~an algebra over
the Kontsevich--Soibelman operad $\KS$ on the pair $(\HHC(\CCC),\HHH(\CCC))$.
This completes the construction.

{\bf 2.2.} This paper is organized as follows:
Section~\ref{NC} collects conventions and some of~the notation that we will use.
In~Section~\ref{operadsec}, we discuss algebras over the Kontsevich--Soibelman
operad. The~main result of~Section~\ref{operadsec} is Corollary~\ref{algebraequivalence}
(=~Theorem~\ref{main2}).
Along the way,
we introduce several topological colored operads ($\infty$-operads).
In~Section~\ref{cohomologysec}, we give a~brief review of~Hoch\-schild cohomology
spectra that we will use.
In~Section~\ref{homologysec}, we give a~construction of~the assignment
$\CCC\mapsto \HHH(\CCC)$ which satisfies the requirements for our goal
(partly
because we are not able to find a~suitable construction in the literature).
The results of~this section will be quite useful for various purposes other than the subject of~this
paper.
In~Section~\ref{constructionsec}, we prove Theorem~\ref{mainconstruction} (=~Theorem~\ref{main1}).
Namely, we construct a~$\KS$-algebra $(\HHC(\CCC),\HHH(\CCC))$.
In~Section~\ref{actionsec}, we study the action morphisms
determined by~the structure of~the $\KS$-algebra on~$(\HHC(\CCC),\HHH(\CCC))$.
In~Section~\ref{equivariantsec}, we give a~generalization to
an equivariant setting (cf.~Theo\-rem~\ref{main3}): $\CCC$ is endowed with the action of~a~group (a group
object in the $\infty$-category of~spaces).

\section{Notation and convention}\label{NC}

Throughout this paper we use the theory of~{\it quasi-categories}.
We~assume that the reader is familiar with this theory
and operads.
A quasi-category is a~simplicial set which
satisfies the weak Kan condition of~Boardman--Vogt.
The theory of~quasi-categories from the viewpoint of~models of~$(\infty,1)$-categories
were extensively developed by~Joyal and Lurie \cite{Jo, HA, HTT}.
Following \cite{HTT}, we shall refer to quasi-categories
as {\it $\infty$-categories}.
Our main references are~\cite{HA} and~\cite{HTT}.
Given an~ordinary category $\CCC$, by~passing to the nerve $\NNNN(\CCC)$,
we think of~$\CCC$ as the $\infty$-category $\NNNN(\CCC)$.
We~usually abuse notation by~writing $\CCC$ for $\NNNN(\CCC)$
even when $\CCC$ should be thought of~as a~simplicial set
or an~$\infty$-category.

We~use the theory of~{\it $\infty$-operads} which is
thoroughly developed in~\cite{HA}.
The notion of~$\infty$-ope\-rads gives
one of~the models of~colored operads.
Thanks to Hinich \cite{Hin}, there is a~comparison between
algebras over differential graded operads
and algebras over $\infty$-operads in values in~chain complexes.
In~particular, in characteristic zero, \cite{Hin} establishes
an equivalence between two notions of~algebras, see {\it loc.\ cit}.

Here is a~list of~some of~the conventions and notation that we will use:
\begin{itemize}\itemsep=0pt

\item $\ZZ$: the ring of~integers, $\RRR$ denotes the set of~real numbers which we regard as either a~topo\-logical space or a~ring.

\item $\Delta$: the category of~linearly ordered non-empty
finite sets (consisting of~[0], [1], \dots, \mbox{$[n]=\{0,\dots,n\}$}, \dots).

\item $\Delta^n$: the standard $n$-simplex.

\item $\textup{N}$: the simplicial nerve functor (cf.~\cite[Section~1.1.5]{HTT}).

\item $\SSS$: $\infty$-category of~small spaces. We~denote by~$\widehat{\SSS}$
the $\infty$-category of~large spaces (cf.~\cite[Section~1.2.16]{HTT}).

\item $\CCC^\simeq$: the largest Kan subcomplex of~an $\infty$-category $\CCC$.

\item $\CCC^{\rm op}$: the opposite $\infty$-category of~an $\infty$-category. We~also use the superscript ``op" to~indicate the opposite category for ordinary categories and enriched categories.

\item $\operatorname{Cat}_\infty$: the $\infty$-category of~small $\infty$-categories.

\item $\SP$: the stable $\infty$-category of~spectra.

\item $\Fun(A,B)$: the function complex for simplicial sets $A$ and $B$.

\item $\Fun_C(A,B)$: the simplicial subset of~$\Fun(A,B)$ classifying
maps which are compatible with
given projections $A\to C$ and $B\to C$.

\item $\Map(A,B)$: the largest Kan subcomplex of~$\Fun(A,B)$ when $A$ and $B$ are $\infty$-categories.

\item $\Map_{\mathcal{C}}(C,C')$: the mapping space from an~object $C\in\mathcal{C}$ to $C'\in \mathcal{C}$, where $\mathcal{C}$ is an~$\infty$-category.
We~usually view it as an~object in $\mathcal{S}$ (cf.~\cite[Section~1.2.2]{HTT}).

\item $\FIN$: the category of~pointed finite
sets $\langle 0 \rangle, \langle 1 \rangle,\dots \langle n \rangle,\dots$,
where $\langle n, \rangle=\{*,1,\dots, n \}$ with the base point $*$.
We~write $\Gamma$ for $\NNNN(\FIN)$. $\langle n\rangle^\circ=\langle n\rangle\backslash*$.
Notice that the (nerve of) Segal's gamma category is the opposite category
of~our $\Gamma$.

\item $\mathcal{P}^{\textup{act}}$: If~$\mathcal{P}$ is an~$\infty$-operad,
we write $\mathcal{P}^{\textup{act}}$ for the subcategory
of~$\mathcal{P}$ spanned by~active morphisms.

\item $\triv^\otimes$: the trivial $\infty$-operad \cite[Example~2.1.1.20]{HA}.

\item $\assoc^\otimes$: the associative operad \cite[Section~4.1.1]{HA}, we use
the notation slightly different from~{\it loc.~cit.}
Informally,
an $\assoc$-algebra (an algebra over $\assoc^\otimes$)
is an~unital associative algebra.
For a~symmetric monoidal $\infty$-category $\CCC^\otimes$,
we write $\Alg_{\assoc}(\CCC)$ for the $\infty$-category of~$\assoc$-algebra objects.
We~refer to an~object of~$\Alg_{\assoc}(\CCC)$
as an~associative algebra object in~$\CCC^\otimes$.
We~refer to a~monoidal $\infty$-category over $\assoc^\otimes$ as an~associative monoidal $\infty$-category.

\item $\LM^\otimes$: the $\infty$-operad defined in~\cite[Definition~4.2.1.7]{HA}.
An algebra over $\LM^\otimes$ is a~pair $(A,M)$
such that an~unital associative algebra $A$
and a~left $A$-module $M$.
For a~symmetric monoidal $\infty$-category $\CCC^\otimes \to \Gamma$,
we write $\LMod\big(\CCC^\otimes\big)$ or $\LMod(\CCC)$
for $\Alg_{\LM^\otimes}\big(\CCC^\otimes\big)$.

\item $\mathbf{E}^\otimes_n$: the $\infty$-operad of~little $n$-cubes.
For a~symmetric monoidal $\infty$-category $\CCC^\otimes$,
we write $\Alg_{\mathbf{E}_n}(\CCC)$ for
the $\infty$-category of~$\mathbf{E}_n$-algebra objects.

\end{itemize}

\section{Operads}\label{operadsec}

{\bf 4.1.} We~will define several simplicial colored operads which are relevant to us.
By a~simplicial colored operad, we mean a~colored operad in the
symmetric monoidal category of~simplicial sets.
A simplicial colored operad is also referred to as
a symmetric multicategory enriched over the
category of~simplicial sets.

\begin{Definition}
\label{recti}
Let $(0,1)$ denote the open interval $\{x\in \RRR\,|\, 0<x<1\}$.
For $n \ge0$,
let $(0,1)^{n}$ be the $n$-fold product, i.e., the $n$-dimensional cube.
An open embedding $f\colon(0,1)^{n}\to (0,1)^{n}$ is~said to be
rectilinear if it is given by
\begin{gather*}
f(x_1,\dots,x_n)=(a_1x_1+b_1,\dots,a_nx_n+b_n)
\end{gather*}
for some real constants $0<a_1,\dots,a_n\le 1,\ 0\le b_1,\dots,b_n<1$, provided that
the formula defines an~embedding.
An embedding $f\colon(0,1)^{n}\to (0,1)^{n}$ is said to be
shrinking if it is given by~$f(x_1,\dots,x_n)=(a_1x_1,\dots,a_nx_n)$
for some $0<a_1,\dots,a_n\le1$.

Let $S^1$ denote the circle $\RRR/\ZZ$ which we regard as a~topological space.
A continuous map $f\colon(0,1)^n\times S^1 \to (0,1)^n\times S^1$
is said to be rectilinear (resp.~shrinking) if $f=(\phi,\psi)$
such that $\phi\colon(0,1)^n\to (0,1)^n$ is rectilinear (resp.~shrinking)
and $\psi\colon S^1\to S^1$ is given by~a rotation
\begin{gather*}
S^1=\RRR/\ZZ\ni x \mapsto x+r \in \RRR/\ZZ=S^1
\end{gather*}
with $r\in\RRR/\ZZ$. In~particular, when $n=0$,
$f\colon S^1 \to S^1$ is rectilinear if it is given by~a rotation.

Let $n\ge1$. If~a~continuous map $f\colon(0,1)^n\to (0,1)^{n-1}\times S^1$ is
said to be rectilinear if it factors as $(0,1)^n\stackrel{g}{\to} (0,1)^{n-1}\times \RRR \stackrel{h}{\to} (0,1)^{n-1}\times \RRR/\ZZ$,
where $h$ is the projection, and
$g$ is given by~$f(x_1,\dots,x_n)=(a_1x_1+b_1,\dots,a_nx_n+b_n)$
for some real constants $0<a_1,\dots,a_n\le1$, $0\le b_1,\dots,b_{n-1}<1$, $b_n\in \RRR$, provided
that
the formula defines an~open embedding.

\end{Definition}

\begin{Definition}
Let $\CYL$ be a~simplicial colored operad defined as follows:
\begin{enumerate}\itemsep=0pt
\renewcommand{\labelenumi}{(\roman{enumi})}

\item The set of~colors of~$\CYL$ has a~single element,
which we will denote by~$C$.

\item Let $I=\langle r \rangle^\circ$ be a~finite set and let
$\{C\}_{I}$ be a~set of~colors indexed by~$I$.
By abuse of~notation, we write $C^{\sqcup r}$
for $\{C\}_{I}$, where $r$ is the number of~elements of~$I$.
We~remark that $C^{\sqcup r}$ does not mean the coproduct.
We~define $\mul_{\CYL}(\{C\}_I,C)=\mul_{\CYL}\big(C^{\sqcup r},C\big)$ to be the singular simplicial complex of~the space
\begin{gather*}
\textup{Emb}^{\rm rec}\big(\big((0,1)\times S^1\big)^{\sqcup r}, (0,1)\times S^1\big)
\end{gather*}
of~embeddings $\big((0,1)\times S^1\big)^{\sqcup r}\to (0,1)\times S^1$
such that the restriction to each component
$(0,1)\times S^1 \to (0,1)\times S^1$ is rectilinear.
Here $\big((0,1)\times S^1\big)^{\sqcup r}$
is the disjoint union of~$(0,1)\times S^1$, whose set of~connected
components is identified with
$I$. The space $\textup{Emb}^{\rm rec}\big(\big((0,1)\times S^1\big)^{\sqcup r}, (0,1)\times S^1\big)$ is endowed with the standard topology, that is,
the subspace of~the mapping space
with compact-open topology.

\item The composition law in $\CYL$ is given by~the composition of~rectilinear
embeddings, and a~unit map is the identity map.

\end{enumerate}
The color $C$ together with $\mul_{\CYL}(\{C\}_I,C)$
constitutes a~fibrant simplicial colored operad.
By~a~fibrant simplicial colored operad we mean that
every simplicial set $\mul_{\CYL}(\{C\}_I,C)$ is a~Kan complex.
Note that the singular simplicial complex of~a~topological space is a~Kan complex.
\end{Definition}

\begin{Definition}
\label{cylmdef}
Let $\CYLM$ be a~simplicial colored operad defined as follows:
\begin{enumerate}\itemsep=0pt
\renewcommand{\labelenumi}{(\roman{enumi})}

\item The set of~colors of~$\CYLM$ has two elements denoted by~$C$ and $\cbar$.

\item Let $I=\langle r\rangle^\circ $ be a~finite set and let
$\{C,\cbar\}_{I}$ be a~set of~colors indexed by~$I$,
which we think of~as a~map $p\colon I\to \{C,\cbar\}$.
We~also write $C^{\sqcup m}\sqcup \cbar^{\sqcup n}$ for
$\{C,\cbar\}_{I}$ when $p^{-1}(C)$ (resp.~$p^{-1}(\cbar)$)
has $m$ elements
(resp.~$n$ elements).
Let
\begin{gather*}
\textup{Emb}^{\rm rec}\big(\big((0,1)\times S^1\big)^{\sqcup m}\sqcup \big((0,1)\times S^1\big)^{\sqcup n}, (0,1)\times S^1\big)
\end{gather*}
denote the space of~embeddings $\big((0,1)\times S^1\big)^{\sqcup m}\sqcup \big((0,1)\times S^1\big)^{\sqcup n}\to (0,1)\times S^1$
such that the restriction to each component
is rectilinear (the topology is induced by~compact-open topology).
We~refer to it as the space of~rectilinear embeddings.
For $n\ge1$, let
\begin{gather*}
\mul_{\CYLM}^{t}\big(C^{\sqcup{m}}\sqcup \cbar^{\sqcup n},\cbar\big)
\\ \qquad{} :=
{\textup{Emb}^{\rm rec}\big(\big((0,1)\times S^1\big)\times p^{-1}(C)\sqcup \big((0,1)\times S^1\big)\times p^{-1}(C_M), (0,1)\times S^1\big)}
\end{gather*}
be its subspace that consists of~those rectilinear embeddings $f$
such that each restriction to any component in
$\big((0,1)\times S^1\big)\times p^{-1}(C_M)$ is (not only rectilinear but also)
shrinking.
Here $\big((0,1)\times S^1\big)\times p^{-1}(C)\sqcup \big((0,1)\times S^1\big)\times p^{-1}(C_M)$ denotes the finite disjoint union
of~$(0,1)\times S^1$ indexed by~$p^{-1}(C)\sqcup p^{-1}(C_M)\simeq I$,
but we distinguish between
components indexed by~$p^{-1}(C)$ and those indexed in $p^{-1}(C_M)$
since they play different roles. Notice that $\mul_{\CYLM}^{t}\big(C^{\sqcup{m}}\sqcup \cbar^{\sqcup n},\cbar\big)$
is the empty space for $n\ge2$.
We~define $\mul_{\CYLM}\big(C^{\sqcup{m}}\sqcup \cbar^{\sqcup n},\cbar\big)$ to be the singular
simplicial complex of~$\mul_{\CYLM}^{t}\big(C^{\sqcup{m}}\sqcup \cbar^{\sqcup n},\cbar\big)$.
When $n=0$, we define $\mul_{\CYLM}\big(C^{\sqcup{m}},\cbar\big)$ to be the empty simplicial set.

\item We~set $\mul_{\CYLM}(C^{\sqcup{m}},C)=\mul_{\CYL}\big(C^{\sqcup{m}},C\big)$.
If~$n\neq 0$, $\mul_{\CYLM}\big(C^{\sqcup{m}}\sqcup \cbar^{\sqcup n},C\big)$ is the empty set.

\item The composition law is given by~the composition of~rectilinear embeddings, and a~unit map is the identity map.
\end{enumerate}
The colors $C$, $\cbar$ together with simplicial sets of~maps constitute a~fibrant simplicial colored operad.
\end{Definition}

\begin{Definition}\label{dcyldef}
Let $\DCYL$ be a~simplicial colored operad defined as follows:
\begin{enumerate}\itemsep=0pt
\renewcommand{\labelenumi}{(\roman{enumi})}

\item The set of~colors of~$\DCYL$ has two elements, which we denoted by~$D$
and $C$.

\item Let $I=\langle r \rangle^\circ$ be a~finite set and let
$\{D,C\}_{I}$ be a~set of~colors indexed by~$I$,
that is, a~map $p\colon I\to \{D,C\}$.
By abuse of~notation we write $D^{\sqcup{l}}\sqcup C^{\sqcup m}$ for
$\{D,C\}_{I}$ when $p^{-1}(D)$ (resp.~$p^{-1}(C)$)
has $l$ elements
(resp.~$m$ elements).
We~define $\mul_{\DCYL}\big(D^{\sqcup{l}},D\big)$ to be the singular simplicial complex of~the space
\begin{gather*}
\textup{Emb}^{\rm rec}\big(\big((0,1)^2\big)^{\sqcup l},(0,1)^2\big)
\end{gather*}
of~embeddings from the disjoint union $(0,1)^2 \times p^{-1}(D)$ to $(0,1)^2$
such that the restriction to~each component
is rectilinear, where the space comes equipped with
the subspace topology of~the mapping space
with compact-open topology.

If~$m\ge 1$,
$\mul_{\CYLM}\big(D^{\sqcup{l}}\sqcup C^{\sqcup m},D\big)$ is the empty
set.

\item We~define $\mul_{\DCYL}\big(D^{\sqcup{l}}\sqcup C^{\sqcup m},C\big)$
to be the singular complex of~the space of~embeddings
\begin{gather*}
\textup{Emb}^{\rm rec}\big((0,1)^2\times p^{-1}(D) \sqcup \big((0,1)\times S^1\big)\times p^{-1}(C),(0,1)\times S^1\big)
\end{gather*}
such that the rescriction to a~component $(0,1)^2$ is rectilinear,
and the restriction to a~component $(0,1)\times S^1$ is rectilinear.
\item The composition law and the unit are defined in an~obvious way.
\end{enumerate}
The colors $D$, $C$ together with simplicial sets of~maps
constitute a~fibrant simplicial colored operad.
\end{Definition}

\begin{Definition}Let $\DCYLM$ be a~simplicial colored operad defined as follows.
\begin{enumerate}\itemsep=0pt
\renewcommand{\labelenumi}{(\roman{enumi})}

\item The set of~colors of~$\DCYLM$ has three elements,
which we denote by~$D$, $C$, and $\cbar$.

\item Let $I=\langle r \rangle^\circ$ be a~finite set and let
$\{D,C,\cbar\}_{I}$ be a~set of~colors indexed by~$I$,
that is, a~map $p\colon I\to \{D,C,\cbar\}$.
By abuse of~notation, we write $D^{\sqcup{l}}\sqcup C^{\sqcup m}\sqcup \cbar^{\sqcup n}$ for
$\{D,C,\cbar\}_{I}$ when $p^{-1}(D)$ (resp.~$p^{-1}(C)$, $p^{-1}(\cbar)$)
has $l$ elements
(resp.~$m$ elements, $n$ elements).
We~set $\mul_{\DCYLM}\big(D^{\sqcup{l}},D\big)=\mul_{\DCYL}\big(D^{\sqcup{l}},D\big)$.
If~$m+n\ge 1$,
$\mul_{\CYLM}\big(D^{\sqcup{l}}\sqcup C^{\sqcup m}\sqcup \cbar^{\sqcup n},D\big)$ is the empty
set.

\item We~set $\mul_{\DCYLM}\big(D^{\sqcup{l}}\sqcup C^{\sqcup m},C\big)=\mul_{\DCYL}\big(D^{\sqcup{l}}\sqcup C^{\sqcup m},C\big)$. If~$n\ge1$,
\begin{gather*}
\mul_{\DCYLM}\big(D^{\sqcup{l}}\sqcup C^{\sqcup m}\sqcup \cbar^{\sqcup n},C\big)
\end{gather*}
is the empty set.

\item For $n\neq 1$,
$\mul_{\DCYLM}\big(D^{\sqcup{l}}\sqcup C^{\sqcup m}\sqcup \cbar^{\sqcup n},\cbar\big)$
is the empty set.
For $n=1$,
\begin{gather*}
\mul_{\DCYLM}\big(D^{\sqcup{l}}\sqcup C^{\sqcup m}\sqcup \cbar,\cbar\big)
\end{gather*}
is the singular simplicial complex of~the space of~embeddings
\begin{gather*}
(0,1)^2\times p^{-1}(D) \sqcup \big((0,1)\times S^1\big)\times p^{-1}(C) \sqcup \big((0,1)\times S^1\big)\times p^{-1}(C_M) \to (0,1)\times S^1
\end{gather*}
such that the restriction to a~component in $(0,1)^2\times p^{-1}(D)$ is rectilinear,
the restriction to~each component in $\big((0,1)\times S^1\big)\times p^{-1}(C)$
is rectilinear, and the restriction to
$\big((0,1)\times S^1\big)\times p^{-1}(C_M) \simeq(0,1)\times S^1$
is shrinking.
By definition, if $l=0$, $\mul_{\DCYLM}(C^{\sqcup m}\sqcup \cbar,\cbar)=\mul_{\CYLM}(C^{\sqcup m}\sqcup \cbar,\cbar)$.

\item The composition law and the unit map are defined
in an~obvious way.
\end{enumerate}
The colors $D$, $C$, $\cbar$ together with simplicial sets of~maps constitute a~fibrant simplicial colored operad.
\end{Definition}

\begin{Remark}\label{coloredinclusion}
There is a~commutative diagram of~inclusions of~simplicial colored operads:
\begin{gather*}
\xymatrix{
\CYL \ar[r] \ar[d] & \CYLM \ar[d] \\
\DCYL \ar[r] & \DCYLM.
}
\end{gather*}
Each inclusion determines a~simplicial colored full suboperad.
\end{Remark}

We~obtain an~$\infty$-operad from a~fibrant simplicial
colored operad.
We~recall the construction from \cite[Notation~2.1.1.22]{HA}.

\begin{Definition}
\label{associatedinfoperad}
Let $P$ be a~simplicial colored operad. Let $P_{\rm col}$ be the set of~colors
of~$P$.
We~let~$P_{\Delta}$ be a~simplicial category defined as follows:
\begin{enumerate}\itemsep=0pt
\renewcommand{\labelenumi}{(\roman{enumi})}

\item The objects of~$P_{\Delta}$ are maps
$a\colon\langle n\rangle^\circ\to P_{\rm col}$, that is,
pairs $(\langle n\rangle,(C_1,\dots,C_n))$,
where \mbox{$\langle n\rangle \in \FIN$} and
$(C_1,\dots,C_n)$
is a~finite sequence $(a(1),\dots,a(n))$ of~colors.

\item Let $C=(\langle n\rangle,(C_1,\dots,C_n))$ and $C'=(\langle m\rangle,(C_1',\dots,C_m'))$ be two objects.
The hom simplicial set $\Map_{P_{\Delta}}(C,C')$
is given by~\begin{gather*}
\coprod_{\alpha\colon\langle n\rangle\to \langle m\rangle}\prod_{j\in \langle m\rangle^{\circ}}\mul_{P}\left(\{C_i\}_{i\in \alpha^{-1}(j)},C'_j\right)\!.
\end{gather*}

\item Composition is determined by~the composition laws on $\FIN$
and on $P$ in an~obvious way.

\end{enumerate}
There is a~canonical simplicial functor $P_{\Delta}\to \FIN$
which sends $(\langle n\rangle,(C_1,\dots,C_n))$
to $\langle n\rangle$.
If~$P$ is fibrant,
the map of~simplicial nerves $\mathbf{P}:=\NNNN(P_{\Delta})\to \NNNN(\FIN)=\Gamma$
constitutes an~$\infty$-operad (cf.~\cite[Proposition~2.1.1.27]{HA}).
We~shall refer to $\NNNN(P_{\Delta})\to \NNNN(\FIN)=\Gamma$
(or $\NNNN(P_{\Delta})$) as the operadic nerve of~$P$.
We~shall denote by~$\mathsf{P}_{\langle n\rangle}$ the fiber
$\mathbf{P}\times_\Gamma\{\langle n \rangle\}$ over $\langle n \rangle$.
We~usually identify colors with objects in $\mathbf{P}_{\langle1\rangle}$.
\end{Definition}

\begin{Definition}
We~apply the construction in Definiton~\ref{associatedinfoperad}
to $\CYL$, $\CYLM$, $\DCYL$, and $\DCYLM$ to~obtain $\infty$-operads.
\begin{itemize}\itemsep=0pt

\item Let $\cyl$ be the operadic nerve of~$\CYL$.

\item Let $\cylm$ be the operadic nerve of~$\CYLM$.

\item Let $\dcyl$ be the operadic nerve of~$\DCYL$.

\item Let $\dcylm$ be the operadic nerve of~$\DCYLM$.

\end{itemize}
\end{Definition}

We~now recall Kontsevich--Soibelman operad \cite{KS}.

\begin{Definition}
Let $\textup{KS}$ be the simplicial colored full suboperad of~$\DCYLM$ which
consists of~colors~$D$, $\cbar$.
We~refer to $\textup{KS}$ as Kontsevich--Soibelman operad.
Let $\KS$ be the operadic nerve of~$\textup{KS}$ (the notation is slightly
diffrent
from Introduction).
We~abuse terminology by~referring to it as Kontsevich--Soibelman operad.
In~a~nutshell, $\KS\subset \dcylm$ is the maximal
simplicial subcomplex
spanned by~vertices correponding to those tuples which do not contain the color~$C$.
It~is not difficult to check that $\textup{KS}$ is equivalent to that of~\cite{H} or \cite[Section~11.2]{KS}.
In~\cite{DTT}, a~version of~$\textup{KS}$ is called the cylinder operad.
\end{Definition}

\begin{Remark}
We~note that $\DCYL$ has a~simplicial full suboperad $\langle D \rangle \subset \DCYL$
which consists of~the single color $D$.
This operad $\langle D \rangle$ is a~version of~the little $2$-cube operad (e.g., \cite[Definition~5.1.0.1]{HA}).
Let $\etwo^\otimes$ be the operadic nerve of~$\langle D \rangle$,
 which we shall refer to as the $\infty$-operad of~little $2$-cubes.
\end{Remark}

\begin{Remark}
We~have the diagram in Remark~\ref{coloredinclusion}
and inlusions $\textup{KS} \subset \DCYLM$, $\langle D \rangle \subset \DCYL$. These inclusions determine
the following diagram of~$\infty$-operads:
\begin{gather*}
\xymatrix{
\etwo^\otimes \ar[dr] & & \cyl \ar[ld] \ar[rd] & \\
 & \dcyl \ar[rd] & & \cylm \ar[ld] \\
 & \KS \ar[r] & \dcylm.&
}
\end{gather*}
\end{Remark}

Let $\textup{E}_1$ be the simplicial operad of~little $1$-cubes.
The definition is similar to the case of~little $2$-cubes (see, e.g.,~\cite[Proposition~5.1.0.3]{HA}).
Namely, $\textup{E}_1$ has a~single color $D_1$, and for a~finite sequence $(D_1,\dots,D_1)$,
the simplicial set $\mul_{\textup{E}_1}\big(D_1^{\sqcup n},D_1\big)$ is defined to be
the singular simplicial complex of~the space
$\textup{Emb}^{\rm rec}\big((0,1)^{\sqcup n},(0,1)\big)$ of~rectilinear embeddings.
The composition law and the unit
are defined in the obvious way.
Let $\eone^\otimes$ denote the $\infty$-operad of~little $1$-cubes,
that is, the operadic nerve of~$\textup{E}_1$.

\begin{Definition}\label{operadeonemod}
Let $\overline{\textup{E}}_1$ be a~simplicial colored
operad defined as follows.
\begin{enumerate}\itemsep=0pt
\renewcommand{\labelenumi}{(\roman{enumi})}

\item The set of~colors of~$\overline{\textup{E}}_1$ has two elements
which we denote by~$D_1$ and $M$.

\item Let $I=\langle r\rangle^\circ$ be a~finite set and let
$\{D_1,M\}_{I}$ be a~set of~colors indexed by~$I$,
which is a~map $p\colon I\to \{D_1,M\}$.
We~write $D_1^{\sqcup{m}}\sqcup M^{\sqcup n}$ for
$\{D_1,M\}_{J}$ when $p^{-1}(D_1)$ (resp.~$p^{-1}(M)$)
has $m$ elements
(resp.~$n$ elements).
Let $\textup{Emb}^{\rm rec}\big((0,1)^{\sqcup m}\sqcup (0,1), (0,1)\big)$
be the topological space of~embeddings $(0,1)^{\sqcup m} \sqcup (0,1) \to (0,1)$
such that the restriction to each component
is rectilinear.
We~define $\mul_{\overline{\textup{E}}_1}\big(D_1^{\sqcup{m}}\sqcup M,M\big)$ to be the singular simplicial complex of~the subspace
\begin{gather*}
\textup{Emb}^{\rm rec}\big((0,1)\times p^{-1}(D_1) \sqcup (0,1)\times p^{-1}(M), (0,1)\big)\subset \textup{Emb}^{\rm rec}\big((0,1)^{\sqcup m}\sqcup (0,1), (0,1)\big).\!
\end{gather*}
The subspace consists of~those rectilinear embeddings
such that the restriction to $(0,1)\times p^{-1}(M) \simeq (0,1)$
is shrinking.
If~$n\neq 1$,
$\mul_{\overline{\textup{E}}_1}\big(D_1^{\sqcup{m}}\sqcup M^{\sqcup n},M\big)$ is the empty
set.

\item We~set $\mul_{\overline{\textup{E}}_1}\big(D_1^{\sqcup{n}},D_1\big)=\mul_{\textup{E}_1}\big(D_1^{\sqcup{n}},D_1\big)$.
If~$m\neq 0$,
$\mul_{\overline{\textup{E}}_1}\big(D_1^{\sqcup{n}}\sqcup M^{\sqcup m},D_1\big)$ is the empty set.

\item The composition law is given by~the composition of~embeddings, and a~unit map is the identity map.

\end{enumerate}
Let $\eonem^\otimes$ be the operadic nerve of~$\overline{\textup{E}}_1$.

\end{Definition}

\begin{Remark}
\label{eoneass}
There is an~equivalence from
the $\infty$-operad $\eone^\otimes$ to the associative
$\infty$-operad $\textup{Assoc}^\otimes$, see \cite[Definition~4.1.1.3]{HA}.
Indeed, if $f\colon(0,1)^{\sqcup n}\to (0,1)$ is a~rectilinear map,
then it determines a~linear ordering on the set of~connected component
$\pi_0\big((0,1)^{\sqcup n}\big)$ such that $I_1>I_2$ for two components $I_1$ and $I_2$
in $(0,1)^{\sqcup n}$ if $a<b$ for any $a\in f(I_1)$ and any $b\in f(I_2)$.
It~gives rise to a~map from $\textup{Emb}^{\rm rec}\big((0,1)^{\sqcup n},(0,1)\big)$
to the set of~linear ordering on $\pi_0\big((0,1)^{\sqcup n}\big)$.
It~is a~homotopy equivalence so that we have an~equivalence
$\eone^\otimes\to \textup{Assoc}^\otimes$
(for details, see \mbox{\cite[Example~5.1.0.7]{HA}}).

This equivalence
$\eone^\otimes\stackrel{\sim}{\to} \textup{Assoc}^\otimes$ is extended to an~equivalence $\eonem^\otimes\stackrel{\sim}{\to} \LM^\otimes$ of~$\infty$-operads,
where $\LM^\otimes$ is the $\infty$-operad (having two colors) which we use
to describe pairs of~associative algebras and left modules \cite[Definition~4.2.1.7]{HA}.
Indeed, as above, any map
$(0,1)^{\sqcup n}\sqcup (0,1)\to (0,1)$
in (ii) Definition~\ref{operadeonemod} determines a~linear ordering
on $\pi_0((0,1)^{\sqcup n}\sqcup (0,1))$ such that the black component
is the maximal element. It~is easy to see that
this ordering induces an~equivalence $\eonem^\otimes\to \LM^\otimes$
which extends $\eone^\otimes\stackrel{\sim}{\to} \textup{Assoc}^\otimes$.

\end{Remark}

{\bf 4.2.}
Following \cite{HA},
we recall the notion of~algebras over an~$\infty$-operad.
Let $\mathcal{O} \to \Gamma$ be an~$\infty$-operads.
Let $\mathcal{M}^\otimes\to \Gamma$ be a~symmetric monoidal $\infty$-category
whose underlying $\infty$-category is $\mathcal{M}=\mathcal{M}_{\langle 1\rangle}$.
An $\mathcal{O}$-algebra in $\mathcal{M}$
is a~map $f\colon\OO\to \mathcal{M}^\otimes$ over $\Gamma$
which preserve inert morphisms, that is, a~map of~$\infty$-operads.
We~define
$\Alg_{\mathcal{O}}\big(\mathcal{M}^\otimes\big)\subset \Fun_{\Gamma}\big(\mathcal{O},\mathcal{M}^\otimes\big)$ to be the full subcategory of~$\Fun_{\Gamma}\big(\mathcal{O},\mathcal{M}^\otimes\big)$ spanned by~$\OO$-algebras.
We~often write $\Alg_{\mathcal{O}}(\mathcal{M})$
for $\Alg_{\mathcal{O}}\big(\mathcal{M}^\otimes\big)$
when the structure on $\mathcal{M}$ is clear.
We~refer to $\Alg_{\mathcal{O}}\big(\mathcal{M}^\otimes\big)$
as the $\infty$-category of~$\OO$-algebra objects in $\mathcal{M}$,
cf.~\cite[Definition~2.1.3.1]{HA}.
When $\mathcal{O}$ is $\Gamma$,
we write $\CAlg\big(\mathcal{M}^\otimes\big)$
for $\Alg_{\Gamma}\big(\mathcal{M}^\otimes\big)$.

Let $\eone^\otimes\to \Gamma$ be the $\infty$-operad of~little $1$-cubes
with the natural projection.
Let $\big(BS^1\big)_{\Delta}$ be the simplicial category
having a~single object $*$ and
Hom simplicial set $\Hom_{(BS^1)_{\Delta}}(*,*)$.
The~simplicial set $\Hom_{(BS^1)_{\Delta}}(*,*)$
is the simplicial complex of~$S^1=\RRR/\ZZ$, and
the composition is induced by~the ordinary
multiplication $S^1\times S^1\to S^1$.
We~denote by~$BS^1$ the simplicial nerve of~$\big(BS^1\big)_{\Delta}$.
It~can also be regarded as the classifying space of~$S^1$ in $\SSS$.
Let $p\colon\eone^\otimes\times BS^1\to \Gamma$
be the composite $\eone^\otimes\times BS^1 \stackrel{\textup{pr}_1}{\to} \eone^\otimes\to \Gamma$. We~note that $\eone^\otimes\times BS^1\to \Gamma$ is not
an $\infty$-operad.
Let $\MM^\otimes\to \Gamma$ be a~symmetric monoidal $\infty$-category.
Though the above definition of~algebra objects is not applicable to
$\eone^\otimes\times BS^1\to \Gamma$,
we define $\Alg_{\eone^\otimes\times BS^1}\big(\MM^\otimes\big)$ as
follows (cf.~\cite[Definition~2.3.3.20]{HA}).
Let $\rho_i\colon\langle n\rangle \to \langle 1\rangle$ be
the unique inert morphism which sends $i\in \langle n\rangle$ to $1\in \langle 1\rangle$.
Then $\Alg_{\eone^\otimes\times BS^1}\big(\MM^\otimes\big)$ is the full subcategory of~$\Fun_{\Gamma}\big(\eone^\otimes\times BS^1,\MM^\otimes\big)$
spanned by~those
maps $F\colon\eone^\otimes\times BS^1\to \MM^\otimes$
satisfying the condition:
If~$C$ is an~object of~$\eone^\otimes\times BS^1$ lying over~$\langle n\rangle$,
and for $1\le i\le n$
$\alpha_i\colon C\to C_i$ is a~locally $p$-coCartesian morphism
covering
$\rho_i\colon\langle n\rangle \to \langle 1\rangle$,
then $F(\alpha_i)$ is an~inert morphism
in $\MM^\otimes$.
Let $\Delta^0\to BS^1$ be the natural functor which induces
$\eone^\otimes \to \eone^\otimes\times BS^1$.
By the construction of~$\eone^\otimes\times BS^1$,
$F$ belongs to $\Alg_{\eone^\otimes\times BS^1}\big(\MM^\otimes\big)$
if and only if the composite
$\eone^\otimes \to \eone^\otimes\times BS^1\stackrel{F}{\to} \MM^\otimes$
carries any inert morphism in $\eone^\otimes$ lying over
$\langle n\rangle \to \langle 1\rangle$ to an~inert morphism
in $\MM^\otimes$.
As observed in~\cite[Remark~2.1.2.9]{HA}, it is equivalent to the
condition that
$\eone^\otimes \to \eone^\otimes\times BS^1\stackrel{F}{\to} \MM^\otimes$
carries any inert morphism in $\eone^\otimes$ to an~inert morphism
in $\MM^\otimes$.

\begin{Lemma}
\label{eonesequiv}
Let $\MM^\otimes$ be a~symmetric monoidal $\infty$-category
whose underlying category we denote by~$\MM$.
Let $\Fun\big(BS^1,\MM\big)$ denote the functor category $($function complex$)$
which is endowed with the pointwise symmetric monoidal structure induced
by~that of~$\MM^\otimes$. Namely,
the symmetric monoidal structure on $\Fun\big(BS^1,\MM\big)$ is given by~the projection
$\Fun\big(BS^1,\MM\big)^\otimes:=\Fun\big(BS^1,\MM^\otimes\big)\times_{\Fun(BS^1,\Gamma)}\Gamma \to \Gamma$.
Then there is a~canonical equivalence of~$\infty$-categories
\begin{gather*}
\Alg_{\eone^\otimes\times BS^1}\big(\MM^\otimes\big)\simeq \Alg_{\eone}\big(\Fun\big(BS^1,\MM\big)^\otimes\big).
\end{gather*}
Similarly, there is a~canonical equivalence
$\Alg_{\eonem\times BS^1}\big(\MM^\otimes\big)\simeq \Alg_{\eonem}\big(\Fun\big(BS^1,\MM\big)^\otimes\big)$.
\end{Lemma}

\begin{proof}
We~prove
that there is an~isomorphism of~simplicial sets
\begin{gather*}
\Alg_{\eone^\otimes\times BS^1}\big(\MM^\otimes\big)\simeq \Alg_{\eone}\big(\Fun\big(BS^1,\MM\big)\big).
\end{gather*}
Observe that the symmetric monoidal $\infty$-category
$\Fun\big(BS^1,\MM\big)^\otimes$ is defined by~the following
universal property: for a~simplicial set $K$, there is a~natural
bijection of
\[ \Hom_{\sSet}\big(K,\Fun\big(BS^1,\MM\big)^\otimes\big)\] with the set of~pairs $(\alpha,\beta)$ which makes the diagram commute
\begin{gather*}
\xymatrix{
 & & \MM^\otimes \ar[d] \\
 BS^1\times K \ar[r]_{\textup{pr}_2} \ar[rru]^{\beta} & K \ar[r]^{\alpha} & \Gamma.
}
\end{gather*}
The assignment $(\alpha,\beta)\mapsto \alpha$ induces
$\Fun\big(BS^1,\MM\big)^\otimes\to \Gamma$.
Therefore, for a~simplicial set $L$, a~map $L\to \Alg_{\eone}\big(\Fun\big(BS^1,\MM\big)^\otimes\big)$ amounts to
a map $f\colon BS^1\times L\times \eone^\otimes\to \MM^\otimes$ over $\Gamma$
such that for any vertex $(a,l)$ in $BS^1\times L$ and for any inert
morphism $i$ in $\eone^\otimes$, the image $f((a,l,i))$ is an~inert morphism in
$\MM^\otimes$ (note also that
by~construction $BS^1$ has a~single vertex).
Next, we~consider the universal property of~$\Alg_{\eone^\otimes\times BS^1}\big(\MM^\otimes\big)$.
By the observation before this lemma, for a~simplicial set $L$,
a map $L\to \Alg_{\eone^\otimes\times BS^1}\big(\MM^\otimes\big)$
amounts to a~map $g$ such that the diagram
\begin{gather*}
\xymatrix{
 & & \MM^\otimes \ar[d] \\
L\times \eone^\otimes \times BS^1 \ar[r] \ar[rru]^{g} & \eone^\otimes \times BS^1 \ar[r] & \Gamma
}
\end{gather*}
commutes and
for any vertex $(l,a)$ in $L\times BS^1$ and for any inert
morphism $i$ in $\eone^\otimes$, the image $g((l,i,a))$ is an~inert morphism in
$\MM^\otimes$. Comparing universal properties of~$\Alg_{\eone}\!\big(\Fun\big(BS^1\!,\MM\big)\big)$ and $\Alg_{\eone^\otimes\times BS^1}\big(\MM^\otimes\big)$,
we have a~canonical isomorphism of~simplicial sets
$\Alg_{\eone}\!\big(\Fun\big(BS^1\!,\MM\big)\big)$ $\simeq \Alg_{\eone^\otimes\times BS^1}\big(\MM^\otimes\big)$.
The final assertion also follows from an~argument similar to this proof.
\end{proof}

\begin{Construction}
We~will define a~functor $\eone^\otimes \times BS^1\to \cyl$ over $\Gamma$.
To this end,
we consider the following simplicial categories $\big(\textup{E}_1\times BS^1\big)_{\Delta}$
and $\CYL_{\Delta}$.
Let $\big(\textup{E}_1\times BS^1\big)_{\Delta}$
be a~simplicial category defined as follows.
\begin{itemize}\itemsep=0pt

\item The objects of~$\big(\textup{E}_1\times BS^1\big)_{\Delta}$ are
objects $\langle n\rangle$ in $\FIN$ (which we regard as formal symbols).

\item For $\langle n\rangle, \langle m\rangle \in \big(\textup{E}_1\times BS^1\big)_{\Delta}$,
the Hom simplicial set $\Hom_{(\textup{E}_1\times BS^1)_{\Delta}}(\langle n\rangle, \langle m\rangle)$ is the singular complex of~the space
\begin{gather*}
\coprod_{\alpha\colon\langle n\rangle\to \langle m\rangle}
\bigg(
\prod_{1\le j\le m}\textup{Emb}^{\rm rec}\big((0,1)\times \alpha^{-1}(j),(0,1)\big)
\bigg)\times S^1.
\end{gather*}

\item The composition is determined by~the composition of~embeddings
and the multiplication $S^1\times S^1\to S^1$.

\end{itemize}
There is a~canoncial projection $\big(\textup{E}_1\times BS^1\big)_{\Delta}\to \FIN$.
The simplicial nerve $\NNNN\big(\big(\textup{E}_1\times BS^1\big)_{\Delta}\big)\to \NNNN(\FIN)=\Gamma$ is
$p\colon\eone^\otimes \times BS^1\to \Gamma$.

\begin{itemize}\itemsep=0pt

\item The objects of~$\CYL_{\Delta}$ are
objects $\langle n\rangle$ in $\FIN$ (which we regard as formal symbols).

\item For $\langle n\rangle, \langle m\rangle \in \CYL_{\Delta}$,
the Hom simplicial set $\Hom_{\CYL_{\Delta}}(\langle n\rangle, \langle m\rangle)$ is the singular simplicial complex of~the space
\begin{gather*}
\coprod_{\alpha\colon\langle n\rangle\to \langle m\rangle} \prod_{1\le j\le m}\textup{Emb}^{\rm rec}\big((0,1)\times S^1\times \alpha^{-1}(j),(0,1)\times S^1\big).
\end{gather*}
The composition is determined by~the composition of~embeddings.

\end{itemize}
Passing to simplicial nerves, we get $\NNNN(\CYL_{\Delta})\simeq \cyl\to \Gamma$.
Note that there is a~canonical
homeomorphism of~spaces
\begin{gather*}
\textup{Emb}^{\rm rec}\big((0,1)\times S^1\times \alpha^{-1}(j),(0,1)\times S^1\big)\simeq \textup{Emb}^{\rm rec}\big((0,1)\times \alpha^{-1}(j),(0,1)\big)\times \big(S^1\big)^{\times \sharp \alpha^{-1}(j)}
\end{gather*}
(each factor $a\in (S^1)^{\times \sharp \alpha^{-1}(j)}$
determines the rotation of~the restriction $(0,1)\times S^1 \to (0,1)\times S^1$). The diagonal map $S^1\to \big(S^1\big)^{\times \sharp \alpha^{-1}(j)}$
induces a~map
\begin{gather*}
\Hom_{(\textup{E}_1\times BS^1)_{\Delta}}(\langle n\rangle, \langle m\rangle)\to \Hom_{\CYL_{\Delta}}(\langle n\rangle, \langle m\rangle)
\end{gather*}
in the natural way. It~gives rise to a~functor
$Z\colon\big(\textup{E}_1\times BS^1\big)_{\Delta}\to \CYL_{\Delta}$ of~simplicial categories.
Then it induces a~map
\begin{gather*}
z\colon\ \eone^\otimes\times BS^1\to \cyl
\end{gather*}
over $\Gamma$.
Let $\MM^\otimes$ be a~symmetric monoidal $\infty$-category
whose underlying category is $\MM$.
The restriction along $\eone^\otimes \times BS^1\to \cyl$ induces
a functor $z^*\colon\Alg_{\cyl}\big(\MM^\otimes\big)\to \Alg_{\eone^\otimes \times BS^1}\big(\MM^\otimes\big)$.
\end{Construction}

\begin{Proposition}\label{cyleones}The functor
\begin{gather*}
z^*\colon\ \Alg_{\cyl}\big(\MM^\otimes\big)\to \Alg_{\eone^\otimes \times BS^1}\big(\MM^\otimes\big)
\end{gather*}
is an~equivalence of~$\infty$-categories.
\end{Proposition}

Observe that
$\mul_{\CYLM}\big(C^{\sqcup{n}}\sqcup \cbar,\cbar\big)$ is the singular complex of~the space which is homeomorphic to $\mul_{\overline{E}_1}\big(D_1^{\sqcup{n}}\sqcup M,M\big)\times \big(S^{1}\big)^{\times n+1}$.
As in the case of~$\eone^\otimes \times BS^1\to \cyl$, using $\mul_{\overline{E}_1}\big(D_1^{\sqcup{n}}\sqcup M,M\big)\times S^{1}\to \mul_{\overline{E}_1}\big(D_1^{\sqcup{n}}\sqcup M,M\big)\times \big(S^{1}\big)^{\times n+1}$
induced by~the diagonal map $S^1\to \big(S^{1}\big)^{\times n+1}$
we obtain a~morphism $\bar{z}\colon\eonem^\otimes \times BS^1\to \cylm$ over $\Gamma$. The restriction along $\bar{z}$ gives rise to
$\bar{z}^*\colon\Alg_{\cylm}\big(\MM^\otimes\big)\to \Alg_{\eonem^\otimes \times BS^1}\big(\MM^\otimes\big)$.

\begin{Proposition}\label{cylmeonems}The functor
\begin{gather*}
\bar{z}^*\colon\ \Alg_{\cylm}\big(\MM^\otimes\big)\to \Alg_{\eonem^\otimes \times BS^1}\big(\MM^\otimes\big)
\end{gather*}
is an~equivalence of~$\infty$-categories.
\end{Proposition}

\begin{Corollary}\label{equivariant}There are canonical equivalences of~$\infty$-categories
\begin{gather*}
\Alg_{\cyl}\big(\MM^\otimes\big)\simeq \Alg_{\eone^\otimes \times BS^1}\big(\MM^\otimes\big)\simeq \Alg_{\eone}\big(\Fun\big(BS^1,\MM\big)^\otimes\big)\\
\hphantom{\Alg_{\cyl}\big(\MM^\otimes\big)}{} \simeq \Alg_{\textup{Assoc}}\big(\Fun\big (BS^1,\MM\big)^\otimes\big)
\end{gather*}
and
\begin{gather*}
\Alg_{\cylm}\big(\MM^\otimes\big)\simeq \Alg_{\eonem^\otimes \times BS^1}\big(\MM^\otimes\big)\simeq \Alg_{\eonem}\big(\Fun\big(BS^1,\MM\big)^\otimes\big)\\
\hphantom{\Alg_{\cylm}\big(\MM^\otimes\big)}{}\simeq \LMod\big(\Fun\big(BS^1,\MM\big)^\otimes\big).
\end{gather*}
Here $\LMod\big(\Fun\big(BS^1,\MM\big)^\otimes\big)=\Alg_{\LM^\otimes}\big(\Fun\big(BS^1,\MM\big)^\otimes\big)$.
\end{Corollary}

\begin{proof}It~follows from Lemma~\ref{eonesequiv}, Propositions~\ref{cyleones} and~\ref{cylmeonems}, and Remark~\ref{eoneass}.
\end{proof}

The proof~of~Proposition~\ref{cylmeonems} is similar to
Proposition~\ref{cyleones}.
We~prove Proposition~\ref{cyleones}.

\begin{proof}[Proof~of~Proposition~\ref{cyleones}]
We~use the notion of~a~weak approximation
in the sense of~\cite[Definition~2.3.3.6]{HA}
(since $z\colon\eone^\otimes\times BS^1 \to \cyl$ is not an~equivalence of~$\infty$-operads, it is necessary to use more delicate notion).
According to~\cite[Theorem~2.3.3.23]{HA}, if two conditions
\begin{enumerate}\itemsep=0pt
\renewcommand{\labelenumi}{(\roman{enumi})}

\item $\eone^\otimes\times BS^1 \to \cyl$ is a~weak approximation,

\item $z\colon\eone^\otimes\times BS^1 \to \cyl$
induces an~equivalence between the fiber over $\langle 1\rangle$
\end{enumerate}
hold,
then $z^*\colon\Alg_{\cyl}\big(\MM^\otimes\big)\to \Alg_{\eone^\otimes \times BS^1}\big(\MM^\otimes\big)$ is an~equivalence.
We~first prove (ii).
Let
$z_{\langle1\rangle}\colon\big(\eone^\otimes\times BS^1\big)_{\langle 1\rangle} \to \cyl_{\langle 1\rangle}$
be the map of~fibers over $\langle 1\rangle$.
Both fibers consist of~a~unique object (here we denote it by~$*$
whose mapping space $\Map(*,*)$ is (homotopy) equivalent to $S^1$).
Taking into account our construction of~$Z\colon\big(\textup{E}_1\times BS^1\big)_{\Delta}\to \CYL_{\Delta}$, we see that $z_{\langle1\rangle}$ is a~homotopy equivalence $BS^1\to BS^1$.
This proves~(ii).
Next we will prove~(i).
Let $p\colon\eone^\otimes\times BS^1\to \Gamma$ be the projection.
Let $\textup{Tup}_n$ be the subcategory of~$\Gamma_{/\langle n\rangle}$
whose objects are active morphisms $\langle m\rangle \to \langle n\rangle$
and whose morphisms are equivalences.
According to a~criterion \cite[Proposition~2.3.3.14]{HA},
to~prove (i), it is enough to prove that
for any $X\in \eone^\otimes\times BS^1$ with $\langle n\rangle=p(X)$,
$z$ induces a~weak homotopy equivalence
\begin{gather*}
u\colon\ A(X):=\big(\eone^\otimes\times BS^1\big)_{/X}\times_{\Gamma_{/\langle n\rangle}}\textup{Tup}_n\to
\cyl_{/z(X)}\times_{\Gamma_{/\langle n\rangle}}\textup{Tup}_n=:B(z(X)).
\end{gather*}
Note that both domain and target are $\infty$-categories.
Consequently, it will suffice to show that
$u$ is a~categorical equivalence.
Clearly, $u$ is essentially surjective.
We~prove that $u$ is fully faithful.
The general case is essentially the same as the case $n=1$
except for a~more complicated notation, so that
we treat the case of~$n=1$.
We~think of~$D_1$ as the unique object of~$\big(\eone^\otimes\times BS^1\big)_{\langle 1\rangle}$.
Also, we write $D_1^m$ for the unique object of~the fiber $\big(\eone^\otimes\times BS^1\big)_{\langle m\rangle}$ over $\langle m\rangle$ (namely, $D_1=D_1^1$).
Let $f\colon D_1^m \to D_1$ be a~map
in $\eone^\otimes\times BS^1$ lying over an~active morphism
$\alpha\colon\langle m\rangle \to \langle 1\rangle$ of~$\Gamma$.
We~regard~$f$ as the product
$g\times h\colon(0,1)^{\sqcup m}\times S^1\to (0,1)\times S^1$
of~a~rectilinear map $\phi\colon(0,1)^{\sqcup m}\to (0,1)$
and a~rectilinear map $h\colon S^1\to S^1$.
Let $f'\colon D_1^m \to D_1$ be another map
in $\eone^\otimes\times BS^1$ lying over~$\alpha$.
We~have
\begin{gather*}\begin{split}&
\Map_{A(D_1)}(f,f')\simeq \Map^{\rm eq}_{\eone^\otimes\times BS^1}\big(D_1^m,D_1^m\big)\times_{\Map_{\eone^\otimes\times BS^1} (D_1^m,D_1)}\{f\}
\\ & \hphantom{\Map_{A(D_1)}(f,f')}
{}\simeq \big(S^1\times\Sigma_m\big)\times_{(S^1\times \Sigma_m)}\{f\},\end{split}
\end{gather*}
where $\Map^{\rm eq}_{\eone^\otimes\times BS^1}\big(D_1^m,D_1^m\big)$
is the full subcategory of~$\Map_{\eone^\otimes\times BS^1}\big(D_1^m,D_1^m\big)$
spanned by~equi\-va\-len\-ces, and
$\Map^{\rm eq}_{\eone^\otimes\times BS^1}\big(D_1^m,D_1^m\big) \to \Map_{\eone^\otimes\times BS^1}\big(D_1^m,D_1\big)$ is induced by~the composition with~$f'$.
Here $\Sigma_m$ denotes the symmetric group (which comes from permutations
of~components).
Thus, the mapping space is contractible because $(S^1\times \Sigma_m)\to \big(S^1\times \Sigma_m\big)$ is an~equivalence.
Next, we regard the color $C$ in $\CYL$ as an~object in $\cyl$ that lies
over $\langle 1 \rangle$. We~denote by~$C^m$ the unique object of~$\cyl$
that lies over $\langle m\rangle$.
Let $z(f),z(f')\colon C^m\to C$ be the images of~$f$ and~$f'$,
respectively.
Then we have equivalences in $\SSS$
\begin{gather*}
\Map_{B(C)}(z(f),z(f')) \simeq \Map^{\rm eq}_{\cyl}\big(C^m,C^m\big)\times_{\Map_{\cyl}(C^m,C)}\{z(f)\}
\\ \hphantom{\Map_{B(C)}(z(f),z(f'))}
{}\simeq \big(\big(S^1\big)^{\times m}\times\Sigma_m\big)\times_{((S^1)^{\times m} \times \Sigma_m)}\{z(f)\},
\end{gather*}
where
$\Map^{\rm eq}_{\cyl}\big(C^m,C^m\big)$ is the full subcategory of~$\Map_{\cyl}(C^m,C^m)$ spanned by~equivalences.
It~follows from the canonical equivalence
$\big(S^1\big)^{\times m}\times\Sigma_m\to\big(S^1\big)^{\times m}\times \Sigma_m$ that $\Map_{B(C)}(z(f),z(f'))$ is contractible. Thus,
$\Map_{A(D_1)}(f,f')\to \Map_{B(C)}(z(f),z(f'))$ is an~equivalence. We~conclude that $u$ is a~categorical equivalence.
\end{proof}

{\bf 4.3.}
Let $\MM^\otimes$ be a~symmetric monoidal $\infty$-category.
If~$\mathbf{P}^\otimes$ and $\mathbf{Q}^\otimes$ are small
$\infty$-operads
and $c\colon\mathbf{P}^\otimes\to \mathbf{Q}^\otimes$ is a~morphism
of~$\infty$-operads (over $\Gamma$), then we denote
by~$c^*\colon\Alg_{\mathbf{Q}}(\MM)\to \Alg_{\mathbf{P}}(\MM)$
the restriction/forgetful functor
along $c$.
If~there exists a~left adjoint of~$c^*$,
we denote it by~$c_{!}\colon \Alg_{\mathbf{P}}(\MM)\to \Alg_{\mathbf{Q}}(\MM)$.
If~$\MM^\otimes$ admits small colimits
and its monoidal multiplication functor
$\MM\times \MM\to \MM$ preserves small colimits separately
in each variable,
then there exists a~left adjoint $c_!$, cf.~\cite[Corollary~3.1.3.5]{HA}.

The diagram $\dcyl \stackrel{a}{\leftarrow} \cyl \stackrel{b}{\to} \cylm$
induces
$a^*\colon\Alg_{\dcyl}(\MM)\to \Alg_{\cyl}(\MM)$ and
$b^*\colon$ \linebreak $\Alg_{\cylm}(\MM)\to \Alg_{\cyl}(\MM)$.
It~gives rise to the fiber product $\Alg_{\dcyl}(\MM)\times_{\Alg_{\cyl}(\MM)}\Alg_{\cylm}(\MM)$.
We~also have the inclusions $\dcyl \to \dcylm$ and $\cylm\to \dcylm$.
Then
the restriction/forgetful functors induce
\begin{gather*}
\Alg_{\dcylm}(\MM)\to \Alg_{\dcyl}(\MM)\times_{\Alg_{\cyl}(\MM)}\Alg_{\cylm}(\MM).
\end{gather*}
Now suppose that
$\MM^\otimes$ admits small colimits
and the tensor product functor
$\MM\times \MM\to \MM$ preserves small colimits separately
in each variable.
Let $i\colon\etwo^\otimes \to \dcyl$ be the canonical inclusion.
We~then have an~adjoint pair
\begin{gather*}
i_!\colon\ \Alg_{\etwo}(\MM)\rightleftarrows \Alg_{\dcyl}(\MM)\ \,{\colon}i^*,
\end{gather*}
where $i_!$ is fully faithful (indeed, the left adjoint
is given by~the operadic left Kan extension so that
the unit map of~the adjunction is the identity).
We~write $\Alg^{\dtwo}_{\dcyl}(\MM)$ for the essential image of~$i_!$.
Note that $\Alg_{\etwo}(\MM)\simeq \Alg^{\dtwo}_{\dcyl}(\MM)$.
We~set $\Alg^{\dtwo}_{\dcylm}(\MM):=\Alg_{\dcylm}(\MM)\times_{\Alg_{\dcyl}(\MM)}\Alg^{\dtwo}_{\dcyl}(\MM)$.

\begin{Proposition}
\label{another}
Suppose that $q\colon\MM^\otimes \to \Gamma$ is a~symmetric monoidal $\infty$-category
such that the underlying $\infty$-category $\MM$ has small colimits and
the tensor product functor $\otimes\colon\MM \times \MM\to \MM$
preserves small colimits separately in each variable.
Then the functor
\begin{gather*}
\Alg_{\dcylm}(\MM)\to \Alg_{\dcyl}(\MM)\times_{\Alg_{\cyl}(\MM)}\Alg_{\cylm}(\MM)
\end{gather*}
is an~equivalence of~$\infty$-categories.
Moreover, it induces an~equivalence of~$\infty$-categories
\begin{gather*}
\Alg^{\dtwo}_{\dcylm}(\MM)\stackrel{\sim}{\to} \Alg^{\dtwo}_{\dcyl}(\MM)\times_{\Alg_{\cyl}(\MM)}\Alg_{\cylm}(\MM).
\end{gather*}
\end{Proposition}

\begin{Proposition}
\label{another2}
The restriction along the inclusion $\KS\to \dcylm$
induces an~equivalence of~$\infty$-categories
\begin{gather*}
\Alg^{\dtwo}_{\dcylm}(\MM)\to \Alg_{\KS}(\MM).
\end{gather*}
\end{Proposition}

Now we consider the diagram
\begin{gather*}
\xymatrix{
\Alg_{\KS}(\MM) & \ar[l] \Alg_{\dcylm}(\MM) \ar[r]^(0.3){\simeq} & \Alg_{\dcyl}(\MM)\times_{\Alg_{\cyl}(\MM)}\Alg_{\cylm}(\MM) \\
}
\end{gather*}
and its restriction to $\Alg_{\dcylm}^{\dtwo}(\MM)$.

\begin{Corollary}\label{algebraequivalence}
We~have canonical categorical equivalences
\begin{gather*}
\Alg_{\etwo}(\MM)\times_{\Alg_{\cyl}(\MM)}\Alg_{\cylm}(\MM) \simeq \Alg^{\dtwo}_{\dcyl}(\MM)\times_{\Alg_{\cyl}(\MM)}\Alg_{\cylm}(\MM)\simeq \Alg_{\KS}(\MM).
\end{gather*}
Moreover, by~Corollary~{\rm \ref{equivariant}},
the $\infty$-category on the left-hand side
is equivalent to
\begin{gather*}
\Alg_{\etwo}(\MM)\times_{\Alg_{\eone}(\Fun(BS^1,\MM))}\LMod\big(\Fun\big(BS^1,\MM\big)\big),
\end{gather*}
where $\Alg_{\etwo}(\MM)\to\Alg_{\eone}\big(\Fun\big(BS^1,\MM\big)\big)$
is the composite
\begin{gather*}
\Alg_{\etwo}(\MM)\stackrel{i_!}{\simeq}\Alg^{\dtwo}_{\dcyl}(\MM)\to \Alg_{\cyl}(\MM)\simeq \Alg_{\eone}\big(\Fun\big(BS^1,\MM\big)\big).
\end{gather*}
In~particular, we have an~equivalence of~$\infty$-categories
\begin{gather*}
\Alg_{\KS}(\MM)\simeq \Alg_{\etwo}(\MM)\times_{\Alg_{\eone}(\Fun(BS^1,\MM))}\LMod\big(\Fun\big(BS^1,\MM\big)\big).
\end{gather*}
This equivalence commutes with projections to $\Alg_{\etwo}(\MM)$
in the natural way.
\end{Corollary}

The proof~of~Proposition~\ref{another} requires
Lurie--Barr--Beck theorem \cite[Corollary 4.7.3.16]{HA}.
Let us consider the comutative diagram
\begin{gather*}
\xymatrix{
\Alg_{\dcylm}(\MM) \ar[rr]^U \ar[dr]_G & & \Alg_{\dcyl}(\MM)\times_{\Alg_{\cyl}(\MM)}\Alg_{\cylm}(\MM) \ar[ld]^{G'} \\
 & \Alg_{\dcyl}(\MM) \times \MM. &
}
\end{gather*}
The functor $U$ is the functor in the statement in Proposition~\ref{another},
and $G$ is determined by~the forgetful/restriction functor $\Alg_{\dcylm}(\MM)\to \Alg_{\dcyl}(\MM)$
and the functor $\Alg_{\dcylm}(\MM)\to \MM$
given by~the evaluation at $C_M$. The functor $G'$ is determined by~the first projection and the functor $\Alg_{\cylm}(\MM) \to \MM$
given by~the evaluation at $C_M$.
According to Lurie--Barr--Beck theorem \cite[Corollary 4.7.3.16]{HA},
$U$ is a~categorical equivalence if the following
conditions are satisfied:
\begin{enumerate}\itemsep=0pt
\renewcommand{\labelenumi}{(\roman{enumi})}

\item The functors $G$ and $G'$ admit left adjoints $F$ and $F'$, respectively.

\item $\Alg_{\dcylm}(\MM)$ admits geometric realizations of~simplicial objects, which are preserved by~$G$.

\item $\Alg_{\dcyl}(\MM)\times_{\Alg_{\cyl}(\MM)}\Alg_{\cylm}(\MM)$ admits geometric realizations of~simplicial objects, which are preserved by~$G'$.

\item The functors $G$ and $G'$ are conservative.

\item For any object $X$ in $\Alg_{\dcyl}(\MM)\times \MM$,
the unit map $X\to GF(X)\simeq G'UF(X)$ induces an~equivalence
$G'F'(X)\to GF(X)$ in
$\Alg_{\dcyl}(\MM)\times \MM$.

\end{enumerate}

\begin{proof}[Proof~of~Proposition~\ref{another}]
We~will prove the conditions (i), (ii), (iii), (iv), (v).

A morphism $f$ in $\Alg_{\dcylm}(\MM)$ is an~equivalence
if and only if the evaluations at $D$, $C$, $C_M$
are equivalences.
Similarly,
a morphism $f$ in $\Alg_{\dcyl}(\MM)$ is an~equivalence
if and only if the evaluations at objects $D, C$
are equivalences.
 It~follows that
$G$ is conservative. Similarly, we~see that $G'$ is conservative.
Hence (iv) is proved.
The conditions (ii) and (iii) follow from
the existence and the compatibility of~sifted colimits \cite[Proposition~3.2.3.1]{HA}
and the conservativity in (iv).
(If~$\MM^\otimes$ is a~presentably symmetric monoidal $\infty$-category,
the condition (i) follows from adjoint functor theorem
since $G$ and $G'$ preserve small limit and filtered colimit. We~prove (i) in full generality below.)

We~prove the conditions (v) and (i).
For this purpose, we first consider the left adjoint~$F'$ of~$G'$.
The values under~$F'$ can be described in terms of~operadic colimits
if we assume the existence of~a~left adjoint $F'$.
Let $j\colon\cyl\to \cylm$ be the canonical inclusion.
Let $j'\colon\triv^\otimes \to \cylm$ be a~morphism from the trivial
$\infty$-operad to $\cylm$ that is determined by~$C_M$.
Let $\cyl\boxplus \triv^\otimes$ denote the coproduct of~$\cyl$ and $\triv^\otimes$. Namely, it is a~coproduct of~$\cyl$ and $\triv^\otimes$
in the $\infty$-category of~$\infty$-operads, but
we use its explicit construction in~\cite[Construction~2.2.3.3]{HA}.
By the universal property, the morphism $j$ and $j'$ induces
$k\colon\cyl\boxplus \textup{Triv}^\otimes\to \cylm$.
By \cite[Corollary~3.1.3.5]{HA},
we have an~adjoint pair
\begin{gather*}
k_!\colon\ \Alg_{\cyl}(\MM)\times \MM \simeq \Alg_{\cyl\boxplus \textup{Triv}^\otimes}(\MM) \rightleftarrows\Alg_{\cylm}(\MM)\ {\colon}\!k^*.
\end{gather*}
Here we use the canonical categorical
equivalence $\Alg_{\triv}(\MM)\stackrel{\sim}{\to}\MM$.
There are categorical fibrations $\textup{pr}_1\colon\Alg_{\cyl}(\MM)\times \MM\to \Alg_{\cyl}(\MM)$ and $\Alg_{\cylm}(\MM)\to \Alg_{\cyl}(\MM)$
induced by~the inclusion $\cyl\hookrightarrow \cylm$.
The right adjoint $k^*$ commutes with these projections to $\Alg_{\cyl}(\MM)$.
Let $Y\colon\cyl\boxplus \textup{Triv}^\otimes\to \MM^\otimes$ be a~map
of~$\infty$-operads, which we regard as an~object~$Y$ of~$\Alg_{\cyl\boxplus \textup{Triv}^\otimes}(\MM)$.
For $A\in \cylm$,
the evaluation of~$k_!(Y)$ at $A$
is an~operadic $q$-colimit of~the~map
\begin{gather*}
\big(\cyl\boxplus \triv^\otimes\big)_{/A}^{\textup{act}}:=\big(\cyl\boxplus \triv^\otimes\big)\times_{\cylm}\cylm_{/A}^{\textup{act}}\to \cyl\boxplus \triv^\otimes\stackrel{Y}{\to} \MM^\otimes,
\end{gather*}
lying over $\big((\cyl\boxplus \triv^\otimes)_{/A}^{\textup{act}}\big)^{\triangleright}\to \Gamma$.
See \cite[Sections~3.1.1--3.1.3]{HA} for operadic left Kan extensions and operadic
colimits.
Note that
$\mul_{\CYLM}\big(C^{\sqcup{n}}\sqcup \cbar^{\sqcup m},C\big)$ is the empty set
for $m\neq 0$.
Hence
$\big(\cyl\boxplus \triv^\otimes\big)\times_{\cylm}\cylm_{/C}^{\textup{act}}\simeq \cyl\times_{\cylm}\cylm_{/C}^{\textup{act}}$ so that there is a~final object
determined by~the identity
$C\to C$.
It~follows that
$\textup{pr}_1(Y) \to \textup{pr}_1k^*k_!(Y)$ is an~equivalence in $\Alg_{\cyl}(\MM)$ for each $Y \in \Alg_{\cyl\boxplus \textup{Triv}^\otimes}(\MM)\simeq \Alg_{\cyl}(\MM)\times \MM$.
Thus $(k_!,k^*)$ is an~adjunction relative to $\Alg_{\cyl}(\MM)$.
See
\cite[Definition~7.3.2.2]{HA} for the notion of~relative adjunctions.
The base change of~$(k_!,k^*)$ along $\Alg_{\dcyl}(\MM)\to\Alg_{\cyl}(\MM)$
gives rise to an~adjunction
\begin{gather*}
l_!\colon \Alg_{\dcyl}(\MM)\!\times_{\Alg_{\cyl}(\MM)}(\Alg_{\cyl}(\MM)
\!\times \MM)\rightleftarrows \Alg_{\dcyl}(\MM)\!\times_{\Alg_{\cyl}(\MM)}\Alg_{\cylm}(\MM)\ {\colon}\!l^*
\end{gather*}
relative to $\Alg_{\dcyl}(\MM)$,
where the right adjoint is $G'$.
This shows that
there exists a~left adjoint $F'$ of~$G'$.
In~addition, $F'\simeq l_!$.
We~denote informally by~$X=(A,B,M)$ an~object
of~$\Alg_{\dcyl}(\MM)\times \MM\simeq \Alg_{\cyl\boxplus \textup{Triv}^\otimes}(\MM)$, where $(A,B)\in \Alg_{\dcyl}(\MM)$, $M\in \MM$,
$A$ is the restriction to $\etwo^\otimes\subset \dcyl$,
and $B$ is the restriction to $\cyl\subset \dcyl$.
We~compute the image of~$(A,B,M)$ under the left adjoint $l_!$.
The left adjoint $F'=l_!$
is induced by~$k_!$ so that
$l_!(A,B,M)=((A,B),k_!(B,M))$,
where $((A,B),k_!(B,M))$ indicates the object of~the fiber product on the right-hand side.
The pair $(B,M)$ is an~object of~$\Alg_{\cyl}(\MM)\times \MM$.
Thus we will compute $k_!(B,M)$
in terms of~the operadic left Kan extension;
we describe it as a~colimit of~a~certain diagram.
Let $(B,M)\colon\cyl \boxplus \triv^\otimes\to \MM^\otimes$ be
a morphism of~$\infty$-operads which corresponds to $(B,M)\in \Alg_{\cyl}(\MM)\times \MM$. Let $\cyl \boxplus \triv^\otimes \to \cylm$ be a~morphism
determined by~$\cyl\to \cylm$ and the morphism $\triv^\otimes\to \cylm$
classified by~the object $C_M$ in the fiber $\cylm_{\langle 1\rangle}$
(by~using the same symbol $C_M$
we abuse notation).
Let $\cylm^{\textup{act}}$ be the subcategory spanned by~those morphisms
whose images in $\Gamma$ are active (cf.~\cite[Definition~2.1.2.1]{HA}).
Let us consider
$\big(\cyl \boxplus \triv^\otimes\big)^{\textup{act}}_{/C_M}=\big(\cyl \boxplus \triv^\otimes\big)\times_{\cylm}\cylm^{\textup{act}}_{/C_M}$.
We~have a~morphism
\begin{gather*}
p\colon\ K:=\big(\cyl \boxplus \triv^\otimes\big)^{\textup{act}}_{/C_M}\to \cyl \boxplus \triv^\otimes\stackrel{(B,M)}{\longrightarrow} \MM^\otimes
\end{gather*}
that extends $(\cyl\boxplus\triv^\otimes)^{\textup{act}}_{/C_M}\to \Gamma$.
According to~\cite[Proposition~3.1.3.2]{HA}, the operadic left Kan extension $k_!(B,M)\colon\cylm\to \MM^\otimes$
of~$k$ carries $C_M$ to a~colimit of~$p$, that is, the image of~the cone point
under an~operadic $q$-colimit diagram $p'\colon K^{\triangleright}\to \MM^\otimes$
that extends $p$.
Let $\star$ denote the unique object of~$\triv^\otimes$
lying over $\langle1\rangle$.
Let $\textup{Sub}$ be a~category defined as follows \mbox{\cite[Definition~2.2.3.2]{HA}}: The objects of~$\textup{Sub}$ are triples $(\langle n \rangle, S,T)$
such that $\langle n \rangle \in \FIN$, and $S,T\subset \langle n \rangle$
are subsets such that $S\cap T=\ast$ and $S\cup T=\langle n \rangle$. A morphism $(\langle n \rangle, S,T)\to (\langle n' \rangle, S',T')$ is a~morphism $\langle n \rangle \to \langle n \rangle$ in $\FIN$ such that $f(S)\subset S'$ and $f(T)\subset T'$.
Let
\begin{gather*}
(\langle n \rangle, S,T,C^{n-1}=(C,\dots,C),\star)
\end{gather*}
be an~object
of~$\cyl\boxplus\triv^\otimes$ lying over $\langle n\rangle$
such that $(\langle n \rangle, S,T)$ is an~object of~$\textup{Sub}$,
$C^{n-1}$ is the unique
 object of~$\cyl_{\langle n-1\rangle}$ (lying over $\langle n-1\rangle$),
 and $\star\in \triv=\triv^\otimes_{\langle 1\rangle}$.
This presentation is based on the explicit construction of~the coproducts in~\cite[Construction~2.2.3.3]{HA}.
For our purpose below,
we may assume that $T\subset \langle n\rangle$ is of~the form $T=\{*,i\}$
so that by~default $T$ in the above object is
of~the form $T=\{*,i\}$.
The mapping space from
$\big(\langle n \rangle, S,T,C^{n-1},\star\big)$ to~$\big(\langle m \rangle, S',T',C^{m-1},\star\big)$
is given by
\begin{gather*}
\Bigg(\coprod_{\begin{subarray}{c}\alpha\colon\langle n\rangle\to \langle m\rangle\\\alpha(S)\subset S',\ \alpha(T)\subset T'\end{subarray}}
\prod_{j\in \langle m\rangle\backslash T'}
\textup{Emb}^{\rm rec}\big(\big((0,1)\times S^1\big)\times\alpha^{-1}(j), (0,1)\times S^1\big)
\Bigg)\times *,
\end{gather*}
which we regard as an~object in $\SSS$, and $*$ indicates the contractible
space which we regard as the mapping space from $\star$ to $\star$.
Using this description we consider mapping spaces in
$(\cyl\boxplus\triv^\otimes)^{\textup{act}}_{/C_M}$.
We~abuse notation by~writing
$\big(\langle n \rangle, S,T,C^{n-1},\star,\big(C^{n-1},C_M\big)\to C_M\big)$
for an~object of~$K$, where
$\big(\langle n \rangle, S,T,C^{n-1},\star\big)\in \cyl\boxplus\triv^\otimes$
and $(C,\dots,C,C_M)=\big(C^{n-1},C_M\big)\to C_M$ is a~morphism
in $\cylm^{\textup{act}}$
lying over the active morphism $\langle n \rangle\to \langle 1\rangle$,
where $\big(C^{n-1},C_M\big)$ is a~sequence of~$n-1$ $C$'s and a~single $C_M$
which we regard as an~object in $\cylm_{\langle n \rangle}$.
Now it is easy to compute the mapping space
from
$H=\big(\langle n \rangle, S,T,C^{n-1},\star,\big(C^{n-1},C_M\big)\stackrel{f}{\to} C_M\big)$
to $H'=\big(\langle m \rangle, S',T',C^{m-1},\star,\big(C^{m-1},C_M\big)\stackrel{g}{\to} C_M\big)$
in
$\big(\cyl\boxplus\triv^\otimes\big)\times_{\cylm}\cylm^{\textup{act}}_{/C_M}$.
We~think of~$f$ as an~embedding
$\big((0,1)\times S^1\big)^{\sqcup n-1}\sqcup (0,1)\times S^1\to (0,1)\times S^1$
that belongs to
$\mul_{\cylm}\big(C^{\sqcup n-1}\sqcup C_M,C_M\big)$
(namely,
the restricition to the ``right component" $(0,1)\times S^1$ is shrinking).
Consider the restriction $\big((0,1)\times S^1\big)^{\sqcup n-1}\to (0,1)\times S^1$
and its projection $\overline{f}\colon(0,1)^{\sqcup n-1}\to (0,1)$ obtained
by~forgetting the $S^1$-factor, which is
a rectilinear embedding.
If~we denote by~$D_1^n$
the unique object in the fiber $\big(\eone^\otimes\big)_{\langle n\rangle}$
over $\langle n\rangle\in \Gamma$,
we can regard $\overline{f}$ as
a map $D_1^{n-1} \to D_1:=D_1^1$ in $\eone^\otimes$.
Let $\Map_{(\eone^\otimes)_{/D_1}}(\overline{f},\overline{g})$
be the discrete mapping space from $\overline{f}\colon D_1^{n-1} \to D_1$ to $\overline{g}\colon D_1^{m-1} \to D_1$.
Given a~morphism $H\to H'$, by~applying the same procedure to
the induced morphism $C^{n-1}\to C^{m-1}$,
we obtain a~map $\Map_K(H,H')\to \Map_{(\eone^\otimes)_{/D_1}}(\overline{f},\overline{g})$.
Note that $\Map_{\cylm^{\textup{act}}_{/C_M}}(C_M,C_M)$ is contractible.
Consequently, the restriction to the component $C_M$ gives rise to
\begin{gather*}
\Map_K(H,H')\to \Map_{\triv}(\star,\star)\times_{\Map_{\cylm}(C_M,C_M)}\Map_{\cylm^{\textup{act}}_{/C_M}}(C_M,C_M) \simeq *\times_{S^1}*=\ZZ.
\end{gather*}
Taking account of~definitions of~$\infty$-operads
$\cyl\boxplus\triv^\otimes$ and $\cylm$,
we see that
\begin{gather*}
\Map_K(H,H')\to \Map_{(\eone^\otimes)_{/D_1}}(\overline{f},\overline{g})\times \ZZ
\end{gather*}
is an~equivalence in $\SSS$.
Let $L$
be the full subcategory of~$K$ spanned by~the single object
\begin{gather*}
\mathsf{Z}=\big(\langle 2 \rangle, S\simeq \langle 1\rangle,T\simeq \langle 1\rangle,C,\star,(C,C_M)\stackrel{j}{\to} C_M\big)
\end{gather*}
(a morphism $j$ is uniquely determined up to homotopy).
We~now claim that $L\subset K$ is cofinal.
It~will suffice to prove that for each $\mathsf{V}\in K$, the $\infty$-category
$L\times_KK_{\mathsf{V}/}$ is weakly contractible, see \cite[Definition~4.1.3.1]{HA}.
Let $\mathsf{V}=\big(\langle n \rangle,S,T, C^{n-1},\star,\big(C^{n-1},C_M\big)\stackrel{f}{\to} C_M\big)$.
By the above discussion about mapping spaces,
a morphism $u\colon\mathsf{V}\to \mathsf{Z}$ is
uniquely determined by~$a\in \ZZ$
since $\Map_{(\eone^\otimes)_{/D_1}}(\overline{f},\overline{j})$
is contractible.
Let $u'\colon\mathsf{V}\to \mathsf{Z}$
be another object of~$L\times_KK_{\mathsf{V}/}$ that is determined by~$a'\in \ZZ$. Note that $\Map_K(\mathsf{Z},\mathsf{Z})\simeq \ZZ$
and the composition $\Map_K(\mathsf{Z},\mathsf{Z})\times \Map_{K}(\mathsf{V},\mathsf{Z})\simeq \ZZ\times \ZZ\to \Map_{K}(\mathsf{V},\mathsf{Z})\simeq \ZZ$
can be identified with the additive operation $+\colon\ZZ\times \ZZ\to \ZZ$
(up~to~automorphisms of~$\ZZ$).
It~follows that
$\Map_{L\times_KK_{\mathsf{V}/}}(u,u')$ is a~contractible space
so that the $\infty$-category $L\times_KK_{\mathsf{V}/}$ is contractible.
Let $p'\colon K^{\triangleright}\to \MM^\otimes$ be the operadic $q$-colimit
diagram. Let $p''\colon K^{\triangleright}\to \MM$ be the diagram
obtained by~a $q$-coCartesian natural transformation from $p'$.
Since we assume that $\otimes\colon\MM\times \MM\to \MM$
preserves small colimits separately in each variable, then
by~\cite[Propositions~3.1.1.15 and~3.1.1.16]{HA},
$p''$ is a~colimit diagram of~$p''|_{K}\colon$ $K\to \MM$,
and the image of~the cone point under $p''$
is naturally equivalent to the image of~the cone point under $p'$.
Since $L\subset K$ is cofinal, we have a~canonical equivalence
$\operatorname{colim} p''|_{K}\simeq \operatorname{colim} p''|_{L}$.

Indeed, $\operatorname{colim} p''|_{K}$ is equivalent to
$B\otimes M\otimes S^1$
 (this computation is not necessary to
the proof~so that the reader may skip this paragraph, but it may be helpful to get feeling for the operadic left Kan extension $F'$).
By construction, the composite
$L\hookrightarrow K\stackrel{p}{\to}\MM^\otimes$
is equivalent to a~contant diagram so that $p''|_L\colon L\to \MM$ is
a contant diagram taking the value $B\otimes M$.
Note that $\Map_K(\mathsf{Z},\mathsf{Z})=\ZZ$
and there is a~categorical equivalence $L\simeq B\ZZ\simeq S^1$.
We~deduce that $\operatorname{colim} p''|_{L}\simeq (B\otimes M)\otimes S^1$.
Namely, the evaluation $F'(A,B,M)(C_M)$ of~$F'(A,B,M)$ at $C_M$ is
$\operatorname{colim} p''|_{K}\simeq \operatorname{colim} p''|_{L}\simeq (B\otimes M)\otimes S^1$.
Another way to compute it is as follows.
By Corollary~\ref{equivariant}, we have $\Alg_{\cyl}(\MM)\simeq \Alg_{\assoc}\big(\Fun\big(BS^1,\MM\big)\big)$
and $\Alg_{\cylm}(\MM)\simeq \LMod\big(\Fun\big(BS^1,\MM\big)\big)$.
These equivalences commute with forgetful functors arising from
the inclusions $\assoc\to \LM$ and $\cyl\to \cylm$.
The adjunction $(k_!,k^*)$ can be identified with the composite of~adjunctions
\begin{gather*}
\Alg_{\assoc}\big(\Fun\big(BS^1,\MM\big)\big)\times \MM\rightleftarrows \Alg_{\assoc}\big(\Fun\big(BS^1,\MM\big)\big)\times \Fun\big(BS^1,\MM\big)
\\ \hphantom{\Alg_{\assoc}\big(\Fun\big(BS^1,\MM\big)\big)\times \MM}
{}\rightleftarrows \LMod\big(\Fun\big(BS^1,\MM\big)\big),
\end{gather*}
where the left adjunction is induced by~the adjunction $\MM\rightleftarrows \Fun\big(BS^1,\MM\big)$ which consists of~the forgetful functor
$\Fun\big(BS^1,\MM\big)\to \MM$ and the left adjoint free functor which
sends $M$ to~$S^1\otimes M$.
The right adjoint in the right adjunction is given
by~the evaluation of~the module objects
$\LMod\big(\Fun\big(BS^1,\MM\big)\big)\!\to\! \Fun\big(BS^1,\MM\big)$
and
$\LMod\big(\Fun\big(BS^1,\MM\big)\big)\!\to\! \Alg_{\assoc}\big(\Fun\big(BS^1,\MM\big)\big)$
induced by~$\assoc\hookrightarrow \LM$.
The left adjoint functors carry $(B,M)$ to
\begin{gather*}
\big(B,B\otimes S^1\otimes M\big)\in \LMod\big(\Fun\big(BS^1,\MM\big)\big).
\end{gather*}

Next we will consider $F(A,B,M)$.
Let $r\colon\dcyl\boxplus \textup{Triv}^\otimes \to \dcylm$ be a~morphism
of~$\infty$-ope\-rads induced by~$\dcyl \hookrightarrow \dcylm$
and $\triv^\otimes\to \dcylm$ determined by~$C_M\in \dcylm_{\langle1\rangle}$
correpondings to the color $C_M$ (we slightly abuse notation again).
By \cite[Corollary~3.1.3.5]{HA},
we have an~adjunction $r_!\colon\Alg_{\dcyl}(\MM)\times \MM\rightleftarrows \Alg_{\dcylm}(\MM){\colon}\!r^*$. This shows that there exists a~left adjoint $F=r_!$ of~$G$.
Consider
\begin{gather*}
e\colon\ P:=\dcyl\boxplus \textup{Triv}^\otimes\times_{\dcylm}\dcylm^{\textup{act}}_{/C_M} \stackrel{\textup{pr}_1}{\longrightarrow} \dcyl\boxplus \textup{Triv}^\otimes \stackrel{(A,B,M)}{\longrightarrow} \MM^\otimes,
\end{gather*}
where the right arrow classifies the algebra object $(A,B,M)\in \Alg_{\dcyl}(\MM)\times \MM$.
The evaluation of~$r_!(A,B,M)$ at $C_M$ is a~colimit of~$e$, that is, the image of~the cone point under an~operadic $q$-colimit
diagram $e'\colon P^{\triangleright}\to \MM^\otimes$ that extends $e$.
The computation of~the colimit of~$e$ is similar to that of~$p$.
We~infomally denote by~$(\langle n \rangle, S,T,D^{d},C^{c},\star)$
an object of~$\dcyl\boxplus \textup{Triv}^\otimes$,
where $\star\in \triv$, and $(D^{d},C^{c})$
 indicates the sequence of~colors which consists of~$d$ $D$'s
and $c$ $C$'s which we regard as
an object in $\dcyl_{\langle n-1\rangle}$ ($d+c=n-1$).
By abuse of~notation, we write
\begin{gather*}
R=\big(\langle n \rangle, S,T,D^{d},C^{c},\star, (D^d,C^c,C_M)\stackrel{f}{\to}C_M\big)
\end{gather*}
for an~object of~$P$, where $f\colon\big(D^d,C^c,C_M\big)\to C_M$ is a~morphism
in $\dcylm^{\textup{act}}$ that lies over the active morphism $\langle n\rangle \to \langle 1\rangle$.
We~compute the mapping space from $R$ to another object
\begin{gather*}
R'=\big(\langle m \rangle, S',T',D^{d'},C^{c'},\star, (D^{d'},C^{c'},C_M)\stackrel{g}{\to}C_M\big).
\end{gather*}
Given a~morphism $\phi\colon R\to R'$ we have the induced morphism
$\big(D^d,C^c,C_M\big)\to \big(D^{d'},C^{c'},C_M\big)$ in $\dcylm^{\textup{act}}_{/C_M}$.
Notice that it is given by~the union of~$\big(D^d,C^c\big)\to \big(D^{d'},C^{c'}\big)$
and~$C_M\to C_M$ over~$C_M$. Moreover, we can think of~$\big(D^d,C^c\big)\to \big(D^{d'},C^{c'}\big)$ as
a rectilinear
embedding $\big((0,1)^2\big)^{\sqcup d}\sqcup \big((0,1)\times S^1\big)^{\sqcup c}\to
\big((0,1)^2\big)^{\sqcup d'}\sqcup \big((0,1)\times S^1\big)^{\sqcup c'}$
over $(0,1)\times S^1$. In~this way,
 we obtain the induced morphism
\begin{gather*}
\Map_P(R,R')\to \Map_{\dcyl^{\textup{act}}_{/C}}\big(\big(D^d,C^c\big),\big(D^{d'},C^{c'}\big)\big).
\end{gather*}
As in the case of~$K$,
the restriction to $C_M$ gives rise to a~morphism
\begin{gather*}
\Map_P(R,R')\to \Map_{\triv}(\star,\star)\times_{\Map_{\dcylm}(C_M,C_M)}\Map_{\dcylm^{\textup{act}}_{/C_M}}(C_M,C_M) \simeq *\times_{S^1}*=\ZZ.
\end{gather*}
It~gives rise to an~equivalence in $\SSS$:
\begin{gather*}
\Map_P(R,R')\to \Map_{\dcyl^{\textup{act}}_{/C}}\big(\big(D^d,C^c\big),\big(D^{d'},C^{c'}\big)\big)\times \ZZ.
\end{gather*}
Let $Q\subset P$ be the full subcategory spanned
by~$\mathsf{Z}$
which we think of~as an~object of~$P$
in the obvious way.
As in the case of~$K$,
using the above description of~$\Map_P(R,R')$ we see that
for any $\mathsf{V}\in P$, $Q\times_{P}P_{\mathsf{V}/}$ is weakly contractible
so that
$Q\subset P$ is
cofinal.
Let $e'\colon P^{\triangleright}\to \MM^\otimes$ be an~operadic
$q$-colimit diagram that extends $e$.
Let $e''\colon P^{\triangleright}\to \MM=\MM_{\langle 1\rangle}$ be the diagram obtained by~a $q$-coCartesian natural transformation from $e'$.
Then the image of~the cone point under $e'$ is
$\colim e''|_P \simeq \colim e''|_{Q}$.
(We~can also deduce that $r_!(A,B,M)(C_M)=F(A,B,M)(C_M)$
is $B\otimes M\otimes S^1\in \MM$ in the same way as described above.)

The projection of~$G'F'(A,B,M)\to GF(A,B,M)$ to $\MM$
is the canonical map $\colim p''|_K \to \colim e''|_{P}$.
We~note that
the canonical functor $K\to P$ induces an~equivalence
$L\stackrel{\sim}{\to} Q$.
Consequently, $\colim p''|_K \to \colim e''|_{P}$
can be identified with the equivalence
$\colim p''|_L \stackrel{\sim}{\to} \colim e''|_{Q}$.
Next we consider the projection of~$G'F'(A,B,M)\to GF(A,B,M)$
to $\Alg_{\dcyl}(\MM)$.
Taking into account
the equivalences $\dcyl\boxplus \textup{Triv}^\otimes\times_{\dcylm}\dcylm^{\textup{act}}_{/D}\simeq \dcyl\times_{\dcylm}\dcylm^{\textup{act}}_{/D}$
and $\dcyl\boxplus \textup{Triv}^\otimes\times_{\dcylm}\dcylm^{\textup{act}}_{/C}\simeq \dcyl\times_{\dcylm}\dcylm^{\textup{act}}_{/C}$
and the presentation of~$F$ in terms of~operadic colimits,
we see that
the evaluations of~unit maps
$A\to GF(A,B,M)(D)$ and $B\to GF(A,B,M)(C)$ are equivalence
so that
the evaluations $A\simeq G'F'(A,B,M)(D)\to GF(A,B,M)(D)$
and $B\simeq G'F'(A,B,M)(C)\to GF(A,B,M)(C)$ of~$G'F'(A,B,M)\to GF(A,B,M)$
are also
equivalences.
Consequently, $G'F'(A,B,M)\to GF(A,B,M)$ is an~equivalence since
evaluations at $D$, $C$, and the projection to $\MM$ are equivalences.
This proves (v). We~also have proved the existence of~$F$ and $F'$, that is, (i).
The final assertion is clear.
\end{proof}

\begin{proof}[Proof~of~Proposition~\ref{another2}]
The inclusion $j\colon \KS\to \dcylm$ induces an~adjunction
\begin{gather*}
j_!\colon\ \Alg_{\KS}(\MM)\rightleftarrows \Alg_{\dcylm}(\MM)\ {\colon}\!j^*.
\end{gather*}
Since $j_!$ is fully faithful, it is enough to prove that
the essential image of~$j_!$ is $\Alg_{\dcylm}^{\mathbf{D}}(\MM)$.
Let $X=(A,B,M)\colon \dcylm\to \MM^\otimes $ be an~object of~$\Alg_{\dcylm}(\MM)$
whose evaluations at $D$, $C$ and $C_M$ are $A$, $B$ and $M$, respectively.
Note that the forgetful functor
$\Alg_{\dcylm}(\MM) \to \MM\times \MM\times \MM$
induced by~evaluations at $D$, $C$, and $C_M$
is conservative.
Thus if we write $q\colon \MM^\otimes \to \Gamma$ for the structure
map, $X$ belongs to $\Alg_{\dcylm}^{\mathbf{D}}(\MM)$
if and only if
\begin{gather*}
p'\colon\ K^{\triangleright}:=\big(\etwo^\otimes\times_{\dcyl}\dcyl^{\textup{act}}_{/C}\big)^{\triangleright}\to \dcylm\to \MM^\otimes
\end{gather*}
is an~operadic $q$-colimit diagram that extends $p=p'|_{K}$, see \cite[Definition~3.1.2.1]{HA}.
Let $Y\colon \KS\to \MM^\otimes$ be an~object of~$\Alg_{\KS}(\MM)$.
Let us consider the image of~$C$ under
$j_!(Y)\colon \dcylm\to \MM^\otimes$.
It~is an~operadic $q$-colimit of~\begin{gather*}
e\colon\ L:=\KS\times_{\dcylm}\dcylm^{\textup{act}}_{/C}\to \KS\to \MM^\otimes.
\end{gather*}
In~view of~\cite[Propositions~3.1.1.15 and 3.1.1.16]{HA},
to prove the image of~$j_!$ is contained in $\Alg_{\dcylm}^{\mathbf{D}}(\MM)$,
it is enough to observe that the natural functor
$\etwo^\otimes\times_{\dcyl}\dcyl^{\textup{act}}_{/C}\to \KS\times_{\dcylm}\dcylm^{\textup{act}}_{/C}$ is a~categorical equivalence.
It~follows from the fact that
 if
a sequence $E=(D,\dots,D$, $C,\dots,C,C_M,\dots C_M)$ (regarded as an~object
of~$\dcylm$) contains $C_M$, then
there is no morphism from $E$ to $C$ in $\dcylm^{\textup{act}}$.
Consequently, we have a~new adjunction
$j_!\colon \Alg_{\KS}(\MM)\rightleftarrows \Alg_{\dcylm}^{\mathbf{D}}(\MM)\, {\colon}\!j^*$.
By the comparison of~operadic $q$-colimit diagrams
from $K^{\triangleright}$ and $L^{\triangleright}$, the counit map of~this adjunction is an~equivalence.
Therefore, we obtain
a categorical equivalence $\Alg_{\KS}(\MM)\simeq \Alg_{\dcylm}^{\mathbf{D}}(\MM)$ induced by~$j_!$ (or $j^*$).
\end{proof}

\section{Hoch\-schild cohomology}
\label{cohomologysec}

In~this section, we recall Hoch\-schild cohomology spectra of~stable $\infty$-categories $\CCC$.
The definition is based on the principle that,
under a~suitable condition on $\CCC$,
Hoch\-schild cohomology of~$\CCC$ is the endomorphism
algebra of~the identity functor $\CCC\to \CCC$.
Moreover, since $\Fun(\CCC,\CCC)$ has the monoidal structure
given by~the composition,
Hoch\-schild cohomology is the endomorphism algebra
of~the unit object of~$\Fun(\CCC,\CCC)$
so that it comes equipped with the structure
of~an $\etwo$-algebra, cf.~\cite{Bat, KT}
(see also references cited in {\it loc.~cit.} for Deligne conjecture concerning Hoch\-schild cochains).
We~establish some notation.
Let $R$ be a~commutative ring spectrum.
Let $\Mod_R^\otimes$ be the symmetric monoidal $\infty$-category
of~$R$-module spectra whose underlying category we denote by~$\Mod_R$.
Let $\Alg_{\assoc}(\Mod_R)$ be the $\infty$-category
of~the associative algebra objects
in $\Mod_R$.
Let $\PR$ be the $\infty$-category of~presentable $\infty$-categories
whose morphisms are those functors that preserve small colimits.
This category $\PR$ admits a~symmetric monoidal structure,
see \cite[Notation~4.8.1.7 and Proposition~4.8.1.15]{HA}.
The $\infty$-category of~small spaces $\SSS$ is a~unit object in $\PR$.
For $\DDD, \DDD'\in \PR$,
the tensor product $\DDD\otimes \DDD'$
comes equipped with a~functor $\DDD\times \DDD'\to \DDD\otimes\DDD'$
which preserves small colimits separately in each variable and satisfies
the following universal property: for any $\FF\in \PR$, the composition induces a~fully faithful functor
\begin{gather*}
\LFun(\DDD\otimes \DDD',\FF)\to \Fun(\DDD\times \DDD',\FF)
\end{gather*}
whose essential image is spanned by~those functors $\DDD\times \DDD'\to \FF$
which preserves small colimits separately in each variable,
where $\LFun(-,-)$ indicates the full subcategory of~$\Fun(-,-)$
spanned by~those functors which preserves small colimits.
The underlying associative monoidal $\infty$-category $\Mod_R^\otimes$
can be regarded as an~associative algebra object in
$\PR$ since $\Mod_R$ is presentable and the
tensor product functor $\Mod_R\times \Mod_R\to \Mod_R$
preserves small colimits separately in each variable.
We~denote by~$\LMod_{\Mod_R^\otimes}\left(\PR\right)$ the $\infty$-category
of~left $\Mod_R^\otimes$-module \mbox{objects} in $\PR$.
Given $A\in \Alg_{\assoc}(\Mod_R)$, we write $\RMod_A:=\RMod_A(\Mod_R)$
for the \mbox{$\infty$-cate\-gory} of~right module objects of~$A$ in $\Mod_R$.
We~remark that the forgetful lax symmetric monoidal functor
$\Mod_R\to \SP$ induces $\Mod_A(\Mod_R)\to \Mod_A(\SP)$,
where we use the same notation $A$ to indicate the image in
$\Alg_{\assoc}(\SP)$.
The functor $\Mod_A(\Mod_R)\to \Mod_A(\SP)$ is an~equivalence
of~$\infty$-categories (for example, apply Lurie--Barr--Beck theorem
to this functor endowed with projections to $\SP$) so
 that the notation $\RMod_A$ is consistent with that of~\cite{HA}.
The category $\RMod_A$ has a~natural left module structure
$\Mod_R\times \RMod_A\to \RMod_A$ informally given by~$(M,N)\mapsto M\otimes_R N$. In~what follows, when we treat the tensor product of~objects in~$\Mod_R$
(over $R$), we write $\otimes$ for $\otimes_R$.
The assignment $A\mapsto \RMod_A$ gives rise to a~functor
\begin{gather*}
\Alg_{\assoc}(\Mod_R)\to \LMod_{\Mod_R^\otimes}\left(\PR\right)
\end{gather*}
that sends $A$ to $\RMod_A$ and carries a~morphism $f\colon A\to B$ to
the base change functor $\RMod_A\to \RMod_B;\ N\mapsto N\otimes_AB$, that is,
a left adjoint of~the forgetful functor $\RMod_B\to \RMod_A$,
see \cite[Section~4.8.3, Notation~4.8.5.10 and Theorem~4.8.5.11]{HA}.
We~have the induced functor
\begin{gather*}
I\colon\ \Alg_{\assoc}(\Mod_R)\simeq \Alg_{\assoc}(\Mod_R)_{R/} \to \LMod_{\Mod_R^\otimes}\left(\PR\right)_{\Mod_R/},
\end{gather*}
which sends $A$ to the base change functor
$\Mod_R=\RMod_R\to \RMod_A$.
The functor $I$ is fully faithful and admits
a right adjoint $E$.
A morphism $f\colon \Mod_R\to \DDD$ in $\LMod_{\Mod_R^\otimes}\left(\PR\right)$
is determined by~the image $f(R)$ of~$R\in \Mod_R$
in an~essentially unique way (up to a~contractible space of~choices).
Therefore, an~object of~$\LMod_{\Mod_R^\otimes}\left(\PR\right)_{\Mod_R/}$
is regarded as a~pair $(\DDD,D)$
such that $\DDD$ belongs to $\LMod_{\Mod_R^\otimes}\left(\PR\right)$
and $D$ is an~object of~$\DDD$.
The essential image of~$I$ can naturally be identified with
$\Alg_{\assoc}(\Mod_R)$. Namely, it
consists of~pairs of~the form $(\RMod_A,A)$:
$I$ carries $A$ to $(\RMod_A,A)$.
Put another way, the essential image is spanned by~pairs $(\DDD,D)$ such that $\DDD$ is a~compactly generated stable
$\infty$-category equipped with a~single compact generator $D$.
The right adjoint $E$ sends $(\DDD,D)$
to an~endormorphism algebra object $\textup{End}(D) \in \Alg_{\assoc}(\Mod_R)$
\cite[Theorem~4.8.5.11]{HA}.
Since the left adjoint $I$ is fully faithful,
the unit map $\textup{id}\to E\circ I$
is a~natural equivalence. Namely, the adjunction $(I,E)$
is a~colocalization.
If~we denote
$\mathcal{A}\subset \LMod_{\Mod_R^\otimes}\left(\PR\right)_{\Mod_R/}$
by~the essential image of~$I$, then
$E$ induces a~categorical equivalence
$\mathcal{A}\stackrel{\sim}{\to} \Alg_{\assoc}(\Mod_R)$.

The functor $I$ is extended to a~symmetric monoidal functor.
To explain this, note that $\Alg_{\assoc}(\Mod_R)$ comes equipped with a~symmetric monoidal structure induced by~that of~$\Mod_R^\otimes$,
see \cite[Section~3.2.4]{HA} or Construction~\ref{symfunctorsym}.
Since $\Mod_R^\otimes$ is a~symmetric monoidal $\infty$-category
such that $\Mod_R$ has small colimits and the tensor product functor
$\Mod_R\times \Mod_R\to \Mod_R$ preserves small colimits separately in each
variable,
we define $\Mod^\otimes_{\Mod_R^\otimes}\left(\PR\right)$
to be the symmetric monoidal $\infty$-category of~$\Mod_R^\otimes$-module objects in $\PR$, cf.~\cite[Section~3.3.3]{HA}.
If~we denote the underlying $\infty$-category
by~$\Mod_{\Mod_R^\otimes}\left(\PR\right)$, then
$\Mod_{\Mod_R^\otimes}\left(\PR\right)\simeq \LMod_{\Mod_R^\otimes}\left(\PR\right)$.
Hence $\LMod_{\Mod_R^\otimes}\left(\PR\right)_{\Mod_R/}\simeq \Mod_{\Mod_R^\otimes}\left(\PR\right)_{\Mod_R/} \simeq \Alg_{\mathbf{E}_0^{\otimes}}(\Mod_{\Mod_R^\otimes}\left(\PR\right))$ inherits a~symmetric monoidal structure.
In~summary, we have the adjunction
\begin{gather*}
I\colon\ \Alg_{\assoc}(\Mod_R)\rightleftarrows \LMod_{\Mod_R^\otimes}\left(\PR\right)_{\Mod_R/}\ {\colon}\!E
\end{gather*}
whose left adjoint is symmetric monoidal and fully faithful, and whose right adjoint is
lax symmetric monoidal.
It~gives rise to an~adjunction
\begin{gather*}
\begin{array}{ll}
I\colon\ \Alg_{\assoc}(\Alg_{\assoc}(\Mod_R))\!\!&\rightleftarrows \Alg_{\assoc}\big(\Mod_{\Mod_R^\otimes}\big(\PR\big)_{\Mod_R/}\big)\vspace{.5ex}
\\ \hphantom{I\colon \Alg_{\assoc}(\Alg_{\assoc}(\Mod_R))}
&\simeq \Alg_{\assoc}\big(\Mod_{\Mod_R^\otimes}\big(\PR\big)\big)\ {\colon}\!E,
\end{array}
\end{gather*}
where we abuse notation by~writing $(I,E)$ for the
induced adjunction.
By virtue of~the cano\-nical equivalence $\eone^\otimes\simeq \assoc^\otimes$
and the $\infty$-operad version of~Dunn additivity theorem \cite[Theorem~5.1.2.2]{HA},
we have a~canonical equivalence
\begin{gather*}
\Alg_{\etwo}(\Mod_R)\simeq \Alg_{\assoc}(\Alg_{\assoc}(\Mod_R))
\end{gather*}
(we can also use additivity theorem
to the equivalence on the right-hand side).

We~refer to an~object of~$\PR_R:=\Mod_{\Mod_R^\otimes}\left(\PR\right)$ as
an $R$-linear presentable $\infty$-category.
Note that the underlying
$\infty$-category of~an $R$-linear presentable $\infty$-cartegory
is stable.

\begin{Lemma}
\label{linearfunctorcategory}
Let $\DDD$ be an~$R$-linear presentable $\infty$-category.
Let $\otimes_R\colon \PR_R\times \PR_R\to \PR_R$ be the tensor product functor.
There exists a~morphism object from $\DDD$ to itself $($i.e., an~internal hom object$)$
$\operatorname{\mathcal{M}{\rm or}}_R(\DDD,\DDD)\in \PR_R$
equipped with $e\colon \operatorname{\mathcal{M}{\rm or}}_R(\DDD,\DDD)\otimes_R \DDD\to \DDD$
for $\DDD$.
Moreover, $(\operatorname{\mathcal{M}{\rm or}}_R(\DDD,\DDD),e)$ is
promoted to an~object of~$\mathcal{E}\in \Alg_{\assoc}\left(\PR_R\right)$
together with a~left module action $\mathcal{E}\otimes_R \DDD\to \DDD$.
\end{Lemma}

\begin{proof}
According to~\cite[Corollaries~4.7.1.40 and 4.7.1.41]{HA},
the second assertion follows from the first assertion.
We~will show the existence of~a~morphism object $\operatorname{\mathcal{M}{\rm or}}_R(\DDD,\DDD)$.
Recall that a~morphism object for $\DDD$ and $\DDD'$
is an~$R$-linear presentable
$\infty$-category $\CCC$ together with
a~morphism $\CCC\otimes_R \DDD\to \DDD'$ such that the composition
induces
an equivalence
\begin{gather*}
\Map_{\PR_R}(\FF,\CCC)\simeq \Map_{\PR_R}(\FF\otimes_R \DDD,\DDD')
\end{gather*}
for each $\FF\in \PR_R$,
which informally carries $\FF\to \CCC$ to $\FF\otimes_R \DDD \to \CCC\otimes_R\DDD\to \DDD$.
We~first consider the case where $R$ is the sphere spectrum $\SSSS$.
Let $\LFun(\DDD,\DDD')$ be the full subcategory of~$\Fun(\DDD,\DDD')$
that consists of~colimit-preserving functors.
Then $\LFun(\DDD,\DDD')$ together with the
evaluation functor
$\LFun(\DDD,\DDD')\times \DDD\to \DDD'$ exhibits
$\LFun(\DDD,\DDD')$ as an~internal hom object.
Thus we have a~morphism object $\LFun(\DDD,\DDD')$.
Next, we consider the general case.
Let $F\colon \PR\rightleftarrows \Mod_{\Mod_R^\otimes}\left(\PR\right)\,{\colon}\!U$ be
an adjunction which consists of~the forgetful functor $U$
and the free functor $F$ given informally by~$\CCC\mapsto \CCC\otimes \Mod_R$.
Here $\otimes$ indicates the tensor product in $\PR$.
If~we suppose that $\DDD$ is a~free object, i.e. $\DDD=F(\CCC)=\CCC\otimes \Mod_R$, then there is a~morphism object for $\DDD$ and $\DDD'$.
Indeed, we observe that
$\operatorname{\mathcal{M}{\rm or}}_R(\DDD,\DDD')=\LFun(\CCC,\DDD')$ together with
\begin{gather*}
\LFun(\CCC,\DDD')\otimes_R(\CCC\otimes\Mod_R) \simeq \LFun(\CCC,\DDD') \otimes\CCC \to \DDD'
\end{gather*}
constitutes a~morphism object for $\DDD$ and $\DDD'$.
To prove that
\begin{gather*}
\theta\colon\ \Map_{\PR_R}\big(\PPP,\LFun\big(\CCC,\DDD'\big)\big)\to \Map_{\PR_R}(\PPP\otimes_R(\CCC\otimes \Mod_R),\DDD')
\end{gather*}
is an~equivalence,
we may and will assume that $\PPP$ is a~free object
since the
tensor operation functor $\otimes_R$ preserves small colimits separately in
each variable (see the proof~of~\cite[Proposition~5.1.2.9]{HA}), and
$\PPP$ is a~(small) colimit of~the diagram of~free objects:
for example, using the adjunction $(F,U)$ we have a~simplicial diagram
of~free objects whose colimit is $\PPP$.
When $\PPP=\CCC'\otimes \Mod_R$,
by~the adjunction we see that $\theta$ is an~equivalence.
We~put $\DDD=\colim_{i\in I}\DDD_i$, where each $\DDD_i$ is a~free object.
Then for any $\PPP\in \PR_R$ there exist natural equivalences
\begin{gather*}
\Map_{\PR_R}(\PPP\otimes_R(\colim_{i\in I} \DDD_i),\DDD') \simeq \Map_{\PR_R}(\colim_{i\in I}(\PPP\otimes_R\DDD_i),\DDD')
\\ \hphantom{\Map_{\PR_R}(\PPP\otimes_R(\colim_{i\in I} \DDD_i),\DDD')}
{}\simeq \lim_{i\in I}\Map_{\PR_R}(\PPP\otimes_R\DDD_i,\DDD')
\\ \hphantom{\Map_{\PR_R}(\PPP\otimes_R(\colim_{i\in I} \DDD_i),\DDD')}
{}\simeq \lim_{i\in I}\Map_{\PR_R}(\PPP,\operatorname{\mathcal{M}{\rm or}}_R(\DDD_i,\DDD'))
\\ \hphantom{\Map_{\PR_R}(\PPP\otimes_R(\colim_{i\in I} \DDD_i),\DDD')}
{}\simeq \Map_{\PR_R}(\PPP,\lim_{i\in I}\operatorname{\mathcal{M}{\rm or}}_R(\DDD_i,\DDD')).
\end{gather*}
Hence there exists a~morphism object $\operatorname{\mathcal{M}{\rm or}}_R(\DDD,\DDD')$, that is, $\lim_{i\in I}\operatorname{\mathcal{M}{\rm or}}_R(\DDD_i,\DDD')$.
\end{proof}

\begin{Remark}
Let $\operatorname{\mathcal{M}{\rm or}}_R(\DDD,\DDD')^{\simeq}$
be the largest Kan subcomplex
of~the underlying $\infty$-category of~$\operatorname{\mathcal{M}{\rm or}}_R(\DDD,\DDD')$, which we regarded as an~object in $\SSS$. By the above proof,
$\operatorname{\mathcal{M}{\rm or}}_R(\DDD,\DDD')^{\simeq}$ is~equivalent to the mapping space $\Map_{\PR_R}(\DDD,\DDD')$.
\end{Remark}

We~shall write $\operatorname{\mathcal{E}{\rm nd}}_R(\DDD)$ for $\mathcal{E} \in \Alg_{\assoc}\left(\PR_R\right)$.

\begin{Definition}
\label{cohomologydef}
Let $\DDD$ be an~$R$-linear presentable $\infty$-category.
Applying $E\colon \Alg_{\assoc}\left(\PR_R\right)\to \Alg_{\etwo}(\Mod_R)$,
we define the Hoch\-schild cohomology $R$-module spectrum of~$\DDD$
to be
\begin{gather*}
\HH^\bullet_R(\DDD):=E(\operatorname{\mathcal{E}{\rm nd}}_R(\DDD))
\end{gather*}
in $\Alg_{\etwo}(\Mod_R)$.
We~often abuse notation by~identifying $\HH^\bullet_R(\DDD)$
with its image in $\Mod_R$.
If~no confusion can arise, we write $\HH^\bullet(\DDD)$ for $\HH^\bullet_R(\DDD)$.
\end{Definition}

Let $\ST$ be the $\infty$-category of~small
stable idempotent-complete $\infty$-categories
whose morphisms are exact functors.
Let $\CCC$ be a~small
stable idempotent-complete $\infty$-category
and let
$\Ind(\CCC)$ denote the $\infty$-category of~Ind-objects.
Then $\Ind(\CCC)$ is a~compactly generated stable $\infty$-category.
The inclusion $\CCC\to \Ind(\CCC)$ identifies
the essential image with the full subcategory $\Ind(\CCC)^\omega$ spanned by~compact obejcts in $\Ind(\CCC)$.
Given $\CCC,\CCC'\in \ST$, if
we write $\Fun^{\textup{ex}}(\CCC,\CCC')$ for the full subcategory
spanned by~exact functors,
the left Kan extension \cite[Proposition~5.3.5.10]{HTT} gives rise to
a fully faithful functor $\Fun^{\textup{ex}}(\CCC,\CCC')\to \Fun^{\textup{L}}(\Ind(\CCC),\Ind(\CCC'))$ whose essential image consists of~those functors that carry $\CCC$ to $\CCC'$.
We~set $\PR_{\textup{St}}=\Mod_{\SP^\otimes}\left(\PR\right)$,
which can be regarded as the full subcategory of~$\PR$
that consists of~stable presentable $\infty$-categories.
The assignment $\CCC\mapsto \Ind(\CCC)$
identifies $\ST$ with the subcategory
of~$\PR_{\textup{St}}$ whose objects are
compactly generated stable $\infty$-categories,
and whose morphisms are those functors that preserve compact objects.
The $\infty$-category $\ST$ inherits a~symmetric monoidal structure
from the structure on $\PR_{\textup{St}}$.
The stable $\infty$-category of~compact spectra is a~unit object in $\ST$.
Given two objects $\CCC$ and $\CCC'$ of~$\ST$, the tensor product
$\CCC\otimes \CCC'$ is naturally equivalent to
the full subcategory
$(\Ind(\CCC)\otimes \Ind(\CCC'))^{\omega}\subset \Ind(\CCC)\otimes \Ind(\CCC')$ spanned by~compact objects.
If~we let $\textup{Cgt}_{\textup{St}}^{\textup{L}}$ denote the full subcategory of~$\PR_{\textup{St}}$ spanned by~compactly generated stable
$\infty$-categories, then we have a~sequence
\begin{gather*}
\ST\to \textup{Cgt}_{\textup{St}}^{\textup{L}} \subset \PR_{\textup{St}},
\end{gather*}
where $\textup{Cgt}_{\textup{St}}^{\textup{L}} \subset \PR_{\textup{St}}$ is closed under
the tensor product so that
$\textup{Cgt}_{\textup{St}}^{\textup{L}}$ inherits a~symmetric monoidal structure from the
structure on $\PR_{\textup{St}}$,
and the left arrow is a~symmetric monoidal faithful functor given by~$\CCC\mapsto \Ind(\CCC)$.
In~$\textup{Cgt}_{\textup{St}}^{\textup{L}}$, every object is dualizable.
For more details, we~refer the readers to~\cite[Section~3]{BGT1}, \cite[Section~4.8]{HA}.

Consider $\RMod_A$ for $A\in \Alg_{\assoc}(\Mod_R)$.
We~let $\RPerf_{A}$ be the full subcategory
of~$\RMod_{A}$ spanned by~compact objects.
This subcategory is the smallest stable subcategory which
contains $A$ (regarded as a~right module)
and is closed under retracts.
When~$A$ belongs to $\CAlg(\Mod_R)$,
we write $\Perf_A$ for $\RPerf_{A}$.
In~this case, $\Perf_A$ is closed under taking tensor product
so that it inherits a~symmetric monoidal structure from
that of~$\Mod_A^\otimes$.
We~usually regard the symmetric monoidal $\infty$-category
$\Perf^\otimes_R$ as an~object of~$\CAlg(\ST)$, and
we write $\ST_R$ for $\Mod_{\Perf_R^\otimes}(\ST)$.
We~refer to an~object of~$\ST_R$
as a~small $R$-linear stable $\infty$-category.
Since $\Mod_R^\otimes\simeq \Ind(\Perf_R)^\otimes$,
there is a~natural symmetric monoidal functor $\ST_R\to \PR_R=\Mod_{\Mod_R^\otimes}\left(\PR\right)$ which carries $\CCC$ to~$\Ind(\CCC)$.

\begin{Definition}\label{cohomologydefsmall}
Given $\CCC\in \ST_R$,
we define the Hoch\-schild cohomology $R$-module spectrum $\HH^\bullet_R(\CCC)$
to be
$\HH_R^\bullet(\Ind(\CCC))$.
If~no confusion can arise, we write $\HH^\bullet(\CCC)$ for $\HH^\bullet_R(\CCC)$.
\end{Definition}

\section{Hochschild homology}\label{homologysec}

Let $R$ be a~commutative ring spectrum.
Suppose that we are given a~small $R$-linear stable \mbox{$\infty$-cate\-gory}
$\CCC$.
In~this section, we assign to $\CCC\in \ST_R$ the Hoch\-schild homology $R$-module
spectrum $\HH_{\bullet}(\CCC)\in \Mod_R$.
For the main purpose of~this paper,
we require the following additional structures:

\begin{itemize}\itemsep=0pt
\item the $R$-module spectrum $\HH_{\bullet}(\CCC)$ has
an action of~the circle $S^1$.
Namely, $\HH_{\bullet}(\CCC)$ is promoted to an~object
of~$\Fun\big(BS^1,\Mod_R\big)$, and the assignment $\CCC \mapsto \HHH(\CCC)$
gives rise to a~functor $\ST_R\to \Fun\big(BS^1,\Mod_R\big)$,

\item if $\Fun\big(BS^1,\Mod_R\big)$ is equipped with
a pointwise symmetric monoidal strcuture induced by~that of~$\Mod_R$,
then the above functor $\ST_R\to \Fun\big(BS^1,\Mod_R\big)$
is promoted to
a~symmetric monoidal functor from $\ST_R$ to $\Fun\big(BS^1,\Mod_R\big)$.
\end{itemize}
To this end,
we will use enriched models of~stable idempotent-complete
$\infty$-categories, i.e., spectral categories.

{\it Symmetric spectra.}
We~give a~minimal review of~the theory of~symmetric spectra,
introduced and developed
in~\cite{HSS}.
This theory provides a~nice foundation of~the homotopy theory of~highly structured ring spectra as well as a~theoretical basis
for spectral categories.
We~let $\SPS$ be the closed symmetric monoidal category of~symmetric spectra.
We~write $\mathbb{S}$ for the unit object which we call the sphere spectrum.
We~use the notation slightly different from \cite{HSS, S}:
$\mathbb{S}$ is $S$ in~\cite{HSS}.
We~use a~symmetric monoidal proper combinatorial model category structure on
$\SPS$ satisfying the monoid axiom in the sense of~\cite[Definition 3.3]{SS},
in which a~weak equivalence is a~stable equivalence.
There are several versions of~such model structures.
We~here focus on two model structures.
One is described in~\cite[Theorems~3.4.4 and~5.4.2, Corollaries~5.3.8 and~5.5.2]{HSS}
which is called the stable model structure.
In~\cite{S}, it is proved that there is
another model structure called the stable $\mathbb{S}$-model structure.
The difference (relevant to us)
between stable model structure and stable $\mathbb{S}$-model structure
 is that
cofibrations in the stable model structure \cite[Theorem~3.4.4]{HSS}
are contained in the class of~cofibrations in
the stable $\mathbb{S}$-model structure while
both have the same class of~weak equivalences.
Let $\CAlg\big(\SPS\big)$ denote the category of~commutative algebra objects
in $\SPS$. We~refer to an~object of~$\CAlg\big(\SPS\big)$
as a~commutative symmetric ring spectrum.
The category $\CAlg\big(\SPS\big)$ admits a~model category structure:
we use the model structure on $\CAlg\big(\SPS\big)$, defined in~\cite[Theorem~3.2]{S} in which a~morphism is a~weak equivalence if the underlying morphism
in $\SPS$ is a~stable
equivalence.
The stable $\mathbb{S}$-model structure on $\SPS$ has
the following pleasant property:
if $\RR$ is a~cofibrant object in $\CAlg\big(\SPS\big)$, then
the underlying object $\RR$ in $\SPS$ is cofibrant
with respect to the stable $\mathbb{S}$-model structure,
see \cite[Section~4]{S}.

Let $\RR$ be a~commutative symmetric ring spectrum,
which we think of~as a~model of~$R\in \CAlg(\SP)$.
Unless otherwise stated, we assume that
$\RR$ is cofibrant in $\CAlg\big(\SPS\big)$.
We~let $\SPS(\RR)$ denote the category of~$\RR$-module objects
in $\SPS$, which is endowed with the natural symmetric monoidal
structure induced by~the structure on $\SPS$.
In~virtue of~\cite[Theorem~2.6]{S} (or \cite[Theorem~4.1]{SS}),
there is a~combinatorial symmetric monoidal
projective model structure on $\SPS(\RR)$ satisfying the monoid axiom,
in which a~morphism is a~weak equivalence (resp.~a fibration)
if the underlying morphism in $\SPS$ is a~stable equivalence
(resp.~a fibration with respect to stable $\mathbb{S}$-model structure).
We~refer to this model structure as the stable $\RR$-model structure.

\begin{Definition}
Let $\RR$ be a~commutative symmetric ring spectrum.
An $\RR$-spectrum category is a~category
enriched over $\SPS(\RR)$.
More explicitly,
a (small) $\RR$-spectrum category $\AAA$ consists of~the data:
\begin{itemize}\itemsep=0pt

\item A (small) set of~objects,

\item An $\RR$-module symmetric spectrum $\AAA(X,Y)\in \SPS(\RR)$ for each ordered pair of~objetcs $(X,Y)$,

\item The composition law $\AAA(Y,Z)\wedge_{\RR} \AAA(X,Y)\to \AAA(X,Z)$
satisfying the standard associativity axiom,

\item $\mathbb{S}\to \mathcal{A}(X,X)$ for each object $X$ that
satisfies the standard unit axiom.

\end{itemize}
Here $\wedge_{\RR}$ denotes the wedge product
over $\RR$, which defines the tensor product in $\SPS(\RR)$.
A~fun\-c\-tor of~$\RR$-spectral categories is an~enriched functor,
that is, a~functor as enriched categories.
We~refer to them as $\RR$-spectral functors.
We~write $\CAT_{\RR}$ for
the category of~$\RR$-spectral categories whose morphisms
are $\RR$-spectral functors.
We~refer to an~$\SSSS$-spectral category
(resp.~an $\SSSS$-spectral functor)
as a~spectral category (resp.~a spectral functor). We~write $\wedge$ for $\wedge_{\SSSS}$.
\end{Definition}

Thanks to works \cite[Corollary~2.4]{BGT1}, \cite[Theorem~1.1]{Muro}, \cite[Theorem~7.25]{Tab},
$\CAT_{\RR}$ admits a~combinatorial model structure
whose weak equivalences are Dwyer--Kan equivalences (DK-equivalences for short).
See, e.g., \cite[Definition 2.1]{BGT1} for DK-equivalences.

Let us recall the notion of~Morita equivalences in the context
of~spectral categories, see
\cite[Sections 2 and~4]{BGT1} for an~excellent account.
Let $\AAA$ be a~small spectral category and
let $\Fun_{\SSSS}\left(\AAA^{\rm op},\SPS\right)$ be the spectral
category of~spectral functors.
There is a~combinatorial spectral model structure
where the class of~weak equivalences (resp.~fibrations)
are objectwise stable equivalences
(resp.~objectwise fibrations with respect to the stable $\mathbb{S}$-model structure) \cite[Appendix]{SS2}.
The enriched Yoneda embedding
$\AAA\to \Fun_{\SSSS}\left(\AAA^{\rm op},\SPS\right)$
is contained in the full subcategory of~cofibrant objects.
If~we replace $\AAA$ by~a fibrant object in $\CAT_{\SSSS}$,
the embedding lands in
the full subcategory $\Fun_{\SSSS}\left(\AAA^{\rm op},\SPS\right)^{\rm cf}$
that consists of~cofibrant and fibrant objects.
Let $F\colon \AAA\to \BBB$ be a~spectral functor of~spectral categories.
Then we have a~Quillen adjunction
\begin{gather*}
F_!\colon\ \Fun_{\SSSS}\left(\AAA^{\rm op},\SPS\right)\rightleftarrows \Fun_{\SSSS}\left(\BBB^{\rm op},\SPS\right)\ {\colon}\!F^*,
\end{gather*}
where $F^*$ is determined by~the composition with
$\AAA^{\rm op}\to \BBB^{\rm op}$, see \cite[Appendix]{SS2}.
Let $\DDD(\AAA)$ be the homotopy category of~$\DDD^{\Sigma}(\AAA):=\Fun_{\SSSS}\left(\AAA^{\rm op},\SPS\right)^{\rm cf}$. It~constitutes a~triangulated category.
Let $\DDD_{\rm pe}(\AAA)$ be the smallest thick subcategory of~$\DDD(\AAA)$
that contains the image of~$\AAA$ under the Yoneda embedding.
The subscript ``pe" stands for ``perfect".
We~write $\DDD^{\Sigma}_{\rm pe}(\AAA)$ for the full subcategory
of~$\Fun_{\SSSS}\left(\AAA^{\rm op},\SPS\right)^{\rm cf}$ spanned by~those
objects that belong to $\DDD_{\rm pe}(\AAA)$.
If~$F\colon \AAA\to \BBB$ is a~spectral functor,
we have
the induced (left-derived) functor $\mathbb{L}F_!\colon \DDD(\AAA)\to \DDD(\BBB)$.
Since
$\mathbb{L}F_!$ is an~exact functor of~triangulated categories,
it follows that the restrcition of~$\mathbb{L}F_!$ induces
$\mathbb{L}F_!\colon \DDD_{\rm pe}(\AAA)\to \DDD_{\rm pe}(\BBB)$.

\begin{Definition}
We~say that a~spectral functor $F\colon \AAA\to \BBB$ is a~Morita equivalence
if the induced functor $\mathbb{L}F_!\colon \DDD_{\rm pe}(\AAA)\to \DDD_{\rm pe}(\BBB)$
is an~equivalence of~categories.
When $F\colon \AAA\to \BBB$ is an~$\RR$-spectral functor,
$F$ is said to be a~Morita equivalence
if $F$ is a~Morita equivalence as a~spectral functor.
\end{Definition}

Let $\RR$ be a~commutative symmetric ring spectrum.
Let us recall the tensor product of~$\RR$-spectral categories.
Suppose that we are given $\AAA, \BBB\in \CAT_{\RR}$.
The tensor product $\AAA\wedge_\RR\BBB$ is defined by~the following data:
\begin{itemize}\itemsep=0pt

\item The set of~objects of~$\AAA\wedge_{\RR}\BBB$ is the set of~pairs
$(A,B)$, where $A$ is an~object of~$\AAA$, and $B$ is an~object
of~$\BBB$,

\item $\AAA\wedge_\RR\BBB((a,b),(a',b'))=\AAA(a,a')\wedge_{\RR}\BBB(b,b')$
for $(a,b)$, $(a',b')\in \AAA\wedge_\RR\BBB$.

\end{itemize}
This tensor product determines a~symmetric monoidal structure
on $\CAT_{\RR}$.
A unit object is defined as follows:
Let $B\RR$ be the spectral category which has a~single object $\ast$
together with the morphism ring spectrum $B\RR(\ast,\ast)=\RR$.
The composition $\RR\wedge\RR\to \RR$ and the unit
$\mathbb{S}\to \RR$ are determined by~the algebra structure
on $\RR$ in an~obvious way. Clearly, $B\RR$
is a~unit object in $\CAT_{\RR}$.
Since $\RR$ is commutative, we can also think of~$B\RR$ as a~symmetric monoidal
spectral category. Namely, it is a~commutative algebra object
in the symmetric monoidal category $\CAT_{\SSSS}$.
Note that an~$\RR$-spectral category $\AAA$ is
regarded as a~$B\RR$-module in $\SPS$.
Namely, there is a~canonical equivalence of~categories
$\CAT_{\RR}\stackrel{\sim}{\rightarrow} \Mod_{B\RR}(\CAT_{\mathbb{S}})$,
where the target is the category of~$B\RR$-module objects
in $\CAT_{\mathbb{S}}$.

For technical reasons, we use the notion of~pointwise-cofibrant spectral categories, cf.~\cite[Sec\-tion~4]{BGT2}.
We~say that an~$\RR$-spectral category $\AAA$ is pointwise-cofibrant
if each morphism spectrum $\AAA(X,Y)$ is cofibrant in $\SPS(\RR)$
with respect to the stable $\RR$-model structure.
Using the same argument as that in the proof~in~\cite[Proposition~4.1]{BGT2},
we have:

\begin{Proposition}[\cite{BGT2}]\quad
\label{flatspectral}
\begin{enumerate}\itemsep=0pt
\renewcommand{\labelenumi}{(\roman{enumi})}

\item[{\rm (i)}] Every $\RR$-spectral category is functorially Morita equivalent to
a pointwise-cofibrant $\RR$-spec\-t\-ral category with the same objects.

\item[{\rm (ii)}] The subcategory of~pointwise-cofibrant $\RR$-spectral category is closed under the tensor pro\-duct.

\item[{\rm (iii)}] If~$\AAA$ is a~pointwise-cofibrant $\RR$-spectral category, the tensor operation $\AAA\wedge_\RR(-)$ preserves Morita equivalences and colimits.

\item[{\rm (iv)}] If~$\AAA$ and $\BBB$ are both pointwise-cofibrant $\RR$-spectral categories, then the $\AAA\wedge_{\RR}\BBB$ computes the derived tensor product.

\end{enumerate}
\end{Proposition}

We~denote by~$\CAT_{\SSSS}^{\rm pc}$ the category of~small pointwise-cofibrant
spectral categories.
By Proposition~\ref{flatspectral},
$\CAT_{\SSSS}^{\rm pc}$ admits a~symmetric monoidal structure
given by~tensor products, and the tensor products
preserves Morita equivalences in each variable.
Similarly, we denote by~$\CAT_{\RR}^{\rm pc}$ the category of~small
pointwise-cofibrant
$\RR$-spectral categories.

{\it Inverting morphisms.}
We~recall the notion of~$\infty$-categories obtained from an~$\infty$-category
endowed with a~set of~morphisms.
We~refer the readers to~\cite[Sections~1.3.4 and~4.1.3]{HA} for more details.
Let $\CCC$ be an~$\infty$-category.
Suppose that we are given a~set $S$ of~edges (morphisms)
(we~assume all equivalences are contained in $S$).
Then there exists
an $\infty$-category
$\CCC\big[S^{-1}\big]$ together with $\xi\colon \CCC\to \CCC\big[S^{-1}\big]$
such that for any $\infty$-category $\DDD$ the composition
induces a~fully faithful functor
\begin{gather*}
\Map\big(\CCC\big[S^{-1}\big],\DDD\big)\to \Map(\CCC,\DDD)
\end{gather*}
whose essential image consists of~those functors $F\colon \CCC\to\DDD$ which carry edges in $S$
to equivalences in $\DDD$.
We~shall refer to $\CCC\big[S^{-1}\big]$
as the $\infty$-category obtained from $\CCC$
by~inverting $S$. We~note that
$\CCC\big[S^{-1}\big]$ is generally not locally small even when $\CCC$ is so.
When $\CCC$ is an~ordinary category,
an explicit construction of~$\CCC\big[S^{-1}\big]$
is given by~the hammock localization \cite{DK}.
Let~$\CCC^\otimes$ be a~symmetric monoidal $\infty$-category.
Let~$S$ be a~set of~edges in~$\CCC$
such that all equivalences are contained in~$S$.
Assume that
for any object $C\in \CCC$ and any morphism $C_1\to C_2$ in $S$,
the induced morphisms $C\otimes C_1\to C\otimes C_2$
and $C_1\otimes C\to C_2\otimes C$
belong to $S$.
Then there exists a~symmetric monoidal $\infty$-category $\CCC\big[S^{-1}\big]^\otimes$
together with a~symmetric monoidal functor
$\tilde{\xi}\colon \CCC^\otimes\to \CCC\big[S^{-1}\big]^\otimes$ whose underlying functor is equivalent to $\xi$.
There is a~universal property: for any symmetric
monoidal $\infty$-category $\DDD^\otimes$
the composition induces a~fully faithful functor
$\Map^\otimes\big(\CCC\big[S^{-1}\big]^\otimes,\DDD^\otimes\big)\to \Map^\otimes\big(\CCC^\otimes,\DDD^\otimes\big)$
whose essential image consists of~those
functors $F\colon \CCC^\otimes\to\DDD^\otimes$ which carry morphisms in $S$
to equivalences in $\DDD$.
Here $\Map^\otimes(-,-)$ indicates the space of~symmetric monoidal functors.

\begin{Example}\label{spectralalgebra}
Let $\SPS(\RR)^c$ be the full subcategory that consists of~cofibrant objects. The ten\-sor product $\wedge_{\RR}$ given by~the wedge product over $\RR$ preserves cofibrant objects.
In~addition, if $C\in \SPS(\RR)^c$ and $f\colon C_1\to C_2$
is a~weak equivalence (i.e., stable equivalence), then
$C\wedge_{\RR}f$ is a~weak equivalence.
Consequently, we have the symmetric monoidal
$\infty$-category $\SPS(\RR)^c\big[W^{-1}\big]^\otimes$
obtained by~inverting weak equivalence.
In~the case of~$\RR=\mathbb{S}$,
by~the characterization of~$\SP^\otimes$ \cite[Corollary~4.8.2.19]{HA},
there is a~canonical (unique) symmetric monoidal equivalence
$\SPS(\mathbb{S})^c\big[W^{-1}\big]^\otimes\simeq \SP^\otimes$.
In~addition, if we denote $R$ by~the image of~$\RR$
in $\CAlg(\SP)$, then by~\cite[Theorem~4.3.3.17]{HA} there is a~canonical symmetric monoidal
equivalence
$\SPS(\RR)^c\big[W^{-1}\big]^\otimes$ $\stackrel{\sim}{\to} \Mod_R^\otimes(\SPS(\mathbb{S})^c[W^{-1}]^\otimes) \simeq \Mod_R^\otimes(\SP^\otimes)=\Mod_R^\otimes$.
Let $\Alg_{\assoc}\big(\SPS(\RR)^c\big)$ be the category of~associative
algebra objects in $\SPS(\RR)^c$, which is endowed with the symmetric
monoidal structure induced by~that of~$\SPS(\RR)^c$.
Then if $\Alg_{\assoc}\big(\SPS(\RR)^c\big)\big[W^{-1}\big]^\otimes$
denotes the associated symmetric monoidal
$\infty$-category
obtained by~inverting weak equivalences,
then
we have equivalences of~symmetric monoidal $\infty$-categories
\begin{gather*}
\Alg_{\assoc}\big(\SPS(\RR)^c\big)\big[W^{-1}\big]^\otimes\simeq
\Alg_{\assoc}^\otimes\big(\big(\SPS(\RR)^c\big)\big[W^{-1}\big]\big)\simeq \Alg_{\assoc}^\otimes(\Mod_R),
\end{gather*}
where the left equivalence follows from
the rectification result \cite[Theorem~4.1.8.4]{HA}.
In~particular, given an~associative algebra $A\in \Alg_{\assoc}(\Mod_R)$,
there is an~associative algebra $\mathbb{A}\in \Alg_{\assoc}\big(\SPS(\RR)^c\big)$
together with an~equivalence $\sigma\colon \mathbb{A}\simeq A$ in $\Alg_{\assoc}(\Mod_R)$. In~this case, we say that $\mathbb{A}$ (together with~$\sigma$) represents~$A$.
\end{Example}

\begin{Example}We~use the symbol $M$ to indicate the class of~Morita equivalences in
$\CAT_{\SSSS}^{\rm pc}$ or~$\CAT_{\RR}^{\rm pc}$.
By Proposition~\ref{flatspectral}, we can invert
the class of~Morita equivalences $M$ to obtain
symmetric monoidal $\infty$-categories $\CAT_{\SSSS}^{\rm pc}\big[M^{-1}\big]^\otimes$ and
$\CAT_{\RR}^{\rm pc}\big[M^{-1}\big]^\otimes$.
Thanks to multiplicative Morita theory \cite[Theorem~4.6]{BGT2},
there is a~canonical equivalence of~symmetric monoidal
\mbox{$\infty$-categories}
\begin{gather*}
\ST^{\otimes}\simeq \CAT_{\SSSS}^{\rm pc}\big[M^{-1}\big]^\otimes.
\end{gather*}
\end{Example}

\begin{Construction}\label{laxsymmetric}There is a~sequence of~symmetric monoidal $\infty$-categories
\begin{gather*}
\CAT_{\RR}^{\rm pc}\to \CAT_{\SSSS}^{\rm pc}\to \CAT_{\SSSS}^{\rm pc}\big[M^{-1}\big]
\end{gather*}
such that $\CAT_{\RR}^{\rm pc}$ and $\CAT_{\SSSS}^{\rm pc}$ are ordinary symmetric
monoidal categories, the left arrow is the forgetful
lax symmetric monoidal
functor,
and the right arrow is the canonical symmetric monoidal functor.
Consequently, if we write $R$
for the image of~$\RR$ in $\CAlg(\SP)$,
it gives rise to a~lax symmetric monoidal functor
\begin{gather*}
\pi\colon\ \CAT_{\RR}^{\rm pc}\to \Mod_{B\RR}\big(\CAT_{\SSSS}^{\rm pc}\big[M^{-1}\big]\big)\simeq \Mod_{\Perf_R}(\ST).
\end{gather*}
To obtain $\CAT_{\RR}^{\rm pc}\to \CAT_{\SSSS}^{\rm pc}$, we need to check that
the essential image of~$\CAT_{\RR}^{\rm pc}\to \CAT_{\SSSS}$
is contained in $\CAT_{\SSSS}^{\rm pc}$.
It~is enough to show that the forgetful functor $\SPS(\RR)\to \SPS(\mathbb{S})=\SPS$ carries cofibrant objects
to cofibrant objects in $\SPS$.
For this purpose, we recall that cofibrations in $\SPS$
with respect to the stable $\mathbb{S}$-model structure
is the smallest weakly saturated class \cite[Definition~A.1.2.2]{HTT} of~morphisms
that contains $\{\mathbb{S}\otimes i\}_{i\in \operatorname{Mon}}$, where $\operatorname{Mon}$ is
the class of~monomorphisms of~symmetric sequences, and
$\mathbb{S}\otimes i$ denotes the morphism of~symmetric spectrum
induced by~$i$, namely, $\mathbb{S}\otimes(-)$ is the left adjoint of~the forgetful functor from $\SPS$ to the category of~symmetric sequences, see~\cite{S}.
The class of~cofibrations in $\SPS(\RR)$
with respect to the stable $\mathbb{\RR}$-model structure
is the smallest weakly saturated class of~morphisms
containing $\{\RR\otimes i=\RR\wedge (\SSSS\otimes i)\}_{i\in \operatorname{Mon}}$.
Note that we assume that $\RR$ is a~cofibrant object in
$\CAlg\big(\SPS\big)$ so that the underlying object~$\RR$ is cofibrant in~$\SPS$.
It~follows that the underlying morphisms
$\RR\otimes i$ in $\SPS$ are cofibrations.
Since $\SPS(\RR)\to \SPS$ preserves colimits,
$\SPS(\RR)\to \SPS$ preserves cofibrations.
\end{Construction}

The following is a~rectification result for $\ST^\otimes_R$.

\begin{Proposition}
\label{symequivalence}
$\pi\colon \CAT_{\RR}^{\rm pc}\to \Mod_{B\RR}\big(\CAT_{\SSSS}^{\rm pc}\big[M^{-1}\big]\big)$ in Construction~\ref{laxsymmetric} induces
 a~symmetric monoidal equivalence
 \begin{gather*}
 \CAT_{\RR}^{\rm pc}\big[M^{-1}\big]^\otimes\simeq \Mod^\otimes_{B\RR}\big(\CAT_{\SSSS}^{\rm pc}\big[M^{-1}\big]\big).
 \end{gather*}
 In~particular, we have
 a~canonical symmetric monoidal equivalence
 \begin{gather*}
 \CAT_{\RR}^{\rm pc}\big[M^{-1}\big]^\otimes\simeq \Mod_{\Perf_R}^\otimes(\ST)=\ST_R^\otimes.
 \end{gather*}
\end{Proposition}

\begin{proof}The underlying functor is an~equivalence of~$\infty$-categories.
In~fact, there are sequence of~categorical equivalences
\begin{gather*}
\CAT_{\RR}^{\rm pc}\big[M^{-1}\big]\to \Mod_{B\RR}(\CAT_{\SSSS}^{\rm pc})\big[M^{-1}\big]\to \Mod_{B\RR}\big(\CAT_{\SSSS}^{\rm pc}\big[M^{-1}\big]\big) ,
\end{gather*}
where the middle $\infty$-category is obtained from
the ordinary category $\Mod_{B\RR}\big(\CAT_{\SSSS}^{\rm pc}\big)$
by~inverting Morita equivalences (we do not claim that
the middle $\infty$-category
admits a~monoidal structure).
The right arrow is induced by~$\Mod_{B\RR}\big(\CAT_{\SSSS}^{\rm pc}\big)\to \Mod_{B\RR}\big(\CAT_{\SSSS}^{\rm pc}\big[M^{-1}\big]\big)$.
The left arrow is induced by~the inclusion
$\CAT_{\RR}^{\rm pc}\hookrightarrow \Mod_{B\RR}\big(\CAT_{\SSSS}^{\rm pc}\big)\subset \Mod_{B\RR}\big(\CAT_{\SSSS}\big)\simeq \CAT_{\RR}$. The functorial cofibrant resolution
$\CAT_{\RR}\to \CAT_{\RR}^{\rm pc}$ induces an~inverse equivalence of~$\CAT_{\RR}^{\rm pc}\big[M^{-1}\big]\allowbreak \to \Mod_{B\RR}\big(\CAT_{\SSSS}^{\rm pc}\big)\big[M^{-1}\big]$
(a cofibrant one is pointwise-cofibrant).
It~is proved in~\cite[Theorem~5.1]{C}
that the right arrow is a~categorical equivalence.
Next, we prove that $\pi\colon \CAT_{\RR}^{\rm pc}\to \Mod_{B\RR}\big(\CAT_{\SSSS}^{\rm pc}\big[M^{-1}\big]\big)$
is symmetric monoidal.
Namely, we show that the natural morphisms
$\Perf_R\to \pi(B\RR)$ and
$\pi(\AAA)\otimes_{R}\pi(\BBB)\to\pi(\AAA\wedge_{\RR}\BBB)$
are equivalences in $\Mod_{B\RR}\big(\CAT_{\SSSS}^{\rm pc}\big[M^{-1}\big]\big)\simeq \Mod_{\Perf_R}(\ST)$, where $\pi(\AAA)\otimes_{R}\pi(\BBB)$ indicates the tensor product
in $\Mod_{\Perf_R}(\ST)$.
By construction,
the morphism $\Perf_R\to \pi(B\RR)$ is an~equivalence.
We~prove that $\pi(\AAA)\otimes_{R}\pi(\BBB)\to\pi(\AAA\wedge_{\RR}\BBB)$
is an~equivalence.
We~here write $\widehat{\AAA}_{\rm pe}$
for the $\RR$-spectral full subcategory of~$\Fun_{\RR}\big(\AAA^{\rm op},\SPS(\RR)\big)^{\rm cf}$ that consists of~cofibrant-fibrant objects lying over $\DDD_{\rm pe}(\AAA)_{\RR}$
($\widehat{\AAA}_{\rm pe}$ is DK-equivalent to $\DDD_{\rm pe}^\Sigma(\AAA)$
as spectral categories,
see the proof~of~Lemma~\ref{Moritainvert} and Claim~\ref{Moritaclaim}
for the notation).
For any other $\RR$-spectral category $\mathcal{P}$,
we define $\widehat{\mathcal{P}}_{\rm pe}$ in the same way.
By Claim~\ref{Moritaclaim} below, the image of~$\widehat{(\AAA\wedge_{\RR}\BBB)}_{\rm pe}$ in $\Mod_{B\RR}\big(\CAT_{\SSSS}^{\rm pc}\big[M^{-1}\big]\big)$ is equivalent to $\pi(\AAA\wedge_{\RR}\BBB)$ in the natural way. Let $\{ \AAA_{\lambda}\}_{\lambda\in \Lambda}$ be
the filtered family (poset) of~$\RR$-spectral full subcategories
of~$\AAA$ such that for any $\AAA_\lambda$,
$\DDD(\AAA_{\lambda})$
(or $\Fun_{\RR}\big(\AAA_{\lambda}^{\rm op},\SPS(\RR)\big)^{\rm cf}$) admits a~single compact generator (so that
$(\widehat{\AAA}_{\lambda})_{\rm pe}$ is Morita equivalent to $B\mathbb{A}$
for some $\mathbb{A}\in \Alg_{\assoc}\big(\SPS(\RR)^c\big)$, where $B\mathbb{A}$
has one object $\ast$ with the morphism ring spectrum $B\mathbb{A}(*,*)=\mathbb{A}$).
Then we have
the filtered family (poset) of~$\RR$-spectral full subcategories
$\big\{\widehat{(\AAA_{\lambda}\wedge_{\RR}\BBB)}_{\rm pe}\big\}_{\lambda\in \Lambda}$
of~$\widehat{(\AAA\wedge_{\RR}\BBB)}_{\rm pe}$. The filtered colimit
of~this family in $\CAT_{\RR}^{\rm pc}\big[M^{-1}\big]$
is $\widehat{(\AAA\wedge_{\RR}\BBB)}_{\rm pe}$.
Indeed, from the categorical
equivalence $\CAT_{\RR}^{\rm pc}\big[M^{-1}\big]\simeq \Mod_{\Perf_R}(\ST)$
and the conservative colimit-preserving functor $\Mod_{\Perf_R}(\ST)\to \ST$, if we transfer $\big\{\widehat{(\AAA_{\lambda}\wedge_{\RR}\BBB)}_{\rm pe}\big\}_{\lambda\in \Lambda}$ into the filtered family of~stable idempotent-complete
$\infty$-categories, we are reduced to proving that the colimit of~the resulting diagram in $\ST$ is a~stable $\infty$-category
obtained from the spectral category $\widehat{(\AAA\wedge_{\RR}\BBB)}_{\rm pe}$.
We~think of~$\big\{\widehat{(\AAA_{\lambda}\wedge_{\RR}\BBB)}_{\rm pe}\big\}_{\lambda\in \Lambda}$
as a~filtered diagram (poset) of~full subcategories of~$\widehat{(\AAA\wedge_{\RR}\BBB)}_{\rm pe}$. Since $\cup_{\lambda\in \Lambda} \widehat{(\AAA_{\lambda}\wedge_{\RR}\BBB)}_{\rm pe}=\widehat{(\AAA\wedge_{\RR}\BBB)}_{\rm pe}$, a~colimit of~this diagram in~$\Cat$ is naturally equivalent to
$\widehat{(\AAA\wedge_{\RR}\BBB)}_{\rm pe}$.
Note that a~filtered colimit of~stable $\infty$-categories in $\ST$ naturally
coincides with a~colimit which is taken in $\Cat$ \cite[Proposition~1.1.4.6]{HA}.
It~follows that $\widehat{(\AAA\wedge_{\RR}\BBB)}_{\rm pe}$ is a~colimit of~$\big\{\widehat{(\AAA_{\lambda}\wedge_{\RR}\BBB)}_{\rm pe}\big\}_{\lambda\in \Lambda}$
in $\ST$.
We~then deduce that $\colim_{\lambda\in \Lambda}\widehat{(\AAA_{\lambda}\wedge_{\RR}\BBB)}_{\rm pe} \to \widehat{(\AAA\wedge_{\RR}\BBB)}_{\rm pe}$
is an~equivalence in $\Mod_{\Perf_R}(\ST)$.
In~addition, we note that the tensor product in
$\Mod_{\Perf_R}(\ST)$ preserves small colimits
in each variable since $\ST^\otimes$ is a~symmetric
monoidal compactly generated $\infty$-category whose tensor product preserves
small colimits in each variable (see \cite[Corollary~4.25]{BGT1}).
Therefore, taking into account Proposition~\ref{flatspectral}(3),
we may and will suppose that $\AAA=B\mathbb{A}$ for some $\mathbb{A}\in \Alg_{\assoc}\big(\SPS(\RR)^c\big)$.
Taking the same procedure to $\BBB$,
we may and will suppose that $\BBB=B\mathbb{B}$
for some $\mathbb{B}\in \Alg_{\assoc}\big(\SPS(\RR)^c\big)$.
We~write $A$ and $B$ for the images of~$\mathbb{A}$ and $\mathbb{B}$ in $\Alg_{\assoc}(\Mod_R)$, respectively.
In~this situation, we have
a canonical equivalence
$\pi(\AAA)\otimes_{R}\pi(\BBB)\simeq (\RMod_A\otimes_R\RMod_{B})^\omega\simeq (\RMod_{A\otimes_{R}B})^{\omega} \simeq \widehat{(B\mathbb{A}\wedge_{\RR}\mathbb{B})}_{\rm pe} \simeq \pi(\AAA\wedge_{\RR}\BBB)$
(cf.~\cite[Proposition~4.1(2)]{BFN}, \cite[Theorem~4.8.5.16]{HA}),
where $(-)^\omega$ indicates the full subcategory of~compact objects. Hence $\pi\colon \CAT_{\RR}^{\rm pc}\to \Mod_{B\RR}\big(\CAT_{\SSSS}^{\rm pc}\big[M^{-1}\big]\big)$ is symmetric monoidal, which induces
a symmetric monoidal functor
$\CAT_{\RR}^{\rm pc}\big[M^{-1}\big]^\otimes\to\Mod_{B\RR}^\otimes\big(\CAT_{\SSSS}^{\rm pc}\big[M^{-1}\big]\big)$ whose underlying functor
is a~categorical equivalence. Thus we have the desired symmetric monoidal equivalence
$\CAT_{\RR}^{\rm pc}\big[M^{-1}\big]^\otimes\stackrel{\sim
}{\to}\Mod_{B\RR}^\otimes\big(\CAT_{\SSSS}^{\rm pc}\big[M^{-1}\big]\big)$.
\end{proof}

Using the equivalence in Proposition~\ref{symequivalence},
we obtain equivalences of~symmetric monoidal \mbox{$\infty$-categories}
\begin{gather*}
\theta\colon\ \ST_R^\otimes=\Mod_{\Perf_R}^\otimes(\ST)\simeq \Mod_{\Perf_R}^\otimes\big(\CAT_{\SSSS}^{\rm pc}\big[M^{-1}\big]\big)\simeq\CAT_{\RR}^{\rm pc}\big[M^{-1}\big]^\otimes.
\end{gather*}

\looseness=1
Next, to an~$\RR$-spectral category we assign
Hoch\-schild homology $R$-module spectrum endowed with circle action.
The construction is based on the Hoch\-schild--Mitchell cyclic nerves
\mbox{(cf.~\cite{BGT2, BM})}.

\begin{Definition}\label{Mitchell}
Let $\AAA$ be a~pointwise-cofibrant small $\RR$-spectral category.
Let $p\ge0$ be a~nonnegative integer.
Let
\begin{gather*}
\HH(\AAA)_p:=\bigvee_{(X_0,\dots,X_p)} \AAA(X_{p-1},X_p)\wedge_\RR\dots \wedge_\RR \AAA(X_0,X_1)\wedge_\RR \AAA(X_p,X_0).
\end{gather*}
The coproduct is taken over the set of~sequences
$(X_0,\dots,X_p)$ of~objects of~$\AAA$.
The composition in $\AAA$ determines degeneracy maps
$d_0,\dots,d_p\colon \HH(\AAA)_p\to \HH(\AAA)_{p-1}$,
and the unit map $\SSSS\to \AAA(X_i,X_i)$ determines
face maps $s_0,\dots,s_p\colon \HH(\AAA)_p\to \HH(\AAA)_{p+1}$.
The cyclic permutation
$(X_0,\dots,X_p)\mapsto (X_p,X_0,\dots,X_{p-1})$
gives rise to the action of~the cyclic group $\ZZ/(p+1)\ZZ$ on~$\HH(\AAA)_p$.
The family $\HH(\AAA)_\bullet:=\{\HH(\AAA)_p\}_{p\ge0}$ equipped with
degeneracy maps and face maps
form a~simplicial object in $\SPS(\RR)$.
If~we take into account the action of~$\ZZ/(p+1)\ZZ$ on $\HH(\AAA)_p$,
then $\HH(\AAA)_{\bullet}=\{\HH(\AAA)_p\}_{p\ge0}$ is promoted to
a cyclic object,
that is, a~functor $\Lambda^{\rm op}\to \SPS(\RR)$
from (the opposite category of) the cyclic category of~Connes
such that the composite $\Delta^{\rm op}\to \Lambda^{\rm op}\to \SPS(\RR)$
is the simplicial object.
Here $\Lambda$ denotes the Connes' cyclic category,
see \cite[Section~2]{Con}, \cite{L}
for the definition of~the cyclic category.
\end{Definition}

We~let $\Fun\big(\Lambda^{\rm op},\SPS(\RR)\big)$ denote the ordinary
functor category from $\Lambda^{\rm op}$ to $\SPS(\RR)$.
The category $\Fun\big(\Lambda^{\rm op},\SPS(\RR)\big)$ inherits
a symmetric monoidal structure given by~the pointwise tensor product
$F\otimes G([n])=F([n])\wedge_\RR G([n])$.

From the definition of~the tensor product of~$\RR$-spectral categories
and the construction of~$\HH(\AAA)_{p}$, it is straightforward
to check that
the assignment $\AAA\mapsto \HH(\AAA)_{\bullet}$ determines
a~symmetric monoidal functor
\begin{gather*}
\HH(-)_{\bullet}\colon \CAT_{\RR}^{\rm pc}\to \Fun\big(\Lambda^{\rm op},\SPS(\RR)\big) .
\end{gather*}
The image of~$\HH(-)_{\bullet}$ is contained in $\Fun\big(\Lambda^{\rm op},\SPS(\RR)^c\big)$
since the stable $\RR$-model structure satisfies
the axiom of~symmetric monoidal model categories.
Let $\SPS(\RR)^c\big[W^{-1}\big]$ be the symmetric monoidal $\infty$-category
obtained from
$\SPS(\RR)^c$
by~inverting stable equivalences.
The underlying $\infty$-category is presentable since
$\SPS(\RR)$ is a~combinatorial model category.
There is a~canonical symmetric monoidal functor
$\Fun\big(\Lambda^{\rm op},\SPS(\RR)^c\big)\to \Fun\big(\Lambda^{\rm op},\SPS(\RR)^c\big[W^{-1}\big]\big)$
induced by~the symmetric monoidal
functor $\SPS(\RR)^c\to \SPS(\RR)^c\big[W^{-1}\big]$.

We~recall the following results from \cite{Con, Hoy, L}:

\begin{Lemma}\label{cycsimp}\qquad
\begin{enumerate}\itemsep=0pt
\renewcommand{\labelenumi}{(\roman{enumi})}
\item[{\rm (i)}] Let $\Lambda\to \overline{\Lambda}$ be the groupoid completion.
Namely, it is induced by~a unit map of~the adjunction $\Cat\rightleftarrows \SSS\, {\colon}\!\iota$, where $\iota$ is the fully faithful inclusion.
Then $\overline{\Lambda}$ is equivalent to~$BS^1$ in $\SSS$.

\item[{\rm (ii)}] Let $\CCC$ be a~presentable $\infty$-category.
Let $F\colon \Lambda^{\rm op}\!\to \CCC$ be a~cyclic object
in $\CCC$.
Let \mbox{$F'\colon BS^1\!\to \CCC$} be a~functor.
Let $\Delta^0\to BS^1$ be the map determined by~the unique object
of~$BS^1$.
Consider the commutative diagram
\begin{gather*}
\xymatrix{
\Delta^{\rm op} \ar[d] \ar[r] & \Lambda^{\rm op} \ar[d] \\
\Delta^0 \ar[r] & BS^1,
}
\end{gather*}
where we regard $\Delta^0$ as the groupoid completion of~$\Delta^{\rm op}$
$(\Delta^{\rm op}$ is sifted so that the groupoid completion is given
by~the contractible space $\Delta^0$,
 cf.~{\rm \cite[Section~5.5.8]{HTT})}.
Then $F'$ is a~left Kan extension of~$F$ along
$\Lambda^{\rm op}\to BS^1$ if and only if
the composite $\Delta^0\to BS^1\to \CCC$ is a~colimit
of~the composite $\Delta^{\rm op}\to \Lambda^{\rm op}\to \CCC$.
\end{enumerate}
\end{Lemma}

\begin{proof}The first assertion is a~result of~Connes, see \cite[Th\'eor\`eme~10]{Con},
\cite{Hoy,L}.
The second assertion is proved in~\cite[Proposition~1.1]{Hoy}.
\end{proof}

\begin{Lemma}\label{symmonoloc}
Let $R\colon \Fun\big(BS^1,\SPS(\RR)^c\big[W^{-1}\big]\big)\to
\Fun\big(\Lambda^{\rm op},\SPS(\RR)^c\big[W^{-1}\big]\big)$ be the functor induced by~composition with
$\Lambda^{\rm op}\to \overline{\Lambda}^{\rm op}\simeq BS^1$.
Let
\begin{gather*}
L\colon\ \Fun\big(\Lambda^{\rm op},\SPS(\RR)^c\big[W^{-1}\big]\big)\to
\Fun\big(BS^1,\SPS(\RR)^c\big[W^{-1}\big]\big)
\end{gather*}
be a~left adjoint
$($the existence of~a~left adjoint follows from the adjoint functor theorem~{\rm \cite{HTT}}
and the fact that both $\infty$-categories are presentable$)$. Then this left adjoint is symmetric monoidal.
\end{Lemma}

\begin{proof}Since the right adjoint is a~symmetric monoidal functor,
the left adjoint is an~oplax symmetric monoidal functor.
Thus it is enough to show that
$L(F\otimes G)\to L(F)\otimes L(G)$ is an~equivalence
in $\Fun\big(BS^1,\SPS(\RR)^c\big[W^{-1}\big]\big)$ for $F$, $G\colon \Lambda^{\rm op}\to \SPS(\RR)^c\big[W^{-1}\big]$.
To this end, note first that
$L(F),L(G)\colon BS^1\to \SPS(\RR)^c\big[W^{-1}\big]$
are given by~left Kan extensions of~$F$ and~$G$, respectively,
along $\Lambda^{\rm op}\to BS^1$.
By Lemma~\ref{cycsimp}(ii),
$F'\colon BS^1\to \SPS(\RR)^c\big[W^{-1}\big]$ is a~left Kan extension
of~$F\colon \Lambda^{\rm op}\to \SPS(\RR)^c\big[W^{-1}\big]$ if and only if
the image of~the unique object $\ast$ of~$BS^1$ under $F'$ is a~colimit of~the composite $F_{\Delta}\colon \Delta^{\rm op}\to \Lambda^{\rm op}\to \SPS(\RR)^c\big[W^{-1}\big]$.
Let $G_{\Delta}$ be the restriciton of~$G$ to $\Delta^{\rm op}$.
Then $L(F\otimes G)(\ast)$ is a~colimit of~the composite $\Delta^{\rm op}\stackrel{\textup{diag}}{\longrightarrow} \Delta^{\rm op} \times \Delta^{\rm op} \stackrel{F_{\Delta}\times G_{\Delta}}{\to} \SPS(\RR)^c[W^{-1}]\times \SPS(\RR)^c\big[W^{-1}\big]\stackrel{\otimes}{\to} \SPS(\RR)^c\big[W^{-1}\big]$.
On the other hand, $L(F)\otimes L(G)(\ast)$
is a~colimit of~$\Delta^{\rm op} \times \Delta^{\rm op} \stackrel{F_{\Delta}\times G_{\Delta}}{\to} \SPS(\RR)^c\big[W^{-1}\big]\times \SPS(\RR)^c\big[W^{-1}\big]\stackrel{\otimes}{\to} \SPS(\RR)^c\big[W^{-1}\big]$ since the tensor product
of~$\SPS(\RR)^c\big[W^{-1}\big]$ preserves small colimits separately in each variable.
Note that the diagonal functor $\Delta^{\rm op}\to \Delta^{\rm op}\times \Delta^{\rm op}$
is cofinal.
Thus the canonical morphism $L(F\otimes G)(\ast)\to L(F)\otimes L(G)(\ast)$ is an~equivalence.
This means that $L$ is symmetric monoidal.
\end{proof}

Assembling these lemmata and constructions, we obtain
a sequence of~symmetric monoidal functors
\begin{gather*}
H\colon\ \CAT_{\RR}^{\rm pc}\stackrel{\HH(-)_\bullet}{\longrightarrow}
\Fun\big(\Lambda^{\rm op},\SPS(\RR)^c\big)
\to \Fun\big(\Lambda^{\rm op},\SPS(\RR)^c\big[W^{-1}\big]\big)
\\ \hphantom{H\colon \ \CAT_{\RR}^{\rm pc}\ \ }
{}\simeq \Fun\big(\Lambda^{\rm op},\Mod_R\big) \stackrel{L}{\to} \Fun\big(BS^1,\Mod_R\big).
\end{gather*}
More explicitly, the composition of~$H$ and the forgetful functor
 $\Fun\big(BS^1,\Mod_R\big)\to \Mod_R\simeq \SPS(\RR)^c\big[W^{-1}\big]$
sends a~pointwise-cofibrant $\RR$-spectral category
$\AAA$ to the homotopy colimit of~the simplicial object
$\Delta^{\rm op}\hookrightarrow \Lambda^{\rm op}\stackrel{\HH(\AAA)_\bullet}{\longrightarrow} \SPS(\RR)$.

Next, we observe the invariance under Morita equivalences.

\begin{Lemma}\label{Moritainvert}
Let $\AAA$ and $\BBB$ be pointwise-cofibrant $\RR$-spectral categories.
Let $F\colon \AAA\to \BBB$ be an~\mbox{$\RR$-spec\-t\-ral} functor
which is a~Morita equivalence.
Then $H\colon \CAT_{\RR}^{\rm pc}\to \Fun\big(BS^1,\Mod_R\big)$
carries~$F$ to an~equivalence in $\Fun\big(BS^1,\Mod_R\big)$.
\end{Lemma}

\begin{proof}It~will suffice to prove that
$H(F)$ is an~equivalence in $\Mod_R$.
By \cite[Theorems~5.9 and~5.11]{BM},
the image of~$H(F)$ in $\Mod_R\simeq \SPS(\RR)\big[W^{-1}\big]$ is an~equivalence
if $F\colon \AAA\to \BBB$ is a~Morita equivalence {\it over} $\RR$.
We~explain the notion of~a~Morita equivalence {\it over} $\RR$,
which we~distinguish from the notion of~Morita equivalences for the moment.
Let $\Fun_{\RR}\big(\AAA^{\rm op},\SPS(\RR)\big)$ be the $\RR$-spectral category
of~$\RR$-spectral functors. As in the case of~$\RR=\SSSS$,
it admits a~combinatorial $\RR$-spectral model structure
whose weak equivalences (resp.~fibrations) are objectwise
stable equivalences (resp.~fibrations).
Let $\DDD(\AAA)_{\RR}$ denote the homotopy (triangulated) category of~the full subcategory $\Fun_{\RR}\big(\AAA^{\rm op},\SPS(\RR)\big)^{\rm cf}$ spanned by~cofibrant and fibrant objects.
Let $\DDD_{\rm pe}(\AAA)_{\RR}$ be the smallest thick subcategory
that contains the image of~the Yoneda embedding $\AAA\to \DDD(\AAA)_\RR$.
We~define $\DDD_{\rm pe}(\BBB)_{\RR}$ in a~similar way.
The functor $F$ induces $(\mathbb{L}F_!)_{\RR}\colon \DDD_{\rm pe}(\AAA)_{\RR}\to \DDD_{\rm pe}(\BBB)_{\RR}$
as~$F$ induces $\mathbb{L}F_!\colon \DDD_{\rm pe}(\AAA)\to \DDD_{\rm pe}(\BBB)$.
We~say that $F$ is a~Morita equivalence over $\RR$ if
$(\mathbb{L}F_!)_{\RR}$ is an~equivalence.
Thus to prove our assertion, it is enough to show the following claim:

\begin{Claim}\label{Moritaclaim}
There exist equivalences $\DDD_{\rm pe}(\AAA)_{\RR}\simeq \DDD_{\rm pe}(\AAA)$
and $\DDD_{\rm pe}(\BBB)_{\RR}\simeq \DDD_{\rm pe}(\BBB)$
which identifies $(\mathbb{L}F_!)_{\RR}$ with
$\mathbb{L}F_!$ up to natural equivalence.
In~particular, $F$ is a~Morita equivalence over $\RR$ if and only if
$F$ is a~Morita equivalence.
\end{Claim}

\Proof~We~let $\DDD^{\Sigma}_{\rm pe}(\AAA)_{\RR}$
be the full subcategory of~$\Fun_{\RR}\big(\AAA^{\rm op},\SPS(\RR)\big)^{\rm cf}$
spanned by~those objects that belongs to $\DDD_{\rm pe}(\AAA)_{\RR}$.
The Yoneda embedding
$I\colon \AAA\to \DDD^{\Sigma}_{\rm pe}(\AAA)_{\RR}$ (after replacing $\AAA$ by~a fibrant one)
induces $I_!\colon \DDD_{\rm pe}^{\Sigma}(\AAA)\to \DDD_{\rm pe}^{\Sigma}\big(\DDD^{\Sigma}_{\rm pe}(\AAA)_{\RR}\big)$.
By \cite[Proposition~4.11]{BGT1},
we deduce that the canonical functor
$\DDD^{\Sigma}_{\rm pe}(\AAA)_{\RR}\to \DDD_{\rm pe}^{\Sigma}\big(\DDD^{\Sigma}_{\rm pe}(\AAA)_{\RR}\big)$
is a~DK-equivalence. Thus we have
$\DDD_{\rm pe}^{\Sigma}(\AAA)\to \DDD_{\rm pe}^{\Sigma}(\AAA)_{\RR}$
(in the homotopy category of~$\CAT_{\SSSS}$).
By the correspondence between
stable idempotent-complete $\infty$-categories and
spectral categories up to Morita equivalences
(\cite[Theorems~4.22 and~4.23]{BGT1}),
passing to stable (idempotent-complete) $\infty$-categories, we may replace
$\DDD_{\rm pe}^{\Sigma}(\AAA)\to \DDD_{\rm pe}^{\Sigma}(\AAA)_{\RR}$
by~the induced exact functor of~stable idempotent-complete
$\infty$-categories which we denote by~$f\colon \Perf(\AAA)\to \Perf(\AAA)_{\RR}$.
Namely, it is enough to show that $f$ is an~equivalence.
Taking account of~the Yoneda embedding, we see that
$f$ is fully faithful on the full subcategory spanned by~the image of~$\AAA$.
It~follows that $f$ is fully faithful on
the smallest stable subcategory of~$\Perf(\AAA)$, containing
$\AAA$. Note that $\Perf(\AAA)_{\RR}$
is the idempotent completion of~the smallest stable subcategory
that contains the image of~$\AAA$ under the Yoneda embedding.
Then we conclude that~$f$ is an~equivalence.
Similarly, $\DDD_{\rm pe}(\BBB)_{\RR}\simeq \DDD_{\rm pe}(\BBB)$.
Finally, using the construction of~these equivalences
and the functoriality of~left adjoints $(-)_!$,
we identify $(\mathbb{L}F_!)_{\RR}$ with
$\mathbb{L}F_!$.
\end{proof}

\begin{Corollary}\label{HHMoritainv}The symmetric monoidal functor $H\colon \CAT_{\RR}^{\rm pc}\to \Fun\big(BS^1,\Mod_R\big)$
factors as
\begin{gather*}
\xymatrix{
\CAT_{\RR}^{\rm pc} \ar[r] \ar[d] & \Fun\big(BS^1,\Mod_R\big) \\
\CAT_{\RR}^{\rm pc}\big[M^{-1}\big] \ar[ru]_{\overline{H}}.
}
\end{gather*}
\end{Corollary}

\begin{proof}It~follows from the universal property of~$\CAT_{\RR}^{\rm pc}\to \CAT_{\RR}^{\rm pc}\big[M^{-1}\big]$
and Lemma~\ref{Moritainvert}.
\end{proof}

Composing with $\theta$, we obtain a~sequence of~symmetric monoidal functors
\begin{gather*}
\xymatrix{
\Mod_{\Perf_R}(\ST)\simeq \Mod_{\Perf_R}\big(\CAT_{\SSSS}^{\rm pc}\big[M^{-1}\big]\big) \ar[r]^(0.7){\theta} & \CAT_{\RR}^{\rm pc}\big[M^{-1}\big] \ar[r]^(0.4){\overline{H}} & \Fun\big(BS^1,\Mod_R\big).
}
\end{gather*}

\begin{Definition}\label{homologydef}
Let $\CCC$ be a~small $R$-linear stable $\infty$-category.
We~denote by~$\HH_\bullet(\CCC)$ the~image of~$\CCC$ in $\Fun\big(BS^1,\Mod_R\big)$
under the above composite $\overline{H}\circ \theta$.
We~often abuse notation by~wri\-ting $\HH_\bullet(\CCC)$ for its image
in $\Mod_R$. We~refer to $\HH_\bullet(\CCC)$
as Hoch\-schild homology $R$-module spectrum of~$\CCC$.
If~$\AAA$ is a~pointwise-cofibrant $\RR$-spectral category,
we refer to the image $\overline{H}(\AAA)$ in~$\Fun\big(BS^1,\Mod_R\big)$ or $\Mod_R$
as Hoch\-schild homology $R$-module spectrum of~$\AAA$.
\end{Definition}

We~record our construction as a~proposition:

\begin{Proposition}\label{hochschildhomologyfunctor}
There is a~sequence of~symmetric monoidal functors
\begin{gather*}
\Mod_{\Perf_R}(\ST)\stackrel{\theta}{\to}\CAT_{\RR}^{\rm pc}\big[M^{-1}\big]\stackrel{\overline{H}}{\longrightarrow} \Fun\big(BS^1,\Mod_R\big)\stackrel{\textup{forget}}{\to} \Mod_R,
\end{gather*}
which to $R$-linear stable $\infty$-categories or pointwise-cofibrant
$\RR$-spectral categories assigns Hoch\-schild homology $R$-module spectra.
In~particular, for any $\infty$-operad $\OO$ it gives rise to
\begin{gather*}
\Alg_{\OO}(\Mod_{\Perf_R}(\ST))\to\Alg_{\OO}\big(\CAT_{\RR}^{\rm pc}\big[M^{-1}\big]\big)\\
\hphantom{\Alg_{\OO}(\Mod_{\Perf_R}(\ST))}{} \to \Alg_\OO\big(\Fun\big(BS^1,\Mod_R\big)\big)\to \Alg_{\OO}(\Mod_R).
\end{gather*}
\end{Proposition}

\section{Construction}\label{constructionsec}

In~this section, we prove Theorem~\ref{mainconstruction}. Namely, we construct the structure of~a~$\KS$-algebra on~the pair of~Hoch\-schild cohomology spectrum and Hoch\-schild homology spectrum. We~maintain the notation of~Section~\ref{homologysec}.

\begin{Definition}\label{onedimrec}
Let $M$ and $M'$ be (possibly empty) finite disjoint unions of~open interval
$(0,1)$ and the circle $\RRR/\ZZ=S^1$.
That is, $M=(0,1)^{\sqcup m}\sqcup \big(S^1\big)^{\sqcup n}$ and $M'=(0,1)^{\sqcup m'}\sqcup \big(S^1\big)^{\sqcup n'}$.
Let $\textup{Emb}^{\rm rec}(M,M')$ be
the space of~rectilinear embeddings.
A rectilinear embedding $M\to M'$ is an~open embedding
such that any restriction $(0,1)\to (0,1)$,
$(0,1)\to S^1$ and $S^1\to S^1$ is rectilinear, see Definition~\ref{recti}.
The topology is induced from
the compact-open topology (or~parameters of~rectilinear maps).

Let $\textup{Mfld}^{\rm rec}_1$ be the
the fibrant simplicial colored operad
whose set $\big(\textup{Mfld}^{\rm rec}_1\big)_{\rm col}$ of~colors consists of~(possibly empty)
finite disjoint unions of~$(0,1)$ and $S^1$.
For a~finite family $\{M_i\}_{i\in I}$ of~colors
and $N\in \big(\textup{Mfld}^{\rm rec}_1\big)_{\rm col}$,
the simplicial hom set $\mul_{\textup{Mfld}^{\rm rec}_1}(\{M_i\}_{i\in I},N)$ is
defined to be the singular complex
of~the space $\textup{Emb}^{\rm rec}(\sqcup_{i\in I}M_i,N)$ of~rectilinear embeddings. The composition is defined in the obvious way.
Then from Definition~\ref{associatedinfoperad},
we obtain the associated $\infty$-operad
$(\Mfld^{\rm rec}_1)^\otimes$ $\to \Gamma$, which is a~symmetric monoidal
$\infty$-category by~construction.
Informally, objects of~this symmetric monoidal $\infty$-category
are finite disjoint unions of~$(0,1)$ and $S^1$, and the
symmetric monoidal structure is given by~disjoint union.
The empty space is a~unit. The mapping spaces are spaces of~rectilinear
embeddings.
Let $\Mfld^{\rm rec}_1$ denote the underlying $\infty$-category.
Let $\Disk_1^{\rm rec}\subset \Mfld^{\rm rec}$ be the full subcategory spanned by~finite disjoint
unions of~$(0,1)$. It~is closed under taking tensor products
so that $\Disk^{\rm rec}_1$ is promoted to a~symmetric monoidal $\infty$-category
$(\Disk_1^{red})^\otimes$ (it~is equivalent to an~ordinary symmetric monoidal 1-category).
\end{Definition}

\begin{Remark}There are several variants which are equivalent to $\Mfld^{\rm rec}_1$.
Let $\Mfld^{\rm fr}_1$ be the $\infty$-category of~framed (or oriented) $1$-manifolds without boundaries whose mapping spaces
are spaces of~embeddings of~framed manifolds (see, e.g., \cite[Section~2]{AF}).
The symmetric monoidal structure is given by~disjoint union.
It~is easy to see that
there is an~equivalence $\Mfld^{\rm rec}_1\stackrel{\sim}{\to} \Mfld^{\rm fr}_1$
as symmetric monoidal $\infty$-categories.
If~we write $\Disk_1^{\rm fr}$
for the full subcategory of~$\Mfld^{\rm fr}_1$
spanned by~framed $1$-disks,
it also induces an~equivalence $\Disk_1^{\rm rec}\stackrel{\sim}{\to} \Disk_1^{\rm fr}$
of~symmetric monoidal \mbox{$\infty$-categories}.
\end{Remark}

From now on, for ease of~notation, we write $\Mfld_1$, $\Disk_1$,
and $\Disk^\otimes_1$ for $\Mfld_1^{\rm rec}$, $\Disk^{\rm rec}_1$,
and~$\big(\Disk_1^{\rm rec}\big)^\otimes$, respectively.

We~set $(\Disk_1)_{/S^1}:=\Disk_1\times_{\Mfld_1}(\Mfld_1)_{/S^1}$.
Let $\lsr$ be the full subcategory of~$\Mfld_1$ that consists of~$S^1$.
By the equivalence $\textup{Emb}^{\rm rec}\big(S^1,S^1\big)\simeq S^1$, it follows that
$\lsr$ is equivalent to $BS^1$, that is, the $\infty$-category
which has one object $\ast$
together with the mapping space \mbox{$\Map_{BS^1}(\ast,\ast)=S^1$}
endowed with the composition law induced by~the multiplication of~$S^1$.
Let $\Disk_1/ \lsr$ be the full subcategory of~$\Fun\big(\Delta^1,\Mfld_1\big)$
which consists of~those functors $h\colon \Delta^1\to \Mfld_1$
such that $h(0)\in \Disk_1$ and $h(1)\in \lsr$.
In~other words, $\Disk_1/ \lsr=\Disk_1\times_{\Mfld_1}\Fun\big(\Delta^1,\Mfld_1\big)\times_{\Mfld_1}\lsr$, where the functor from $\Fun\big(\Delta^1,\Mfld_1\big)$ to the left $\Mfld_1$
(resp.~the right $\Mfld_1$) is induced by~the restriction to the source
(resp.~the target).
The projection
\mbox{$\Delta^0=\{S^1\}\to \lsr=BS^1$}
induces
\begin{gather*}
(\Disk_1)_{/S^1}\simeq \Disk_1\times_{\Mfld_1}\Fun\big(\Delta^1,\Mfld_1\big)\times_{\Mfld_1}\big\{S^1\big\}
\\ \hphantom{(\Disk_1)_{/S^1}}
{}\to \Disk_1\times_{\Mfld_1}\Fun\big(\Delta^1,\Mfld_1\big)\times_{\Mfld_1}\lsr = \Disk_1/ \lsr,
\end{gather*}
where the left categorical equivalence follows from \cite[Proposition~4.2.1.5]{HTT}.

\begin{Lemma}\label{geometricConnes} Let $\Disk^\dagger_1$ be the full subcategory of~$\Disk_1$ spanned by~nonempty spaces (namely, the empty space is omitted from $\Disk_1$).
We~set $\Diskd_1/\lsr=\Diskd_1\times_{\Disk_1}\big(\Disk_1/\lsr\big)$.
Let $\Lambda$ be the cyclic category of~Connes {\rm \cite[Section~2]{Con}}.
There is an~equivalence of~categories $\Lambda^{\rm op}\simeq \Diskd_1/\lsr$.
\end{Lemma}

\begin{proof}
This is a~comparison between definitions which look different.
We~first recall that objects of~$\Lambda$
are $(p)$ for $p\ge0$, which is denoted by~$\Lambda_p$ in~\cite{Con}.
Let $\big(S^1,p\big)$ be the circle $S^1=\RRR/\ZZ$ equipped with
the set of~torsion points $\frac{1}{p+1}\ZZ/\ZZ$.
The hom set $\Hom_{\Lambda}((p),(q))$ is defined to be the set of~homotopy classes of~monotone degree one maps $\phi\colon S^1\to S^1$
such that $\phi\big(\frac{1}{p+1}\ZZ/\ZZ\big)\subset \frac{1}{q+1}\ZZ/\ZZ$.
We~denote points $\frac{0}{p+1},\dots,\frac{p}{p+1}\in\RRR/\ZZ$
by~$x^0_p,\dots,x^p_p$, respectively.
Let $I_p^i=\big\{x_p^i<x<x_p^{i+1}|x\in [0,1)\big\}$ be the open set in $S^1$,
where we use the obvious bijection of~sets $[0,1)\simeq \RRR/\ZZ$.
In~what follows, we regard the superscripts in $x_p^i$ and $I_p^i$
as elements of~$\ZZ/(p+1)\ZZ$.
For $u,v\in [0,1)$, $\{u<x<v\}$ means $\{x\in [0,1)\,|\, u<x<v\}$
if $u<v$, $\{x\in [0,1)\,|\, u<x<1$, \mbox{$0\le x<v\}$}
if $u>v>0$, and $\{x\in [0,1)\,|\, u<x<1\}$ if $u>v=0$.
Given $(p)\in \Lambda$, we think of~$j_p\colon (\RRR/\ZZ)\backslash\big(\frac{1}{p+1}\ZZ/\ZZ\big)=\big(I_p^0\sqcup \dots \sqcup I_p^{p}\big)\hookrightarrow \RRR/\ZZ=S^1$ as an~object of~$\Diskd_1/\lsr$.
We~fix $I^i_p\simeq (0,1)$ such that $I_{p}^i\hookrightarrow \RRR/\ZZ$
is equivalent to $(0,1)\hookrightarrow \RRR\to \RRR/\ZZ$.
We~write $J(p)$ for~it. We~note that every object of~$\Diskd_1/\lsr$
is equivalent to $J(p)$ for some $p\ge0$.
Since each component of~$\Map_{\Mfld_1}\big((0,1)^{\sqcup p+1},S^1\big)$ is naturally
equivalent to $S^1$,
the computation of~mapping spaces shows that $\Diskd_1/\lsr$ is
equivalent to the nerve of~a~$1$-category.
We~may and will abuse notation by~identifying $\Diskd_1/\lsr$
with its homotopy category.
Suppose that we are given a~monotone degree one map $\phi\colon S^1\to S^1$
such that $\phi\big(\frac{1}{p+1}\ZZ/\ZZ\big)\subset \frac{1}{q+1}\ZZ/\ZZ$.
Assume that $p\ge1$ and $\phi\big(j_p\big(I_p^i\big)\big)$ is not a~one-point space.
Let $\sigma(i,\phi)$ be
an element of~$\ZZ/(p+1)\ZZ$ such that
$x_q^{\sigma(i,\phi)}=\phi\big(x_p^i\big)$.
Consider a~rectilinear embedding
\begin{gather*}
\iota_{i,\phi}\colon\ I_q^{\sigma(i,\phi)}\sqcup I_{q}^{\sigma(i,\phi)+1}\sqcup \dots \sqcup I_{q}^{\sigma(i+1,\phi)-1}\simeq (0,1)\sqcup \dots \sqcup (0,1) \hookrightarrow (0,1)\simeq I^i_p
\end{gather*}
such that $\iota_{i,\phi}\big(I^{\sigma(i,\phi)}_q\big)<\dots < \iota_{i,\phi}\big(I_{q}^{\sigma(i+1,\phi)-1}\big)$ in $(0,1)$ (we here abuse notation: for two subsets $S$, $T\subset (0,1)$, $S<T$ if $s<t$ for any pair $(s,t)\in S\times T$).
When $p=0$, we define $\iota_{i,\phi}$
by~replacing $I_q^{\sigma(i,\phi)}\sqcup I_{q}^{\sigma(i,\phi)+1}\sqcup \dots \sqcup I_{q}^{\sigma(i+1,\phi)-1}$ by~$I_q^{\sigma(i,\phi)}\sqcup I_{q}^{\sigma(i,\phi)+1}\sqcup \dots \sqcup I_{q}^{\sigma(i,\phi)-1}$.
Such a~rectilinear embedding is unique up to equivalences.
Given $\phi\in \Hom_{\Lambda}((p),(q))$,
we define the class of~a~map $\phi^*\colon J(q)\to J(p)$ in $\Diskd_1/\lsr$
such that the fiber of~the induced morphism
$I_{q}^0\sqcup \dots \sqcup I_q^q\to I_p^0\sqcup \dots \sqcup I_p^p$
over the connected component $I_p^i$
is (equivalent to) $\iota_{i,\phi}$ if $\phi\big(I_p^i\big)$ is not a~one-point space,
and if otherwise there is no component which maps to $I^i_p$.
Notice that such a~class is unique.
It~is routine
to check that the assignments $(p)\mapsto J(p)$ and $\phi\mapsto \phi^*$
determine a~categorical equivalence $\Lambda^{\rm op}\stackrel{\sim}{\to} \Diskd_1/\lsr$,
where the target is identified with the homotopy category.
\end{proof}

\begin{Lemma}\label{completion}\sloppy
Let $\pi\colon \Diskd_1/\lsr\to \lsr$ be the projection.
It~is a~groupoid completion of~$\Diskd_1/\lsr$.
\end{Lemma}

\begin{proof}By Lemmas~\ref{cycsimp}(i) and~\ref{geometricConnes},
there is
a groupoid completion $c\colon \Lambda^{\rm op}\simeq \Diskd_1/\lsr$ $\to BS^1$.
Thus, by~the universal property, there is a~canonical morphism
from $c\colon \Lambda^{\rm op}\to BS^1$
to~$\pi\colon \Lambda^{\rm op}\simeq \Diskd_1/\lsr\to \lsr$
in $(\CAT_\infty)_{\Lambda^{\rm op}/}$.
It~will suffice to show that the induced morphism
$g\colon BS^1\to \lsr\simeq BS^1$ is an~equivalence, equivalently, it is induced by~an equivalence $S^1\to S^1$ as $\eone$-monoid spaces.
To this end, assume that
$g\colon BS^1\to \lsr\simeq BS^1$ is
induced by~a map $S^1\to S^1$ of~degree $n$,
where $|n|=p+1$, $p>0$. We~will show that this gives rise to
a contradiction.
The automorphism group of~$(p)$
in $\Lambda^{\rm op}\simeq \Diskd_1/\lsr$ is
$\ZZ/(p+1)\ZZ$ so that there is the functor
$h\colon B\ZZ/(p+1)\ZZ\to BS^1$ induced by~$\pi$.
By the factorization $\Lambda^{\rm op}\stackrel{c}{\to} BS^1\stackrel{g}{\to} \lsr\simeq BS^1$
and our assumption, $h\colon B\ZZ/(p+1)\ZZ\to BS^1$
factors through the canonical morphism $\Delta^0\to BS^1$.
Thus, the fiber product of~$B\ZZ/(p+1)\ZZ\stackrel{h}{\to} BS^1\leftarrow \Delta^0=*$
is $B(\ZZ\times\ZZ/(p+1)\ZZ)$.
On the other hand, the space/$\infty$-groupoid
in
$(\Disk_1)_{/S^1}\simeq \Disk_1/\lsr\times_{\langle S^1\rangle}\Delta^0$
spanned by~$J(p)\colon (0,1)^{\sqcup p+1}$ $\to S^1$ (obtained by~discarding non-invertible morphisms) is equivalent to $B\ZZ$.
It~gives rise to a~contradiction $B(\ZZ\times\ZZ/(p+1)\ZZ)\simeq B\ZZ$.
\end{proof}

\begin{Remark}\label{paracyclic}There is another category relevant to the cyclic category:
the paracyclic category
$\Lambda_{\infty}$. Let us recall the definition of~the paracyclic category.
We~follow \cite{GJ}. The set of~objects of~$\Lambda_{\infty}$
is $\{(0)_{\infty},(1)_{\infty},\dots,(p)_{\infty},\dots\}_{p\ge0}$.
The hom set $\Hom_{\Lambda_{\infty}}((p)_{\infty},(q)_{\infty})$
is defi\-ned to be the set of~monotonically increasing maps $f\colon \ZZ\to \ZZ$ such that
$f(i+k(p+1))=f(i)+k(q+1)$ for any $k\in \ZZ$.
We~define a~functor $\Lambda_{\infty}\to \Lambda$
which carries $(p)_\infty$ to $(p)$.
The map
$\Hom_{\Lambda_{\infty}}((p)_{\infty},(q)_{\infty})\to \Hom_{\Lambda}((p),(q))$
carries $f$ to $\phi_f\colon S^1\to S^1$,
where $\phi_f$ is the class of~a~map such that
$\phi_f(x_p^i)=x_q^{f(i)}\in \frac{1}{q+1}\ZZ/\ZZ$
for $i\in \ZZ/(p+1)\ZZ$.
Here, we regard $f(i)$ as belonging to~$\ZZ/(q+1)\ZZ$.
This determines a~functor $\Lambda_{\infty}\to \Lambda$.
Unwinding the definition of~$\Lambda_{\infty}\to \Lambda$,
we see that it is a~(homotopy)
quotient morphism $\Lambda_{\infty}\to \Lambda_{\infty}/B\ZZ\simeq \Lambda$
that comes from
a free action of~$B\ZZ$ on $\Lambda_{\infty}$.
This free action of~$B\ZZ$ is determined by~the natural equivalence from the identity functor $\textup{id}_{\Lambda_{\infty}}$ to itself such that for any $p\ge0$,
the induced map
$(p)_{\infty}\to (p)_{\infty}$ is the map $i\mapsto i+p+1$
 (see \cite{GJ} for details).
The paracyclic category also has a~geometric description.
From the proof of~Lemma~\ref{redcolimit} below,
$\Delta^{\rm op}\to \Lambda^{\infty}$ is (left) cofinal
so that it induces an~equivalence between their groupoid completions.
Since the groupoid completion of~$\Delta^{\rm op}$ is contractible (note
that it's
sifted),
the groupoid completion $\overline{\Lambda}_{\infty}$ of~$\Lambda_\infty$
is a~contractible space.
It~follows that the geometric realization of~$\Lambda$ is equivalent to $BB\ZZ=BS^1$
(see also Lemma~\ref{cycsimp}(i)).
The composition with the opposite functor $\Lambda_{\infty}^{\rm op}\to \Lambda^{\rm op}\simeq \Diskd_1/\lsr \to \lsr=BS^1$ factors through the groupoid completion $\Lambda_{\infty}^{\rm op}\to \overline{\Lambda}_{\infty}^{\rm op}\simeq \Delta^0$.
Consequently, we have the induced functor
$\Lambda_\infty^{\rm op}\to \big(\Diskd_1\big)_{/S^1}\simeq \Diskd_1/\lsr\times_{\langle S^1\rangle}\Delta^0$. This is an~equivalence. Clearly, it is essentially surjective.
The map
$\Hom_{\Lambda_{\infty}^{\rm op}}((q)_\infty,(p)_{\infty})\to \Hom_{\Lambda^{\rm op}}((q),(p))$
is a~homotopy quotient map that comes from a~free action of~$\ZZ$.
We~see that $\Hom_{\Lambda_{\infty}^{\rm op}}((q)_\infty,(p)_{\infty})$
is a~(homotopy) fiber product $\Hom_{\Lambda_{\infty}^{\rm op}}((q)_\infty,(p)_{\infty})/\ZZ\times_{B\ZZ}\Delta^{0}\simeq \Hom_{\Lambda^{\rm op}}((q),(p))\times_{B\ZZ}\Delta^0$.
It~follows that $\Lambda_\infty^{\rm op}\to \big(\Diskd_1\big)_{/S^1}$ is~fully faithful.
Hence $\Lambda_\infty^{\rm op}\stackrel{\sim}{\to} \big(\Diskd_1\big)_{/S^1}$.
\end{Remark}

\begin{Lemma}\label{redcolimit}
Let $\CCC$ be a~presentable $\infty$-category.
Let $\Lambda_\infty$ be the paracyclic category, see~{\rm \cite{GJ}} or Remark~{\rm \ref{paracyclic}}.
Let $\Lambda_{\infty}^{\rm op}\simeq \big(\Diskd_1\big)_{/S^1}\to \Diskd_1/\lsr\simeq \Lambda^{\rm op}$
be the natural functor.
Let $f\colon \Diskd_1/\lsr\to \CCC$ and $g\colon \lsr\to \CCC$ be
functors and let $f\to \pi\circ g$ be a~natural transformation.
Then $g$ is a~left Kan extension of~$f$ along $\pi\colon \Diskd_1/\lsr \to \lsr$ if and only if
$\Delta^0\to BS^1\simeq \lsr \stackrel{g}{\to} \CCC$ determines
a colimit of~the composite $\big(\Diskd_1\big)_{/S^1}\to\Diskd_1/\lsr \stackrel{f}{\to} \CCC$.

Moreover, the inclusion $\big(\Diskd_1\big)_{/S^1}\to (\Disk_1)_{/S^1}$
is cofinal. Therefore, if we suppose that
$f\colon \Diskd_1/\lsr\to \CCC$ is the restriction of~a functor $\tilde{f}\colon \Disk_1/\lsr\to \CCC$,
the above condition
that $g$ is a~left Kan extension of~$f$
is also equivalent to the condition that
$\Delta^0\to BS^1\to \CCC$ determines
a colimit of~the composite $(\Disk_1)_{/S^1}\to\Disk_1/\lsr\stackrel{\tilde{f}}{\to} \CCC$,
where the first arrow $(\Disk_1)_{/S^1}\to\Disk_1/\lsr$ is the natural functor.
\end{Lemma}

\begin{proof}
There is a~faithful functor $\Delta^{\rm op}\to \Lambda^{\rm op}_{\infty}$ that is (left) cofinal \cite[Proposition~4.2.8]{LRo}: it~is the same as the functor $m$ in Remark~\ref{globaldiagram}.
Thus, for any paracyclic object $F\colon \Lambda_{\infty}^{\rm op}\to \CCC$,
the canonical morphism $\colim_{[p]\in \Delta^{\rm op}}F([p]) \to \colim_{(p)_{\infty}\in \Lambda_{\infty}^{\rm op}}F((p)_{\infty})$ is an~equivalence.
Our first assertion now follows from this fact and Lemma~\ref{cycsimp}(ii).

To prove the second assertion, it will suffice to prove
that $\big(\big(\Diskd_1\big)_{/S^1}\big)_{e/}=\big(\Diskd_1\big)_{/S^1}$ $\times_{(\Diskd_1)_{/S^1}}((\Disk_1)_{/S^1})_{e/}$
is weakly contractible, where $e$ is the map $e\colon \phi\to S^1$
from the empty space to $S^1$.
Since $e$ is an~initial object in $(\Disk_1)_{/S^1}$, we are reduced to
proving
that $\big(\Diskd_1\big)_{/S^1}$ $\simeq \Lambda_\infty^{\rm op}$ is weakly contractible.
By Quillen's theorem A,
it is clear because
$\Delta^{\rm op}$ is weakly contractible and $\Delta^{\rm op}\to \Lambda_{\infty}^{\rm op}$ is (left) cofinal (it follows also from the fact that $\big(\Diskd_1\big)_{/S^1}$ is sifted).
\end{proof}

\begin{Remark}
\label{globaldiagram}
We~have the following commutative diagram of~categories:
\begin{gather*}
\xymatrix{
\big\{(0,1)\to S^1\big\} \ar[r] \ar[d]^{\simeq} & \big(\big(\Diskd_1\big)_{/S^1}\big)_{(0,1)\to S^1/} \ar[r]^{\textup{forget}} \ar[d]^{\simeq} & \big(\Diskd_1\big)_{/S^1} \ar[r] \ar[d]^{\simeq} & \Diskd_1/\lsr \ar[d]^{\simeq} \\
\{[0]\} \ar[r] & \Delta^{\rm op} \ar[r]^{m} & \Lambda_{\infty}^{\rm op} \ar[r] & \Lambda^{\rm op}.
}
\end{gather*}
It~is straightforward to observe that the composite
$\big(\big(\Diskd_1\big)_{/S^1}\big)_{(0,1)\to S^1/}\to \Lambda^{\rm op}$ is
a faithful and~essentially surjective functor
whose image is $\Delta^{\rm op}$ contained in $\Lambda^{\rm op}$.
\end{Remark}

Let $\textup{Mfld}^{\rm ic}_1$ be the
colored simplicial full suboperad of~$\textup{Mfld}^{\rm rec}_1$ (Definition~\ref{onedimrec})
whose set $\big(\textup{Mfld}^{\rm ic}_1\big)_{\rm col}$ of~colors is
$\big\{(0,1),S^1\big\}$.
Namely, for $M_i,N\in \{(0,1),S^1\}$,
the simplicial Hom set $\mul_{\textup{Mfld}^{\rm ic}_1}(\{M_i\}_{i\in I},N)$
is the singular complex
of~the space $\textup{Emb}^{\rm rec}(\sqcup_{i\in I}M_i,N)$ of~rectilinear embeddings. The superscript ``$ic$'' stands for the ``interval'' and the ``circle".
Let $\textup{Disk}_1$ be the full suboperad
whose set of~colors is $\{(0,1)\}$.
Notice that $\textup{Disk}_1$ is identical with $\textup{E}_1$ in Section~\ref{operadsec}.
From Definition~\ref{associatedinfoperad},
we obtain the associated $\infty$-operad
$\big(\textup{Mfld}^{\rm ic}_1\big)_{\Delta}\to \FIN$ of~simplicial categories
constructed from $\textup{Mfld}^{\rm ic}_1$.
Also, to $\textup{E}_1=\textup{Disk}_1\subset \textup{Mfld}^{\rm ic}_1$
we associate $(\textup{E}_1)_{\Delta}=(\textup{Disk}_1)_{\Delta}\to \FIN$.

\begin{Construction}\label{exttensor}
We~define
$\rho\colon (\textup{E}_1)_{\Delta}\times \big(\textup{Mfld}^{\rm ic}_1\big)_{\Delta} \to \DCYL_{\Delta}$
which makes the following diagram commute
\begin{gather*}
\xymatrix{
(\textup{E}_1)_{\Delta}\times \big(\textup{Mfld}^{\rm ic}_1\big)_{\Delta} \ar[r]^(0.7){\rho} \ar[d] & \DCYL_{\Delta} \ar[d] \\
\FIN\times \FIN \ar[r]^\wedge & \FIN.
}
\end{gather*}
Here $\DCYL_\Delta \to \FIN$ is the map of~simplicial categories
associated to $\DCYL$.
The lower hori\-zon\-tal arrow $\wedge\colon \FIN\times \FIN\to \FIN$ that sends
a pair $(\langle m\rangle, \langle n\rangle)$ to $\langle mn\rangle$
is defined in~\cite[Nota\-tion~2.2.5.1]{HA}.
Let $X=(\langle m \rangle,(L_1,\dots, L_m))$ be an~object
of~$(\textup{E}_1)_{\Delta}$, where $L_s=(0,1)$ for each $1\le s\le m$.
Let $Y=(\langle n \rangle,(M_1,\dots, M_n))$ be an~object of~$\big(\textup{Mfld}^{\rm ic}_1\big)_{\Delta}$
where $M_t$ is either $(0,1)$ or $S^1$ for each $1\le t\le n$.
We~define $\rho((X,Y))$ to be
$(\langle mn \rangle,(L_s\times M_t))_{\begin{subarray}{c}1\le s\le m \\ 1\le t\le n\end{subarray}}$,
where we abuse notation
by~the identifications $(0,1)^{2}=D$ and $(0,1)\times S^1=C$, see Definition~\ref{dcyldef}.

$X'=(\langle m' \rangle,(L'_1,\dots, L'_m))$ be another object
of~$(\textup{E}_1)_{\Delta}$.
Let $Y'=(\langle n' \rangle,(M'_1,\dots, M'_n))$
be~ano\-ther object of~$(\textup{Mfld}^{\rm ic}_1)_{\Delta}$.
Given a~pair of~morphisms $f\colon \langle m\rangle \to \langle m'\rangle$
and $g\colon \langle n\rangle \to \langle n'\rangle$,
we define a~map
\begin{gather*}
\xymatrix{
\textup{Emb}^{\rm rec}\big(\sqcup_{s\in f^{-1}(s')} L_s,L_{s'}'\big)\times \textup{Emb}^{\rm rec}\big(\sqcup_{t\in g^{-1}(t')} M_t,M_{t'}'\big) \ar[d] \\ \textup{Emb}^{\rm rec}\big(\sqcup_{(s,t)\in f^{-1}(s')\times g^{-1}(t')} L_s\times M_t,L_{s'}'\times M_{t'}'\big)
}
\end{gather*}
that sends $(\phi,\psi)$ to $\phi\times \psi$.
Taking the product parameterized by~$(s',t')$ with $1\le s'\le m'$, $1\le t'\le n'$ and taking simplicial
nerves, we obtain morphisms of~hom simplicial sets.
It~gives rise to a~functor $\rho\colon (\textup{E}_1)_{\Delta}\times \big(\textup{Mfld}^{\rm ic}_1\big)_{\Delta} \to \DCYL_{\Delta}$ which makes the diagram commute.
This construction is a~natural extension of~that in~\cite[Construction~5.1.2.1]{HA}:
Let $\langle D\rangle\subset \DCYL$ be the full suboperad
whose set of~colors is $\{D\}$.
Let $\langle D\rangle_{\Delta}\to \FIN$ be the correpsonding
simplicial full subcategory of~$\DCYL_{\Delta}$.
Then the restriction of~$\rho$ induces
$(\textup{E}_1)_{\Delta}\times (\textup{E}_1)_{\Delta} \to \langle D\rangle_{\Delta}$
lying over $\wedge\colon \FIN\times \FIN \to \FIN$,
which is defined in {\it loc.\ cit}.

Let $\mathbf{Mfld}_1$ be the simplicial nerve of~$\big(\textup{Mfld}_1^{\rm ic}\big)_{\Delta}$.
The simplicial nerves of~the above diagrams
give rise to the~commutative diagram
\begin{gather*}
\xymatrix{
\eone^\otimes\times \eone^\otimes \ar[r] \ar[d] & \etwo^\otimes \ar[d] \\
\eone^\otimes\times \mathbf{Mfld}_1 \ar[r]^{\rho} & \dcyl,
}
\end{gather*}
which lies over $\wedge\colon \Gamma\times \Gamma\to \Gamma$.
We~abuse notation by~writing $\rho$ for the associated map.
\end{Construction}

Given an~$\infty$-operad $\OO^\otimes\to \Gamma$,
there exist a~symmetric monoidal $\infty$-category $\operatorname{Env}(\OO^\otimes)\to \Gamma$ and a~map $\OO^\otimes\to \operatorname{Env}\big(\OO^\otimes\big)$
of~$\infty$-operads such that
for any symmetric monoidal $\infty$-category $\DDD^\otimes$,
the composition induces
a categorical equivalence $\Fun^\otimes\big(\operatorname{Env}\big(\OO^\otimes\big),\DDD^\otimes\big)\stackrel{\sim}{\to} \Alg_{\OO^\otimes}\big(\DDD^\otimes\big)$, see \cite[Section~2.2.4]{HA}.
Here $\Fun^\otimes\big(\operatorname{Env}\big(\OO^\otimes\big),\DDD^\otimes\big)$ denotes
the $\infty$-category of~symmetric monoidal functors.
We~shall refer to $\operatorname{Env}\big(\OO^\otimes\big)$
as the symmetric monoidal envelope of~$\OO^\otimes$
(in {\it loc.\ cit.}, it is referred to as the $\Gamma$-monoidal envelope).
Through the categorical equivalence, for a~map of~$\infty$-operads
$f\colon \OO^\otimes\to \DDD^\otimes$, there exists a~symmetric monoidal functor
$\tilde{f}\colon \operatorname{Env}\big(\OO^\otimes\big)\to \DDD^\otimes$ which is unique up
to a~contractible space of~choices. We~refer to $\tilde{f}$
as a~symmetric monoidal functor that corresponds to $f$.
Let $\OP$ be the $\infty$-category of~(small) $\infty$-operads \cite[Section~2.1.4]{HA}
and let $\CATI^\otimes$ be the $\infty$-category of~(small)
symmetric monoidal
$\infty$-categories whose morphisms are symmetric monoidal functors.
Then the construction of~symmetric monoidal envelopes gives
a~left adjoint $\OP\to \CATI^\otimes$ of~the canonical functor $\CATI^\otimes\to \OP$.
Here are some examples. The symmetric monoidal envelope $\widetilde{\mathbf{E}}_1^\otimes$
of~$\eone^\otimes$ is equivalent to $\Disk_1^\otimes$ as symmetric monoidal
\mbox{$\infty$-categories}.
Similarly, a~symmetric monoidal envelope $\widetilde{\mathbf{E}}_2^\otimes$
of~$\etwo^\otimes$
is equivalent to the symmetric monoidal $\infty$-category $\Disk_2^\otimes$
of~(possibly empty) finite disjoint unions of~$(0,1)^2$ defined
as in the case of~$\Disk_1^\otimes$: mapping spaces
are spaces of~rectilinear embeddings, and the tensor product is
again given by~disjoint union.
Another quick example of~symmetric monoidal envelopes
is $\mathbf{Mfld}_1\to \Mfld_1^\otimes$.

\begin{Construction}\label{symfunctorsym}
Let $q\colon \CCC^\otimes\to \Gamma$ be a~symmetric monoidal $\infty$-category.
Let $p\colon \mathcal{P}^\otimes\to \Gamma$ be a~symmetric monoidal
$\infty$-category (resp.~an $\infty$-operad).
We~construct a~symmetric monoidal structure on the $\infty$-category
$\Fun^\otimes\big(\PPP^\otimes, \CCC^\otimes\big)$ of~symmetric monoidal functors
(resp.~the $\infty$-category $\Alg_{\PPP^\otimes}\big(\CCC^\otimes\big)$
of~algebra objects),
see \cite[Section~3.2.4]{HA} for more details of~the case of~$\Alg_{\PPP^\otimes}\big(\CCC^\otimes\big)$.
We~define a~map $\Fun^\otimes\big(\PPP^\otimes, \CCC^\otimes\big)^\otimes\to\Gamma$
(resp.~$\Alg_{\PPP^\otimes}^\otimes\big(\CCC^\otimes\big)\to \Gamma$)
by~the universal property that for any $\alpha\colon K\to \Gamma$,
the set of~morphisms $K\to \Fun^\otimes\big(\PPP^\otimes, \CCC^\otimes\big)^\otimes$
over~$\Gamma$ (resp.~$K\to \Alg_{\PPP^\otimes}^\otimes\big(\CCC^\otimes\big)$
over~$\Gamma$)
is defined to be the set of~morphisms
$f\colon K\times \PPP^\otimes \to \CCC^\otimes$ such that
\begin{enumerate}\itemsep=0pt
\renewcommand{\labelenumi}{(\roman{enumi})}
\item the diagram
\begin{gather*}
\xymatrix{
K\times \PPP^\otimes \ar[r]^f \ar[d]_{(\alpha,\textup{id})} & \CCC^\otimes \ar[d]^q \\
\Gamma\times\PPP^\otimes \ar[r] & \Gamma
}
\end{gather*}
commutes; here the lower horizontal arrow is induced by~$\wedge\colon \Gamma\times \Gamma\to \Gamma$,

\item for any vertex $k$ of~$K$ and any $p$-coCartesian edge $\phi$
in $\PPP^\otimes$, $f(k,\phi)$ is a~$q$-coCartesian edge
(resp.~for any vertex $k$ of~$K$ and any inert morphism $\phi$
in $\PPP^\otimes$, $f(k,\phi)$ is an~inert morphism in~$\CCC^\otimes$).

\end{enumerate}
The morphism $\Alg_{\PPP^\otimes}^\otimes\big(\CCC^\otimes\big)\to\Gamma$
is a~symmetric monoidal $\infty$-ca\-te\-gory whose underlying $\infty$-cate\-gory
is $\Alg_{\PPP^\otimes}(\CCC)$, cf.~\cite[Proposition~3.2.4.3]{HA}.
Similarly, $\Fun^\otimes\big(\PPP^\otimes, \CCC^\otimes\big)^\otimes\to\Gamma$
is a~symmetric monoidal $\infty$-category whose underlying $\infty$-cate\-gory
is $\Fun^\otimes\big(\PPP^\otimes, \CCC^\otimes\big)$: the proof~of~\cite[Pro\-po\-si\-tion~3.2.4.3]{HA}
based on the theory of~categorical patterns
can also be applied to $\Fun^\otimes\big(\PPP^\otimes, \CCC^\otimes\big)^\otimes$.
An edge $\Delta^1\to \Fun^\otimes\big(\PPP^\otimes, \CCC^\otimes\big)^\otimes$
is a~coCartesian edge if and only if for any $X\in \PPP$,
the
composite $\Delta^1\times \{X\}\subset \Delta^1\times \PPP^\otimes\to \CCC^\otimes$ determines a~$q$-coCartesian edge (this means that the tensor product $F\otimes G$
of~two symmetric monoidal functors
$F\colon \PPP^\otimes \to \CCC^\otimes$ and $G\colon \PPP^\otimes\to \CCC^\otimes$
is informally given by~objectwise tensor products
$(F\otimes G)(X)=F(X)\otimes G(X)$).

Let $\OO^\otimes\to \Gamma$ be an~$\infty$-operad and let $\widetilde{\OO}^\otimes\to \Gamma$
be the symmetric monoidal envelope. The composition with the inclusion $\OO^\otimes\to \widetilde{\OO}^\otimes$ induces a~map over $\Gamma$
\begin{gather*}
\Fun^\otimes\big(\widetilde{\OO}^\otimes, \CCC^\otimes\big)^\otimes\to \Alg^\otimes_{\OO^\otimes}\big(\CCC^\otimes\big)
\end{gather*}
that preserves coCartesian edges, namely, it is a~symmetric monoidal functor.
Since the underlying functor is an~equivalence \cite[Proposition~2.2.4.9]{HA},
it gives rise to a~symmetric monoidal equivalence.
That is, the categorical equivalence $\Fun^\otimes\big(\widetilde{\OO}^\otimes, \CCC^\otimes\big)\simeq \Alg_{\OO^\otimes}\big(\CCC^\otimes\big)$ is promoted to a~symmetric monoidal equivalence in the natural way.
\end{Construction}

Let $A$ be an~$\etwo$-algebra in $\Mod_R$.
By definition,
it is a~map of~$\infty$-operads $A\colon \etwo^\otimes \to \Mod_R^\otimes$
over $\Gamma$. We~denote by~\begin{gather*}
i_!(A)\colon\ \dcyl\to \Mod_R^\otimes
\end{gather*}
the operadic left Kan extension of~$A$ along the inclusion $i\colon \etwo^\otimes\hookrightarrow \dcyl$. If~we think of~the color $S^1$ as an~object in the fiber $(\mathbf{Mfld}_1)_{\langle 1\rangle}$ of~$\mathbf{Mfld}_1\to \Gamma$ over $\langle 1\rangle$,
the full subcategory $\big\langle S^1\big\rangle$ spanned by~$S^1$
determines the inclusion $\iota\colon BS^1\simeq \big\langle S^1\big\rangle\hookrightarrow (\mathbf{Mfld}_1)_{\langle 1\rangle}\subset \mathbf{Mfld}_1$.
Then we have the following diagram
\begin{gather*}
\xymatrix{
\eone^\otimes \times BS^1\ar[r]^{\textup{id}\times \iota} \ar[d] & \eone^\otimes \times \mathbf{Mfld}_1\ar[r]^{\rho} \ar[d] & \dcyl\ar[r]^{i_!(A)} \ar[d] & \Mod_R^\otimes \ar[d] \\
\Gamma\times\{\langle 1\rangle\}\ar[r] & \Gamma\times \Gamma \ar[r]^{\wedge} & \Gamma \ar[r]^{\textup{id}} & \Gamma.
}
\end{gather*}
See Construction~\ref{exttensor} for $\rho$.
The composite $\Gamma\simeq \Gamma\times\{\langle 1\rangle\}\to \Gamma$ of~lower horizontal arrows is the identity map.
Note that the composite $\rho\circ (\textup{id}\times \iota)\colon \eone^\otimes \times BS^1\to \dcyl$ is the map
$z\colon \eone^\otimes \times BS^1\to \cyl\subset \dcyl$
which was defined in the discussion before Proposition~\ref{cyleones}.
Taking into account the above diagram, Proposition~\ref{cyleones} and Lemma~\ref{eonesequiv},
we have the induced functors
\begin{gather*}
\Alg_{\dcyl}(\Mod_R)\to \Alg_{\cyl}(\Mod_R)\stackrel{\sim}{\to} \Alg_{\eone}\big(\Fun\big(BS^1,\Mod_R\big)\big),
\end{gather*}
and we write $i_!(A)_C$ for the image of~$i_!(A)$ in $\Alg_{\eone}\big(\Fun\big(BS^1,\Mod_R\big)\big)\simeq \Alg_{\cyl}(\Mod_R)$.

\begin{Remark}The image of~$C=(0,1)\times S^1$ under $i_!(A)\colon \dcyl\to \Mod_{R}^\otimes$
can be viewed as the factorization homology $\int_{C}A$ in $\Mod_R$ in this context, cf.~\cite{AF, HA}.
\end{Remark}

We~continue to suppose that $A$ is an~$\etwo$-algebra in $\Mod_R$.
Let us consider the Hoch\-schild homology $R$-module spectrum of~$A$
defined as follows.
Let $\Alg_{\assoc}\big(\SPS(\RR)^{c}\big)$ be the category of~associative algebra objects of~$\SPS(\RR)^{c}$,
where $\RR$ be a~(cofibrant) commutative symmetric ring spectrum that
represents~$R$, and $\SPS(\RR)^{c}$ is the full subcategory of~$\SPS(\RR)$ spanned by~cofibrant objects (cf.~Section~\ref{homologysec}).
The ordinary category $\Alg_{\assoc}\big(\SPS(\RR)^{c}\big)$ admits
a symmetric monoidal structure given by~$\mathbb{A}\otimes \mathbb{B}=\mathbb{A}\wedge_{\RR}\mathbb{B}$.
Define a~symmetric monoidal functor $\Alg_{\assoc}\big(\SPS(\RR)^{c}\big)\to
\CAT_{\RR}^{\rm pc}$ which carries $\mathbb{A}$ to $B\mathbb{A}$, where $B\mathbb{A}$ is the
$\RR$-spectral category having one object $\ast$ with
the morphism spectrum $\mathbb{A}=B\mathbb{A}(\ast,\ast)$.
We~define $\HHH(\mathbb{A})$ to be the Hoch\-schild homology $R$-module spectrum of~$B\mathbb{A}$.
Namely, we use canonical symmetric monoidal functors
\begin{gather*}
\Alg_{\assoc}\big(\SPS(\RR)^{c}\big)\to \CAT_{\RR}^{\rm pc}\to \CAT_{\RR}^{\rm pc}\big[M^{-1}\big]\stackrel{\overline{H}}{\longrightarrow} \Fun\big(BS^1,\Mod_R\big),
\end{gather*}
see Corollary~\ref{HHMoritainv}.
By inverting weak equivalences we obtain symmetric monoidal functors
\begin{gather*}
\Alg_{\assoc}(\Mod_R)\simeq \Alg_{\assoc}\big(\SPS(\RR)^{c}\big)\big[W^{-1}\big] \to \CAT_{\RR}^{\rm pc}\big[M^{-1}\big]\to \Fun\big(BS^1,\Mod_R\big),
\end{gather*}
see Example~\ref{spectralalgebra} for the first symmetric equivalence.
This functor sends $A\in \Alg_{\assoc}(\Mod_R)$ to $\HHH(A)$.
Note that there is
a canonical categorical equivalence
\begin{gather*}
\Alg_{\etwo}(\Mod_R)\simeq \Alg_{\assoc}\Alg_{\assoc}(\Mod_R)
\end{gather*}
that follows from the trivial fibration $\eone^\otimes\to \assoc^\otimes$ and the equivalence $\etwo^\otimes\simeq \eone^\otimes\otimes \eone^\otimes$ (Dunn additivity theorem).
Thus, we have the induced functor
\begin{gather*}
\Alg_{\etwo}(\Mod_R)\simeq \Alg_{\assoc}\Alg_{\assoc}(\Mod_R)\to \Alg_{\assoc}\big(\Fun\big(BS^1,\Mod_R\big)\big).
\end{gather*}
Given $A\in \Alg_{\etwo}(\Mod_R)$,
we define $\HH_\bullet(A)$ to be the image of~$A$ in $\Alg_{\assoc}\big(\Fun\big(BS^1,\Mod_R\big)\big)$.

\begin{Proposition}\label{identificationprop}
There is a~canonical equivalence \begin{gather*}
\HH_\bullet(A)\stackrel{\sim}{\to} i_!(A)_C
\end{gather*}
in $\Alg_{\assoc}(\Fun(BS^1,\Mod_R))$.
\end{Proposition}

Let us consider $\Alg_{\assoc}\left(\SPS(\RR)^{c}\right)\stackrel{B(-)}{\longrightarrow} \CAT_{\RR}^{\rm pc}\stackrel{\HH(-)_\bullet}{\longrightarrow} \Fun\left(\Lambda^{\rm op},\SPS(\RR)^c\right)$, see Definition~\ref{Mitchell} for $\HH(-)_\bullet$.
We~write $\HHH^{\Lambda}(-)$ for the composite.
Let $\widetilde{\assoc}^\otimes$ be the symmetric monoidal envelope
of~$\assoc^\otimes$.
The equivalence $\eone^\otimes\stackrel{\sim}{\to} \assoc$ induces
$\Disk_1^\otimes\stackrel{\sim}{\to}\widetilde{\assoc}^\otimes$.
There is a~canonical symmetric monoidal equivalence
$\Fun^\otimes\big(\widetilde{\assoc}^\otimes,\big(\SPS(\RR)^{c}\big)^{\otimes}\big)\simeq \Alg_{\assoc}\big(\SPS(\RR)^{c}\big)$, see Construction~\ref{symfunctorsym}.
We~write $\Disk_1\hookrightarrow \Disk_1^\otimes$ for the inclusion of~the fiber of~the coCartesin fibration $\Disk_1^\otimes\to \Gamma$
over $\langle 1\rangle$.
Using Lemma~\ref{geometricConnes}, we have
\begin{gather*}
\xi\colon\ \Lambda^{\rm op}\simeq \Diskd_1/\langle S^1\rangle\stackrel{\textup{forget}}{\longrightarrow} \Disk_1\hookrightarrow \Disk_1^\otimes\simeq \widetilde{\assoc}^\otimes.
\end{gather*}
The composition with $\xi$
induces
\begin{gather*}
g\colon\ \Alg_{\assoc}\big(\SPS(\RR)^{c}\big)\simeq \Fun^\otimes\big(\widetilde{\assoc}^\otimes,\big(\SPS(\RR)^{c}\big)^{\otimes}\big)\to \Fun\big(\Lambda^{\rm op},\SPS(\RR)^{c}\big),
\end{gather*}
which is a~symmetric monoidal functor.

\begin{Lemma}\label{observationmap}The functor $\HHH^\Lambda(-)\colon \Alg_{\assoc}\big(\SPS(\RR)^{c}\big)\to \Fun\big(\Lambda^{\rm op},\SPS(\RR)^{c}\big)$ can be identified with
\begin{gather*}
g\colon\ \Alg_{\assoc}\big(\SPS(\RR)^{c}\big)\simeq \Fun^\otimes\big(\widetilde{\assoc}^\otimes,\big(\SPS(\RR)^{c}\big)^\otimes\big)\to \Fun\big(\Lambda^{\rm op},\SPS(\RR)^{c}\big)
\end{gather*}
in the natural way.
In~particular,
$\Alg_{\assoc}\big(\SPS(\RR)^{c}\big)\big[W^{-1}\big]\to \Fun\big(\Lambda^{\rm op},\SPS(\RR)^c\big[W^{-1}\big]\big)$ induced by~$\HH(-)_\bullet$ can be identified with
\begin{gather*}
\Alg_{\assoc}\big(\SPS(\RR)^{c}\big)\big[W^{-1}\big]\simeq \Fun^\otimes\big(\widetilde{\assoc}^\otimes,\SPS(\RR)^{c}\big[W^{-1}\big]^\otimes\big)\to \Fun\big(\Lambda^{\rm op},\SPS(\RR)^{c}\big[W^{-1}\big]\big)
\end{gather*}
induced by~the composition with $\xi$.
\end{Lemma}

\begin{proof} We~use the notation in Lemma~\ref{geometricConnes}. Let $\phi_{p-1,i}\colon S^1\to S^1$ be a~monotone degree one map which we think of~as a morphism $(p-1)\to (p)$ in
$\Lambda$
such that
$\phi_{p-1,i}\big(x^k_{p-1}\big)=x_p^k$ for $0\le k\le i-1$,
and $\phi_{p-1,i}\big(x^k_{p-1}\big)=x_p^{k+1}$ for $i\le k\le p-1$
(in particular, $x^i_p$ does not lie in the image of~$\big\{x_p^k\big\}_{0\le k \le p-1}$).
Let $\mathbb{A}$ be an~object of~$\Alg_{\assoc}\big(\SPS(\RR)^{c}\big)$.
Consider the composition $\Lambda^{\rm op}\simeq \Diskd_1/\lsr\to \widetilde{\assoc}^\otimes \to \SPS(\RR)^{c}$, where the final map is a~symmetric monoidal
functor corresponding to $\mathbb{A}$.
By inspection, if $\{\phi_{p-1,i}\}_{0\le i \le p}$
are regarded as morphisms $(p)\to (p-1)$ in $\Lambda^{\rm op}$,
their images in $\SPS(\RR)^{c}$ define $(p+1)$ degeneracy maps $\mathbb{A}^{\wedge p+1}\to \mathbb{A}^{\wedge p}$
given by~the multiplication $\mathbb{A}\wedge \mathbb{A}\to \mathbb{A}$.
Let $\psi_{p,i}\colon S^1\to S^1$ be a~monotone degree one map which we think of~as a~morphism
$(p)\to (p-1)$
such that $\psi_{p,i}\big(x^k_p\big)=x_{p-1}^k$ for $k<i+1$,
and $\psi_{p,i}\big(x^k_p\big)=x^{k-1}_{p-1}$ for $k\ge i+1$.
As in the case of~$\{\phi_{p-1,i}\}_{0\le i \le p}$,
these maps give rise to $p$ face maps $\mathbb{A}^{\wedge p}\to \mathbb{A}^{\wedge p+1}$ given by~the unit $\RR \to \mathbb{A}$.
Consider the rotation $r_{p}\colon S^1\to S^1$
which sends $x^{k}_p$ to $x^{k+1}_{p}$ for $k\in \ZZ/(p+1)\ZZ$.
We~regard $r_p$ as an~isomorphism $(p)\to (p)$.
It~yields the action of~$\ZZ/(p+1)\ZZ$ on $\mathbb{A}^{\wedge p+1}$ given by~the cyclic permutation of~factors.
It~is straightforward to check that these maps constitute a~cyclic object
that coincides with the cyclic object obtained from $B\mathbb{A}$ in
Definition~\ref{Mitchell}.
\end{proof}

\begin{proof}[Proof~of~Proposition~\ref{identificationprop}]
Taking into account Lemma~\ref{observationmap} and
$\SPS(\RR)^{c}\big[W^{-1}\big]^\otimes\simeq \Mod_R^\otimes$,
for an~$\etwo$-algebra $A$, we consider
the image of~$A$ under
\begin{gather*}
h\colon\ \Alg_{\etwo}(\Mod_R)\simeq \Alg_{\assoc}\big(\Fun^\otimes\big(\widetilde{\assoc}^\otimes,\Mod_R^\otimes\big) \big) \to \Alg_{\assoc}\big(\Fun\big(\Lambda^{\rm op},\Mod_R\big)\big),
\end{gather*}
where the right functor is induced by~the composition with $\xi\colon \Lambda^{\rm op}\to \widetilde{\assoc}^\otimes$.
We~abuse notation by~writing $\HH^{\Lambda}_\bullet(A)$ for the image of~$A$
under $h$.
In~the following discussion, we will use the canonical identification
$\Alg_{\eone}^\otimes(-)\simeq \Alg_{\assoc}^\otimes(-)$ which comes from the canonical equivalence \mbox{$\eone^\otimes\simeq \assoc^\otimes$} of~$\infty$-operads.
Let us consider
\begin{gather*}
\Alg_{\etwo}(\Mod_R)\simeq \Alg_{\eone}\big(\Alg_{\assoc}^\otimes(\Mod_R)\big) \simeq \Alg_{\assoc}\big(\Alg_{\eone}^\otimes(\Mod_R)\big)
\\ \hphantom{\Alg_{\etwo}(\Mod_R)}
{}\stackrel{\sim}{\leftarrow} \Fun^\otimes \big(\widetilde{\assoc}^\otimes,\Alg^\otimes_{\eone}(\Mod_R)\big)
\to \Fun \big(\Lambda^{\rm op},\Alg_{\eone}(\Mod_R)\big).
\end{gather*}
The second equivalence follows from Construction~\ref{symfunctorsym},
and the third functor is induced by~$\xi\colon$ $\Lambda^{\rm op}\to \widetilde{\assoc}^\otimes$.
The composition is identified with $h$ via the equivalence
\begin{gather*}
\Alg_{\eone}\big(\Fun\big(\Lambda^{\rm op},\Mod_R\big)\big)\simeq \Fun \big(\Lambda^{\rm op},\Alg_{\eone}(\Mod_R)\big).
\end{gather*}
Let $A_\flat\colon \assoc^\otimes \to \Alg^\otimes_{\eone}(\Mod_R)$
be a~map of~$\infty$-operads that corresponds to $A\in \Alg_{\etwo}(\Mod_R)\simeq \Alg_{\assoc}\big(\Alg_{\eone}\big(\Mod_R^\otimes\big)\big)$.
We~let $\widetilde{A}_\flat\colon \widetilde{\assoc}^\otimes \to \Alg^\otimes_{\eone}(\Mod_R)$ be a~symmetric monoidal functor from the symmetric monoidal envelope $\widetilde{\assoc}^\otimes$ that corresponds to $A_\flat$ (namely, the composite $\assoc^\otimes \to \widetilde{\assoc}^\otimes \to \Alg^\otimes_{\eone}(\Mod_R)$ is equivalent to $A_\flat$).
Observe that the composite $\Lambda^{\rm op}\stackrel{\xi}{\to} \widetilde{\assoc}^\otimes\stackrel{\widetilde{A}_\flat}{\to}\Alg^\otimes_{\eone}(\Mod_R)$
gives rise to a~functor $\Lambda^{\rm op}\to \Alg_{\eone}(\Mod_R)$
which is equivalent to $\HH^{\Lambda}_\bullet(A)$ in $\Fun\big(\Lambda^{\rm op},\Alg_{\eone}(\Mod_R)\big)\simeq \Alg_{\assoc}\big(\Fun\big(\Lambda^{\rm op},\Mod_R\big)\big)$.
Note that $\HH_\bullet(A)$ is defined to be the image of~$\HH^{\Lambda}_\bullet(A)$ under the functor $\Alg_{\assoc}\big(\Fun\big(\Lambda^{\rm op},\Mod_R\big)\big)\to \Alg_{\assoc}\big(\Fun\big(BS^1,\Mod_R\big)\big)$ induced by~the symmetric monoidal functor
$L\colon \Fun\big(\Lambda^{\rm op},\Mod_R\big)\to\Fun\big(BS^1,\Mod_R\big)$ in Lem\-ma~\ref{symmonoloc}.
Here $L$ is a~left adjoint of~the symmetric monoidal
functor $\Fun\big(BS^1,\Mod_R\big) \to\Fun\big(\Lambda^{\rm op},\Mod_R\big)$ induced
by~the composition with $\Lambda^{\rm op}\to BS^1$.
Thus, $\HH_\bullet(A)$ can also be regarded as the image under the (left adjoint) functor $\Fun\big(\Lambda^{\rm op},\Alg_{\eone}(\Mod_R)\big)\to \Fun\big(BS^1,\allowbreak\Alg_{\eone}(\Mod_R)\big)$ given by~left Kan extensions along $\Lambda^{\rm op}\to BS^1$.
Consequently, $\HH_\bullet(A)\colon BS^1\to \Alg_{\eone}(\Mod_R)$ is
a left Kan extension of~$\Lambda^{\rm op}\stackrel{\xi}{\to}\widetilde{\assoc}^\otimes\stackrel{\widetilde{A}_\flat}{\to} \Alg_{\eone}(\Mod_R)$
along $\Lambda^{\rm op}\to BS^1$.
Next, we
let $\widehat{A}_\flat\colon \mathbf{Mfld}_1\to \Alg^\otimes_{\eone}(\Mod_R)$ be an~operadic left Kan extension of~$A_\flat\colon \eone^\otimes\simeq \assoc^\otimes \to \Alg^\otimes_{\eone}(\Mod_R)$.
Let $A'_\flat\colon \Mfld^\otimes_1\to \Alg^\otimes_{\eone}(\Mod_R)$
be a~symmetric monoidal functor which corresponds to $\widehat{A}_\flat$.
The composite
$\Disk_1^\otimes\to \Mfld^\otimes_1\to \Alg^\otimes_{\eone}(\Mod_R)$
is equivalent to $\widetilde{A}_\flat\colon \Disk_1^\otimes\simeq \widetilde{\assoc}^\otimes\to \Alg^\otimes_{\eone}(\Mod_R)$.
Consider the diagram of~$\infty$-categories:
\begin{gather*}
\xymatrix{
\Disk_1/\lsr\ar[r] \ar[d] & \Disk_1 \ar[d] \ar[r] & \Alg_{\eone}(\Mod_R) \\
\lsr \ar[r] & \Mfld_1. \ar[ru] &
}
\end{gather*}
The upper left horizontal arrow is induced by~the restriction to the source.
The left vertical arrow is induced by~the restriction to the target.
The upper right arrow is the underlying functor of~$\widetilde{A}_\flat$.
The arrow $\Mfld_1\to \Alg_{\eone}(\Mod_R)$ is the underlying functor
of~$A'_\flat$. The right triangle commutes whereas the left square does not
commute (but it admits a~canonical natural transformation
induced by~the evaluation map $\Delta^1\times\Fun\big(\Delta^1,\Mfld_1\big)\to \Mfld_1$).
The functor $\Mfld_1\to \Alg_{\eone}(\Mod_R)$
carries $S^1$ to $\colim_{U\to S^1\in (\Disk_1)_{/S^1}}\widetilde{A}_{\flat}(U)$
which means a~colimit of~$\eone^\otimes\times_{\mathbf{Mfld}_1}\big(\mathbf{Mfld}^{\textup{act}}_1\big)_{/S^1} \simeq (\Disk_1)_{/S^1}\to \Disk_1\stackrel{\widetilde{A}_{\flat}}{\to} \Alg_{\eone}(\Mod_R)$.
By Lemma~\ref{redcolimit}, the composite
$\lsr\to \Alg_{\eone}(\Mod_R)$
is a~left Kan extension of~$\Lambda^{\rm op}\simeq\Diskd_1/\lsr \to \Alg_{\eone}(\Mod_R)$.
Since $\Lambda^{\rm op}\simeq\Diskd_1/\lsr \to \lsr\simeq BS^1$ is a~groupoid
completion by~Lemma~\ref{completion},
it follows that the composite
$BS^1\simeq \lsr\to \Alg_{\eone}(\Mod_R)$
is equivalent to $\HH_\bullet(A)$.
In~other words, $\HH_\bullet(A)$ is equivalent to $BS^1\simeq \lsr\hookrightarrow \mathbf{Mfld}_1\stackrel{\widehat{A}_\flat}{\longrightarrow} \Alg_{\eone}(\Mod_R)$.

Next, we relate $i_!(A)_C$ with
$\HH_\bullet(A)\colon BS^1\stackrel{\iota}{\to} \mathbf{Mfld}_1\stackrel{\widehat{A}_\flat}{\longrightarrow} \Alg_{\eone}(\Mod_R)$.
For this purpose, we consider the following setting.
Let $\dcyl \to \widetilde{\dcyl}$ be a~symmetric monoidal
envelope of~$\dcyl$.
Composing with maps into symmetric monoidal envelopes,
we have the left diagram
\begin{gather*}
\xymatrix{
\eone^\otimes\times \eone^\otimes \ar[r] \ar[d] & \widetilde{\mathbf{E}}_2^\otimes \ar[d] & & \eone^\otimes \ar[r] \ar[d] & \Alg_{\eone}^\otimes\big(\widetilde{\mathbf{E}}_2^\otimes\big) \ar[d] \\
\eone^\otimes \times \mathbf{Mfld}_1 \ar[r] & \widetilde{\dcyl}, & & \mathbf{Mfld}_1 \ar[r] & \Alg_{\eone}^\otimes\big(\widetilde{\dcyl}\big)
}
\end{gather*}
lying over $\wedge\colon \Gamma\times \Gamma\to \Gamma$.
Then by~the universal property of~the tensor product
of~$\infty$-operads, it induces the right commutative diagram consisting of~maps of~$\infty$-operads
over $\Gamma$, where $\Alg_{\eone}^\otimes\big(\widetilde{\mathbf{E}}_2^\otimes\big)$
and $\Alg_{\eone}^\otimes\big(\widetilde{\dcyl}\big)$ are symmetric monoidal
$\infty$-categories (defined over $\Gamma$), and the right vertical arrow
is a~symmetric monoidal (fully faithful) functor.
In~the following discussion, we replace $\Mod_R^\otimes$
by~an arbitrary symmetric monoidal presentable $\infty$-category
$\MMM^\otimes$
whose tensor product $\MMM\times \MMM\to \MMM$
preserves small colimits separately in each variable.
The example of~$\MMM^\otimes$ we keep in mind is $\Mod_R^\otimes$.
Let $A$ be an~$\etwo$-algebra object in $\MMM^\otimes$, that is,
a map $A\colon \etwo^\otimes\to \MMM^\otimes$ of~$\infty$-operads over $\Gamma$.
The inclusion $i\colon \etwo^\otimes\hookrightarrow \dcyl$
gives rise to the adjoint pair $i_!\colon \Alg_{\etwo}(\MMM)\rightleftarrows \Alg_{\dcyl}(\MMM)\,{\colon}\!i^*$.
Let $\widetilde{A}\colon \widetilde{\mathbf{E}}_2^\otimes\to \MMM^\otimes$
and $\widetilde{i_!(A)}\colon \widetilde{\dcyl}\to \MMM^\otimes$
be symmetric monoidal functors that correspond to $A$ and $i_!(A)$, respectively.
We~have the diagram of~$\infty$-operads
\begin{gather*}
\xymatrix{
\eone^\otimes \ar[r] \ar[d] & \Alg^\otimes_{\eone}\big(\widetilde{\mathbf{E}}_2^\otimes\big) \ar[d] \ar[r]^{\Alg_{\eone}(\widetilde{A})} & \Alg^\otimes_{\eone}\big(\MMM^\otimes\big) \\
\mathbf{Mfld}_1 \ar[r] & \Alg^\otimes_{\eone}\big(\widetilde{\dcyl}\big) \ar[ru]_{\Alg_{\eone} (\widetilde{i_!(A)} )}. &
}
\end{gather*}
As before,
we let $A_\flat\colon \eone^\otimes \to \Alg_{\eone}\big(\MMM^\otimes\big)$ be the composite
of~top horizontal arrows,
which amounts to $A\colon \etwo^\otimes\to \MMM^\otimes$ since $\eone^\otimes\otimes\eone^\otimes\simeq \etwo^\otimes$.
Let $\hat{A}_{\flat}\colon \mathbf{Mfld}_1\to \Alg_{\eone}^\otimes\big(\MMM^\otimes\big)$
be the operadic left Kan extension of~$A_{\flat}$ along $\eone^\otimes\hookrightarrow \mathbf{Mfld}_1$.
Let $\hat{A}_\sharp\colon \mathbf{Mfld}_1\to \Alg_{\eone}^\otimes\big(\MMM^\otimes\big)$ be
the composite.
We~note that $i_!(A)_C$ is equivalent to $BS^1\to \Alg_{\eone}\big(\MMM^\otimes\big)$ determined by~the composite $BS^1\simeq \lsr\stackrel{\iota}{\hookrightarrow} \mathbf{Mfld}_1\stackrel{\hat{A}_\sharp}{\to} \Alg_{\eone}^\otimes\big(\MMM^\otimes\big)$.
The universal property \cite[Proposition~3.1.3.2]{HA} induces
a canonical morphism $\hat{A_{\flat}}\to \hat{A}_\sharp$.
It~suffices to
prove that the restriction $\hat{A_{\flat}}|_{\langle S^1\rangle}\to \hat{A}_{\sharp}|_{\langle S^1\rangle}$ to $\lsr$ is an~equivalence. (It~gives rise to an~equivalence
$\HH_\bullet(A)\simeq i_!(A)_C$
in $\Fun\big(BS^1,\Alg_{\eone}(\Mod_R)\big)$.)
To this end, it is enough to show the following lemma, which completes
the proof of~Proposition~\ref{identificationprop}.
\end{proof}

\begin{Lemma}\label{cofinalitylemma}
The induced morphism $\hat{A}_\flat\big(S^1\big) \to \hat{A}_\sharp\big(S^1\big)$ is an~equivalence
in the $\infty$-category $\Alg_{\eone}\big(\MMM^\otimes\big)$.
\end{Lemma}

\begin{proof}We~first consider $\hat{A}_\flat\colon \mathbf{Mfld}_1\to \Alg^\otimes_{\eone}\big(\MMM^\otimes\big)$.
The operadic left Kan extension gives
$\hat{A}_\flat\big(S^1\big)=\colim_{U\to S^1\in (\Disk_1)_{/S^1}}\widetilde{A}_\flat(U)$.
Here, $\colim_{U\to S^1\in (\Disk_1)_{/S^1}}\widetilde{A}_\flat(U)$
means a~colimit of~$\eone^\otimes\times$ ${}_{\mathbf{Mfld}_1}(\mathbf{Mfld}_1)^{\textup{act}}_{/S^1}\simeq (\Disk_1)_{/S^1}\to \Disk_1\stackrel{\widetilde{A}_\flat}{\to} \Alg_{\eone}\big(\MMM^\otimes\big)$.
By \cite[Corollary~3.22]{AF} or~\cite[Pro\-po\-si\-tion~5.5.2.15]{HA}, $(\Disk_1)_{/S^1}$ is sifted
(strictly speaking, in the statement in~\cite[Proposition~5.5.2.15]{HA},
mapping spaces between disks are spaces of~(not necessarily rectilinear)
open embeddings, but the overcategory $(\Disk_1)_{/S^1}$ is equivalent to
a nonrectilinear version in {\it loc.~cit.}).
The forgetful functor $\Alg_{\eone}\big(\MMM^\otimes\big)\to \MMM$
preserves sifted colimits.
Consequently, the image of~$\hat{A}_\flat\big(S^1\big)$ in $\MMM$
is a~colimit of~$(\Disk_1)_{/S^1}\to \Alg_{\eone}\big(\MMM^\otimes\big)\to \MMM$.
Given a~topological $r$-manifold $T$,
we let $\textup{Disj}(T)$ be the poset that consists of~open sets $U\subset T$
such that $U$ is homeomorphic to a~finite disjoint union of~$(0,1)^r$.
We~think of~$\textup{Disj}(T)$ as a~category.
Then
according to~\cite[Proposition~5.5.2.13]{HA},
the natural functor $\textup{Disj}\big(S^1\big)\to (\Disk_1)_{/S^1}$ is left cofinal.
Thus, $\colim_{U\in \textup{Disj}(S^1)}\widetilde{A}_\flat(U) \simeq \colim_{U\to S^1\in (\Disk_1)_{/S^1}}\widetilde{A}_\flat(U)$ in $\MMM$, where
$\colim_{U\in \textup{Disj}(S^1)}A_\flat(U)$ is a~co\-limit of~$\textup{Disj}\big(S^1\big)\to (\Disk_1)_{/S^1}\to \MMM$
(this equivalence also follows from the fact that
$(\Disk_1)_{/S^1}$ is obtained from $\textup{Disj}\big(S^1\big)$ by~localizing
with respect to those inclusions $U\subset U'$ that are isotopic to
an isomorphism \cite[Proposition~2.19]{AF}).

Next, we consider the image of~$\hat{A}_\sharp\big(S^1\big)$ in $\MMM$.
If~$\Disk_2$ denotes the underlying $\infty$-category of~$\Disk_2^\otimes$
and we set $C=(0,1)\times S^1$, then
the image of~$C$ in $\MMM$ under the operadic left Kan extension
$i_!(A)$ is $\colim_{V\to C\in (\Disk_2)_{/C}}A(V)$, that is,
a colimit of~$\etwo^\otimes\times_{\dcyl}(\dcyl)^{\textup{act}}_{/C}\simeq (\Disk_2)_{/C}\to \Disk_2\stackrel{\widetilde{A}}{\to} \MMM$, where the latter functor
is the underlying functor of~the induced symmetric monoidal functor
$\widetilde{A}\colon \Disk_2^\otimes\to \MMM^\otimes$ (we abuse notation).

Let $\textup{Disj}^{\rm rec}(C)$ be
the full subcategory (poset) of~$\textup{Disj}(C)$ spanned by~those open sets $V\subset C$
such that $V$ is the image of~a~rectilinear embedding,
and the composite $V \hookrightarrow C=(0,1)\times S^1$ $\stackrel{\textup{pr}}{\to} S^1$
is not surjective.
By applying the argument of~\cite[Proposition~5.5.2.13]{HA}
to $\textup{Disj}^{\rm rec}(C)\to (\Disk_2)_{/C}$,
we see that $\textup{Disj}^{\rm rec}(C)\to (\Disk_2)_{/C}$ is left cofinal.
It~follows that there
is a~canonical equivalence $\colim_{V\in \textup{Disj}^{\rm rec}(C)}\widetilde{A}(V)\simeq \colim_{V\to C\in (\Disk_2)_{/C}}\widetilde{A}(V)$.

To prove that $\hat{A}_\flat\big(S^1\big) \to \hat{A}_\sharp\big(S^1\big)$ is an~equivalence,
it is enough to show that
\begin{gather*}
\colim_{U\in \textup{Disj}(S^1)}\widetilde{A}_\flat(U)\to \colim_{V\in \textup{Disj}^{\rm rec}(C)}\widetilde{A}(V)
\end{gather*}
is an~equivalence in $\MMM$.
Unwinding the definition,
this morphism is the composite of~\begin{gather*}
\colim_{U\in \textup{Disj}(S^1)}\widetilde{A}_\flat(U)\to \colim_{(0,1)\times U \in \textup{Disj}^{\rm rec}(C)}\widetilde{A}((0,1)\times U)\to \colim_{V\in \textup{Disj}^{\rm rec}(C)}\widetilde{A}(V),
\end{gather*}
where the right arrow
is induced by~the universal property of~the colimit,
and the left arrow is an~equivalence because $\widetilde{A}_\flat(U)\simeq \widetilde{A}((0,1)\times U)$.
To see that the right arrow is an~equivalence,
it~will suffice to prove that $\textup{Disj}\big(S^1\big)\to \textup{Disj}^{\rm rec}(C)$
that sends $U$ to $(0,1)\times U$ is left cofinal:
for any $V\in \textup{Disj}^{\rm rec}(C)$, the category
$\textup{Disj}\big(S^1\big)\times_{\textup{Disj}^{\rm rec}(C)} \textup{Disj}^{\rm rec}(C)_{V/}$
is weakly contractible.
Consider the image $W$ of~$V$ under the projection
$(0,1)\times S^1\to S^1$. Then $W$ belongs to
$\textup{Disj}\big(S^1\big)$ since
$V \hookrightarrow C\to S^1$ is not surjective. It~follows that $\textup{Disj}\big(S^1\big)\times_{\textup{Disj}^{\rm rec}(C)} \textup{Disj}^{\rm rec}(C)_{V/}$ has an~initial object
so that the opposite category is filtered. Thus, by~\cite[Proposition~5.5.8.7]{HTT},
$\textup{Disj}\big(S^1\big)\times_{\textup{Disj}^{\rm rec}(C)} \textup{Disj}^{\rm rec}(C)_{V/}$
is weakly contractible as desired.
\end{proof}

\begin{Theorem}\label{mainconstruction}
Let $R$ be a~commutative ring spectrum.
Let $\CCC$ be a~small $R$-linear stable idem\-potent-complete
$\infty$-category, that is, an~object of~$\ST_R=\Mod_{\Perf_R}(\ST)$.
Let $\HH^\bullet(\CCC)$ be the Hoch\-schild cohomology $R$-module spectrum
which belongs to $\Alg_{\etwo}\!(\Mod_R)$, see Definitions~{\rm \ref{cohomologydef}}~and~{\rm \ref{cohomologydefsmall}}.
Let $\HH_\bullet(\CCC)$ be the Hoch\-schild homology $R$-module
spectrum which lies in $\Fun\big(BS^1,\Mod_R\big)$, see Definition~{\rm \ref{homologydef}}.
Then $(\HH^\bullet(\CCC),\HH_\bullet(\CCC))$ is promoted to
an object of~$\Alg_{\KS}(\Mod_R)$ in a~natural way.
\end{Theorem}

\begin{Remark}By Corollary~\ref{algebraequivalence}, we have
\begin{gather*}
\Alg_{\KS}(\Mod_R) \simeq \Alg_{\etwo}(\Mod_R)\times_{\Alg_{\eone}(\Fun(BS^1,\Mod_R))}\LMod\big(\Fun\big(BS^1,\Mod_R\big)\big) \\
\hphantom{\Alg_{\KS}(\Mod_R)}{} \to \Alg_{\etwo}(\Mod_R)\times \Fun\big(BS^1,\Mod_R\big).
\end{gather*}
Theorem~\ref{mainconstruction} means that we can obtain
an object of~$\Alg_{\KS}(\Mod_R)$ which ``lies over''
the pair $(\HH^\bullet(\CCC),\HH_\bullet(\CCC))\in \Alg_{\etwo}(\Mod_R)\times \Fun\big(BS^1,\Mod_R\big)$.
\end{Remark}

The proof~proceeds in Construction~\ref{constmodule}, Proposition~\ref{smallleftmodule} and Construction~\ref{constalg}.

\begin{Construction}\label{constmodule}
We~write $\DDD$ for the $\Ind$-category
$\Ind(\CCC)$, which is an~$R$-linear compactly generated
stable $\infty$-category.
Let $\HH^\bullet(\CCC)=\HH^\bullet(\DDD)$ be the Hoch\-schild cohomology $R$-module
spectrum.
Recall that $\HH^\bullet(\DDD)=E(\operatorname{\mathcal{E}{\rm nd}}_R(\DDD))$.
The counit map of~the adjunction
$I\colon \Alg_{\etwo}(\Mod_R)\rightleftarrows \Alg_{\assoc} \big(\PR_R\big){\colon}\!E$
induces to an~associative monoidal functor
\begin{gather*}
\RMod_{\HH^\bullet(\CCC)}^\otimes=\RMod_{\HH^\bullet(\DDD)}^\otimes \to \operatorname{\mathcal{E}{\rm nd}}_R^\otimes(\DDD),
\end{gather*}
that is, a~morphism in $\Alg_{\assoc}\big(\PR_R\big)$,
where $\RMod_{\HH^\bullet(\CCC)}^\otimes:=\RMod_{\HH^\bullet(\CCC)}^\otimes(\Mod_R)$ is the associative monoidal $\infty$-category of~right modules of~$\HH^\bullet(\CCC)$.
According to~\cite[Corollary~4.7.1.40]{HA},
the associative monoidal $\infty$-category
$\operatorname{\mathcal{E}{\rm nd}}_R^\otimes(\DDD)$
naturally acts on $\DDD$. More precisely,
$\DDD$ is a~left module of~$\operatorname{\mathcal{E}{\rm nd}}_R^\otimes(\DDD)$,
that is, an~object of~$\LMod_{\operatorname{\mathcal{E}{\rm nd}}_R^\otimes(\DDD)}\big(\PR_R\big)$ (this action is universal in an~appropriate sense, cf.~\cite[Section~4.7.1]{HA}).
Then the associative monoidal functor $\RMod_{\HH^\bullet(\CCC)}^\otimes\to \operatorname{\mathcal{E}{\rm nd}}_R^\otimes(\DDD)$
induces a~left $\RMod_{\HH^\bullet(\CCC)}^\otimes$-module
structure on $\DDD$. That is, $\DDD$ is promoted to
 an~object of~$\LMod_{\RMod_{\HH^\bullet(\CCC)}^\otimes}\big(\PR_R\big)$.

Let $\RPerf_{\HH^\bullet(\CCC)}$ be the full subcategory
of~$\RMod_{\HH^\bullet(\CCC)}$ spanned by~compact objects.
This subcategory is the smallest stable subcategory which
contains $\HH^\bullet(\CCC)$ (regarded as a~right module)
and is closed under retracts.
Hence $\RPerf_{\HH^\bullet(\CCC)}$ inherits an~associative monoidal structure
from the structure on $\RMod_{\HH^\bullet(\CCC)}$. We~denote by~$\RPerf^\otimes_{\HH^\bullet(\CCC)}$ the resulting associative
monoidal small $R$-linear stable idempotent-complete
$\infty$-category which we regard as an~object of~$\Alg_{\assoc}(\ST_R)$.
\end{Construction}

\begin{Proposition}
\label{smallleftmodule}
We~continue to assume
that $\CCC$ is a~small $R$-linear stable idempotent-complete
$\infty$-category.
If~we think of~$\DDD=\Ind(\CCC)$
as the left $\RMod_{\HH^\bullet(\CCC)}^\otimes$-module
$($as above$)$,
the restriction exhibits
$\CCC$ as a~left $\RPerf^\otimes_{\HH^\bullet(\CCC)}$-module,
that is, an~object of~$\LMod_{\RPerf^\otimes_{\HH^\bullet(\CCC)}}(\ST_R)$.
In~parti\-cu\-lar, $\CCC$ is promoted to
an object of~$\LMod_{\RPerf^\otimes_{\HH^\bullet(\CCC)}}(\ST_R)$.
\end{Proposition}

\begin{proof}
We~may and will suppose that $\CCC$ is the full subcategory of~compact objects
in $\DDD$.
The tensor product functor
\begin{gather*}
\RMod_{\HH^\bullet(\CCC)}\times \RMod_{\HH^\bullet(\CCC)}\to \RMod_{\HH^\bullet(\CCC)}
\end{gather*}
sends
$\RPerf_{\HH^\bullet(\CCC)} \times \RPerf_{\HH^\bullet(\CCC)}$ to $\RPerf_{\HH^\bullet(\CCC)}\subset \RMod_{\HH^\bullet(\CCC)}$.
It~will suffice to prove that
the action functor
\begin{gather*}
m\colon\ \RMod_{\HH^\bullet(\CCC)}\times\DDD\to \DDD
\end{gather*}
sends $\RPerf_{\HH^\bullet(\CCC)}\times \CCC$ to $\CCC$.
Let $\PPP$ be the full subcategory of~$\RMod_{\HH^\bullet(\CCC)}$
spanned by~those objects $P$ such that the essential
image of~$\{P\}\times \CCC$
is contained in $\CCC$.
Note that $m$ preserves the shift functors ($\Sigma$ or $\Omega$)
and small colimits separately in each variable.
Moreover, the stable subcategory $\CCC\subset \DDD$ is closed
under retracts.
Thus, we see that $\PPP$ is a~stable subcategory which is closed under
retracts.
Since $\HH^\bullet(\CCC)$ is a~unit object,
$\HH^\bullet(\CCC)$ lies in $\PPP$.
Keep in mind that $\RPerf_{\HH^\bullet(\CCC)}$ is
the smallest stable subcategory which
contains $\HH^\bullet(\CCC)$
and is closed under retracts.
It~follows that $\RPerf_{\HH^\bullet(\CCC)}\subset \PPP$.
\end{proof}

\begin{Construction}\label{constalg}
Take $\OO$ to be $\LM$ in Proposition~\ref{hochschildhomologyfunctor}.
By definition, $\LMod(\Mod_{\Perf_R}(\ST))$
is $\Alg_{\LM}(\Mod_{\Perf_R}(\ST))$, and $\LMod\big(\Fun\big(BS^1,\Mod_R\big)\big)=\Alg_{\LM}\big(\Fun\big(BS^1,\Mod_R\big)\big)$.
We then have
\begin{gather*}
\xymatrix{
\LMod(\Mod_{\Perf_R}(\ST))\ar[r] \ar[d] & \LMod\big(\Fun\big(BS^1,\Mod_R\big)\big) \ar[d] \\
\Alg_{\assoc}(\Mod_{\Perf_R}(\ST)) \ar[r] & \Alg_{\assoc}\big(\Fun\big(BS^1,\Mod_R\big)\big),
}
\end{gather*}
where the vertical arrows are given by~the restriction along
$\assoc^\otimes \hookrightarrow \LM$.
By Proposition~\ref{smallleftmodule},
we think of~$\CCC$ as an~object of~$\LMod_{\RPerf^\otimes_{\HH^\bullet(\CCC)}}(\Mod_{\Perf_R}(\ST))$.
Applying the above functor to $\CCC$, we obtain
$\HH_\bullet(\CCC)$ which belongs to $\LMod_{\HH_\bullet(\RPerf_{\HH^\bullet(\CCC)})}\big(\Fun\big(BS^1,\Mod_R\big)\big)$.
The lower horizontal arrow carries the associative monoidal $\infty$-category
$\RPerf^\otimes_{\HH^\bullet(\CCC)}$
to an~object of~$\Alg_{\assoc}\big(\Fun\big(BS^1,\Mod_R\big)\big)$.
That is,
$\HH_\bullet(\RPerf_{\HH^\bullet(\CCC)})$ is
an associative algebra object
in $\Fun\big(BS^1,\Mod_R\big)$.
Consequently, the Hoch\-schild homology $R$-module spectrum
$\HH_\bullet(\CCC)$ is a~left $\HH_\bullet(\RPerf_{\HH^\bullet(\CCC)})$-module object in $\Fun\big(BS^1,\Mod_R\big)$.
Next, we set $A=\HH^\bullet(\CCC)$ in $\Alg_{\etwo}(\Mod_R)$.
By the invariance of~Hoch\-schild homology under Morita equivalences
(cf.~Lem\-ma~\ref{Moritainvert}),
$\HH_\bullet(A)\simeq \HH_\bullet(\RPerf_{\HH^\bullet(\CCC)})$ in
$\Alg_{\assoc}\big(\Fun\big(BS^1,\Mod_R\big)\big)$.
Let
\begin{gather*}
\Alg_{\etwo}(\Mod_R)\stackrel{i_!}\simeq \Alg_{\dcyl}^{\mathbf{D}}(\Mod_R)\to \Alg_{\cyl}(\Mod_R)\simeq \Alg_{\eone}\big(\Fun\big(BS^1,\Mod_R\big)\big)
\end{gather*}
be a~sequence of~functors such that the first one is induced by~left Kan extensions along $i\colon \etwo^\otimes \hookrightarrow \dcyl$
(cf.~the discussion before Proposition~\ref{another}),
the second one is the restriction along $\cyl \to \dcyl$,
and the third functor (equivalence) comes from Corollary~\ref{equivariant}.
By definition, the image of~$A$ in $\Alg_{\eone}\big(\Fun\big(BS^1,\Mod_R\big)\big)$
is $i_!(A)_C$ defined in the discussion before Proposition~\ref{identificationprop}. According to Proposition~\ref{identificationprop},
we have the canonical equivalence $\HH_\bullet(A)\simeq i_!(A)_C$
in $\Alg_{\eone}\big(\Fun\big(BS^1,\Mod_R\big)\big)\simeq \Alg_{\assoc}\big(\Fun\big(BS^1,\Mod_R\big)\big)$.
Therefore, $\HH^\bullet(\CCC)=A\in \Alg_{\etwo}(\Mod_R)$
and the left $\HH_\bullet(A)$-module $\HH_\bullet(\CCC)$
together with $i_!(A)_C\simeq \HH_\bullet(A)$ determines
an object of~\begin{gather*}
\Alg_{\etwo}(\Mod_{R})\times_{\Alg_{\assoc}(\Fun(BS^1,\Mod_R))}\LMod\big(\Fun\big(BS^1,\Mod_R\big)\big)\simeq \Alg_{\KS}(\Mod_R),
\end{gather*}
where the equivalence comes from
the canonical equivalences in Corollary~\ref{algebraequivalence}.
In~other words,
it defines an~object of~$\Alg_{\dcylm}^{\mathbf{D}}(\Mod_R)\subset \Alg_{\dcylm}(\Mod_R)$ which induces a~$\KS$-algebra via the restriction.
Thus, we obtain the desired object of~$\Alg_{\KS}(\Mod_R)$.
\end{Construction}

\section{The action}
\label{actionsec}

In~this section, we study the maps induced by~the action of~Hoch\-schild
cohomology spectrum $\HH^\bullet(\CCC)$ on $\HH_\bullet(\CCC)$, constructed in Theorem~\ref{mainconstruction}.

{\bf 8.1.} Let $R$ be a~commutative ring spectrum.
We~let $\CCC$ be a~small stable $R$-linear $\infty$-category.
We~let $A\in \Alg_{\assoc}(\Mod_R)$ and
suppose that $\CCC=\RPerf_A$.
In~other words, we assume that $\Ind(\CCC)$ admits a~single compact
generator.
In~this setting, we can describe morphisms induced by~module actions
by~means of~concrete algebraic constructions.
For ease of~notation, we write $\HHC(A)$
for $\HHC(\RPerf_A)=\HHC(\RMod_A)$.
We~can safely confuse $\HHH(\RMod_A)$
with the Hoch\-schild homology $R$-module spectrum $\HHH(A)$ of~$A$
because of~the invariance under Morita equivalences.
We~do not distinguish between the notation $\HHH(\RMod_A)$ and $\HHH(A)$: we write $\HHH(A)$ for $\HHH(\RMod_A)$ as well.
Write $A^e$ for $A^{\rm op}\otimes_R A$.
As before, by~$\otimes$ we mean the tensor product over $R$
when we treat the tensor products of~objects in $\Mod_R$ or $\Alg_{\assoc}(\Mod_R)$.

We~define a~morphism $\HHC(A)\otimes \HHH(A)\to \HHH(A)$
which we refer to as the contraction morphism:

\begin{Definition}
Consider the functor
\begin{gather*}
(-)\otimes_{A^e}(-)\colon\ \RMod_{A^e}\times \LMod_{A^e}\to \Mod_R
\end{gather*}
which is informally given by~the two-sided
bar construction $(P,Q) \mapsto P\otimes_{A^e}Q$.
Note that $\RMod_{A^e}$ is left-tensored over $\Mod_R^\otimes$.
If~we regard $\HHC(A)$ as an~object of~$\Mod_R$,
there is a~morphism $\HHC(A)\otimes A \to A$ in $\RMod_{A^e}$,
which exhibits $\HHC(A)$ as a~morphism object from $A$ to itself
in $\RMod_{A^e}$
(i.e., hom $R$-module),
see Corollary~\ref{centercor}.
Let $(\HHC(A)\otimes A)\otimes_{A^e}A\to A\otimes_{A^e}A$
be the morphism induced by~the morphism $\HHC(A)\otimes A\to A$ in $\RMod_{A^e}$.
We~identify $\HH_\bullet(A)$ with $A\otimes_{A^e}A$ in $\Mod_R$,
see Lemma~\ref{twosidedbar}. It~gives rise to
\begin{gather*}
\sigma\colon\ \HHC(A)\otimes \HH_\bullet(A) = \HHC(A)\otimes (A\otimes_{A^e}A) \to A\otimes_{A^e}A = \HHH(A).
\end{gather*}
We~shall refer to it as the contraction morphism.
\end{Definition}

We~denote by~$(\HH^\bullet(\CCC),\HH_\bullet(\CCC))$
the pair endowed with the $\KS$-algebra structure
construc\-ted in Theorem~\ref{mainconstruction}: we will think
that the pair is promoted to an~object of~$\Alg_{\KS}(\Mod_R)$.
Let $D$ and $C_M$ be colors in the colored operad $\textup{KS}$.
There is a~class of~an active mor\-phism
$f_j\colon (\langle 2\rangle,D,C_M)\!\to\! (\langle 1\rangle, C_M)$ in $\KS$
lying over the active morphism
$\rho\colon \langle 2\rangle\!\to\! \langle 1\rangle$ (with \mbox{$\rho^{-1}(\ast)\!=\!\ast$}).
Such a~morphism $f_j$ is unique up to equivalences.
This is induced by~an open embedding
$j\colon (0,1)^2\sqcup (0,1)\times S^1\to (0,1)\times S^1$ such that
$j_1\colon (0,1)^2\to (0,1)\times S^1$ is rectilinear
and $j_2\colon (0,1)\times S^1\to (0,1)\times S^1$ is a~shrinking embedding,
cf.~Definition~\ref{recti}.
If~$h\colon \KS\to \Mod_R^\otimes$ denotes a~map of~$\infty$-operads that encodes
$(\HH^\bullet(\CCC),\HH_\bullet(\CCC))$,
passing to $\Mod_R$ via a~coCartesian natural
transformation, the image of~$f_j$ induces a~morphism in $\Mod_R$:
\begin{gather*}
u\colon\ \HH^\bullet(\CCC)\otimes \HH_\bullet(\CCC)\simeq h(D)\otimes h(C_M)\to h(C_M)\simeq \HH_\bullet(\CCC).
\end{gather*}

\begin{Theorem}
\label{maincontraction}
Let $A$ be an~associative algebra in $\Mod_R$ and let $\CCC$ be $\RPerf_A$.
Then $u\colon \HH^\bullet(\CCC)\otimes \HH_\bullet(\CCC)\to \HH_\bullet(\CCC)$
is equivalent to the contraction morphism $\sigma\colon \HH^\bullet(A)\otimes \HH_\bullet(A) \to \HH_\bullet(A)$ as a~morphism in $\Mod_R$.
\end{Theorem}

\begin{Example}
If~$\CCC=\RPerf_A$ has a~Calabi--Yau structure of~dimension $d$,
then there is a~morphism $w\colon \Sigma^d R\to \HHH(A)$ in $\Mod_R$
(which we can think of~as an~analogue of~a global section of~a~volume form)
such that the composite $\Sigma^d \HHC(A)\simeq \Sigma^dR\otimes \HHH(A) \stackrel{w\otimes \textup{id}}{\longrightarrow} \HHC(A)\otimes \HHH(A)\stackrel{\sigma}{\to} \HHH(A)$ is an~equivalence. Here $\Sigma$ indicates the suspension.
It~follows from Theorem~\ref{maincontraction}
that $\Sigma^d \HHC(A)\to \HHH(A)$ induced by~$w$
and $u$ instead of~$\sigma$ also gives an~equivalence.
\end{Example}

\begin{Construction}We~set $\CCC=\RPerf_A$.
Consider a~morphism
\begin{gather*}
e\colon\ \Mod_R\to \mathcal{M}{\rm or}_R(\RMod_A,\RMod_A)
\end{gather*}
in $\PR_R$
which carries $R$ to the identity functor, cf.~Lemma~\ref{linearfunctorcategory}.
Applying the adjunction
\begin{gather*}
I\colon\ \Alg_{\assoc}(\Mod_R)\rightleftarrows \left(\PR_R\right)_{\Mod_R/}\ {\colon}\!E
\end{gather*}
(see Section~\ref{cohomologysec}) to the morphism $\Mod_R\to \mathcal{M}{\rm or}_R(\RMod_A,\RMod_A)$,
we have the morphisms in~$\left(\PR_R\right)_{\Mod_R/}$
\begin{gather*}
\RMod_{\HH^\bullet(A)\otimes A}\simeq \RMod_{\HH^\bullet(A)} \otimes_R \RMod_A
\\ \hphantom{\RMod_{\HH^\bullet(A)\otimes A}}
{}\to \mathcal{M}{\rm or}_R(\RMod_A,\RMod_A)\otimes_R\RMod_A \to \RMod_A,
\end{gather*}
where the right arrow is the canonical morphism,
and the middle arrow is induced by~the counit map $\RMod_{\HH^\bullet(\RMod_A)}\to \mathcal{M}{\rm or}_R(\RMod_A,\RMod_A)$ of~the adjunction.
Here, $\RMod_{\HH^\bullet(A)}$ is endowed with $p_{\HHC(A)}\colon \Mod_R\to \RMod_{\HH^\bullet(A)}$ which carries $R$ to $\HH^\bullet(A)$.
The morphisms from $\Mod_R$ are omitted from the notation.
Recall that $I\colon \Alg_{\assoc}(\Mod_R)\rightarrow \left(\PR_R\right)_{\Mod_R/}$ that
sends $A$ to $p_A\colon \Mod_R\to \RMod_A$ with $p_A(R)=A$ is fully faithful
so that the full subcategory of~$\left(\PR_R\right)_{\Mod_R/}$
spanned by~objects of~the form $p_A\colon \Mod_R\to \RMod_A$ is equivalent to~$\Alg_{\assoc}(\Mod_R)$.
Thus, the composite $\RMod_{\HH^\bullet(A)\otimes A} \to \RMod_A$
in $\left(\PR_R\right)_{\Mod_R/}$ gives rise to a~morphism of~associative algebras
\begin{gather*}
\alpha\colon\ \HH^\bullet(A)\otimes A\to A,
\end{gather*}
that is, a~morphism in $\Alg_{\assoc}(\Mod_R)$).
Since
\begin{gather*}
\RMod_A\simeq \Mod_R\otimes_R\RMod_A\stackrel{e\otimes\textup{id}}{\to} \mathcal{M}{\rm or}_R(\RMod_A,\RMod_A)\otimes_R\RMod_A\to \RMod_A
\end{gather*}
is naturally
equivalent to
the identity functor, we have
a homotopy from the composite $A\to \HH^\bullet(A)\otimes A\stackrel{\alpha}{\to} A$ to the identity morphism of~$A$, where $A\to \HH^\bullet(A)\otimes A$
is induced by~the morphism from
the unit algebra $R\to \HH^\bullet(A)$.
\end{Construction}

We~can make the following observation:

\begin{Lemma}\label{centerlemma}$\HH^\bullet(A)$ is
a center of~$A$. See {\rm \cite[Section~5.3.1]{HA}} or the proof~below for centers and centralizers.
\end{Lemma}

\begin{proof}The statement of~this lemma is slightly imprecise.
Given $B\in \Alg_{\assoc}(\Mod_R)$,
we define $c(B,A)$ to be the fiber product
$\Map(B\otimes A,A)\times_{\Map(A,A)}\{\textup{id}_A\}$, where $\Map(-,-)$ means the mapping space in
$\Alg_{\assoc}(\Mod_R)$, and $\Map(B\otimes A,A)\to\Map(A,A)$
is induced by~the composition with $R\otimes_R A\to B\otimes_RA$.
The assignment $B\mapsto c(B,A)$ induces a~functor
$c(-,A)\colon \Alg_{\assoc}(\Mod_R)\to \SSS$.
A center of~$A$ is $Z\in \Alg_{\assoc}(\Mod_R)$ that
represents $c(-,A)$, that is, a~centralizer of~the identity morphism
$A\to A$ (cf.~\cite{HA}).
Through the equivalence $\Map(Z,Z)\simeq c(Z,A)$,
the identity of~$Z$ determines
$Z\otimes A\to A$ and a~homotopy/equivalence from the composite
$A\to Z\otimes A\to A$ to~$\textup{id}_A$.
Our lemma claims that $\HH^\bullet(A)$ together with $\alpha\colon \HH^\bullet(A)\otimes A\to A$ and the homotopy induces an~equivalence in $\SSS$:
\begin{gather*}
\theta_B\colon\ \Map(B,\HH^\bullet(A))\to \Map(B\otimes A,A)\times_{\Map(A,A)}\{\textup{id}_A\}
\end{gather*}
for any $B\in \Alg_{\assoc}(\Mod_R)$. Here,
$\Map(B,\HH^\bullet(A))\to \Map(B\otimes A,A)$ is the functor
which sends
$f\colon B\to \HH^\bullet(A)$ to the composite $B\otimes A\stackrel{f\otimes \textup{id}_A}{\to} \HH^\bullet(A)\otimes A\stackrel{\alpha}{\to} A$.

We~will prove our claim. Namely, we show that $\theta_B$ is an~equivalence.
For ease of~notation, we set $p_B\colon \Mod_R\to \RMod_B=\MMM_B$ and $p_A\colon \Mod_R\to \RMod_A=\MMM_A$.
Let $X=\Map_{\PR_R}(\MMM_B\otimes_R\MMM_A,\MMM_A)\times_{\MMM_A^{\simeq}}\{A\}$, where $A$ is the right $A$-module determined by~the multiplication of~$A$, and
$\Map_{\PR_R}(\MMM_B\otimes_R\MMM_A,\MMM_A)\to \Map_{\PR_R}(\Mod_R,\MMM_A)\simeq \MMM_A^{\simeq}$ is induced by~the composition
with~$\Mod_R=\Mod_R\otimes_R\Mod_R\stackrel{p_B\otimes p_A}{\to} \MMM_B\otimes_R\MMM_A$.
Similarly, we define $Y=\Map_{\PR_R}(\MMM_A,\MMM_A)\times_{\MMM_A^{\simeq}}\{A\}$.
Note that
\begin{gather*}
X\simeq \Map_{\left(\PR_R\right)_{\Mod_R/}}(\RMod_B\otimes_R\RMod_A,\RMod_A)\simeq \Map(B\otimes A,A).
\end{gather*}
Similarly, $Y\simeq \Map_{\left(\PR_R\right)_{\Mod_R/}}(\RMod_A,\RMod_A)\simeq \Map(A,A)$.
The morphism $\Map_{\PR_R}(\MMM_B\otimes_R\MMM_A,\MMM_A)\to \Map_{\PR_R}(\MMM_A,\MMM_A)$
given by~the composition with $\Mod_R\to \MMM_B$ determines $X\to Y$.
Let $\Delta^0\to Y$ be the map determined by~the identity functor $\MMM_A$
and the identity morphism of~$A$.
We~have a~canonical equivalence
\begin{gather*}
X\times_Y\Delta^0\simeq \Map(B\otimes A,A)\times_{\Map(A,A)}\{\textup{id}_A\}=c(B,A).
\end{gather*}
By the universal property of~$\mathcal{M}{\rm or}_R(\MMM_A,\MMM_A)$,
\begin{gather*}
\Map_{\PR_R}(\MMM_B,\mathcal{M}{\rm or}_R(\MMM_A,\MMM_A))\simeq \Map_{\PR_R}(\MMM_B\otimes_R\MMM_A,\MMM_A).
\end{gather*}
Using this equivalence, we deduce that
\begin{gather*}
X\times_Y\Delta^0\simeq \Map_{\PR_R}(\MMM_B,\mathcal{M}{\rm or}_R(\MMM_A,\MMM_A))\times_{\mathcal{M}{\rm or}_R(\MMM_A,\MMM_A)^{\simeq}}\{\textup{id}_{\MMM_A}\}.
\end{gather*}
Moreover, taking into account the universal property of~$\RMod_{\HH^\bullet(A)}\to \mathcal{M}{\rm or}_R(\MMM_A,\MMM_A)$ in $\left(\PR_R\right)_{\Mod_R/}$
(we omit $p_{\HH^\bullet(A)}\colon \Mod_R\to \RMod_{\HH^\bullet(A)}$
and $e\colon \Mod_R\to \mathcal{M}{\rm or}_R(\MMM_A,\MMM_A)$ from the notation),
we see that
the composition gives rise to an~equivalence
\begin{gather*}
\Map_{\PR_R}(\MMM_B,\RMod_{\HH^\bullet(A)})\times_{\RMod_{\HH^\bullet(A)}^{\simeq}}\{\HH^\bullet(A)\}
\\ \qquad
{}\simeq\Map_{\PR_R}(\MMM_B,\mathcal{M}{\rm or}_R(\MMM_A,\MMM_A))\times_{\mathcal{M}{\rm or}_R(\MMM_A,\MMM_A))^{\simeq}}
\{\textup{id}_{\MMM_A}\},
\end{gather*}
where the left hand side is naturally equivalent to $\Map(B,\HH^\bullet(A))$.
Unwinding the construction, we have the desired equivalence $\theta_B\colon \Map(B,\HH^\bullet(A)) \simeq c(B,A)$.
\end{proof}

\begin{Corollary}\label{centercor}
Let us regard $\HHC(A)\otimes A$ as a~right $A^e$-module induced by~that of~$A$ $($that is, $\HHC(A)\otimes (-)$ means the tensor product with $\HHC(A)\in \Mod_R)$ and regard
$\alpha\colon \HHC(A)\otimes A\to A$ as a~morphism of~right
$A^e$-moudles in the natural way.
Then the morphism $\alpha\colon \HHC(A)\otimes A\to A$
exhibits $\HHC(A)$ as a~morphism object from $A$ to itself in $\RMod_{A^e}$.
\end{Corollary}

\begin{proof}By Lemma~\ref{centerlemma},
$\HHC(A)$ (endowed with $\alpha$ and an~identification $\sigma$
between
$A\to \HHC(A)$ $\otimes A\to A$ and the identity morphism) is a~center
of~$A$.
According to~\cite[Theorem~5.3.1.30]{HA}, the morphism
$\HHC(A)\otimes A\to A$ of~the right $A^e$ modules, that is obtained from
the center, exhibits $\HHC(A)$ as a~morphism object from $A$ to $A$.
Thus, our claim follows.
\end{proof}

We~describe the bar construction $P\otimes_{A^e}A$
by~means of~symmetric spectra.
Let $\RR$ be a~cofibrant commutative symmetric ring spectrum.
Let $\AB$ be a~cofibrant associative symmetric ring \mbox{$\RR$-module} spectrum
which represents $A$, cf.~Example~\ref{spectralalgebra}.
We~write $\wedge$ for the wedge/tensor pro\-duct $\wedge_{\RR}$ over $\RR$.
Let $B_\bullet(\AB,\AB,\AB)$ be a~simplicial diagram
$\Delta^{\rm op}\to \SPS(\RR)^c$
of~symmetric spectra (called the bar construction),
which is given by~$[p] \mapsto \AB\wedge \dots \wedge \AB=\AB\wedge \AB^{\wedge p}\wedge \AB$.
We~refer to~\mbox{\cite[Definition~4.1.8]{Shhoch}} for the explicit formula of~$B_\bullet(\AB,\AB,\AB)$.
The degeneracy maps $\AB^{\wedge p+2}\to \AB^{\wedge p+1}$
is induced by~the multiplication of~$\AB\wedge \AB\to \AB$,
and face maps $\AB^{\wedge p+2}\to \AB^{\wedge p+3}$
is induced by~the unit map $\RR\to \AB$.
Each term $\AB\wedge \AB^{\wedge p}\wedge \AB$ is a~free left $\AB^e:=\AB^{\rm op}\wedge \AB$-module generated by~$\AB^{\wedge p}=\RR\wedge \AB^{\wedge p}\wedge \RR$.
In~addition, $B_\bullet(\AB,\AB,\AB)$ can be thought of~as a~simplicial
diagram of~left $A^e$-modules.
The homotopy colimit of~$B_\bullet(\AB,\AB,\AB)$ is naturally equivalent
to $\AB$ with respect to stable equivalences \cite[Lemma~4.1.9]{Shhoch}
so that the colimit of~the induced diagram in $\LMod_{A^e}$ is~$A$.
Let $\PP$ be a~right $\AB^e$-module which is cofibrant as an~$\RR$-module.
Let $\PP\wedge_{\AB^e} B_\bullet(\AB,\AB,\AB)$
be a~simplicial diagram induced by~$B_\bullet(\AB,\AB,\AB)$
which carries $[p]$
to $\PP\wedge_{\AB^e} (\AB\wedge \AB^{\wedge p}\wedge \AB) \simeq \PP\wedge \AB^{\wedge p}$.
Consider the composition
$\Alg_{\assoc}\left(\SPS(\RR)^{c}\right)\stackrel{B(-)}{\longrightarrow} \CAT_{\RR}^{\rm pc}\stackrel{\HH(-)_\bullet}{\longrightarrow} \Fun\left(\Lambda^{\rm op},\SPS(\RR)^c\right)\to \Fun\left(\Delta^{\rm op},\SPS(\RR)^c\right)$, where the final morphism is determined by~the restriction $\Delta^{\rm op}\subset \Lambda^{\rm op}$.
We~write $\HH_\bullet^{\Delta}(-)$ for the composite.
Note that $\HH_\bullet^{\Delta}(\AB)$ gives rise to a~simplicial
diagram in $\Mod_R$ whose colimit is $\HH_\bullet(A)$.
The standard computation shows that
$\AB \wedge_{\AB^e} B_\bullet(\AB,\AB,\AB)$ can be identified with~$\HH_\bullet^{\Delta}(\AB)$.

\begin{Lemma}
\label{twosidedbar}
Let $\PP$ be a~right $\AB^e$-module symmetric spectrum which is cofibrant as an~$\RR$-module.
We~write $P$ for the image of~$\PP$ in $\RMod_{A^e}$.
Then $P\otimes_{A^e}A$ can be identified with
a colimit of~the simplicial
diagram induced by~$\PP\wedge_{\AB^e} B_\bullet(\AB,\AB,\AB)$.
In~particular, $\HHH(A)$ can be identified with $A\otimes_{A^e}A$ in $\Mod_R$.
\end{Lemma}

\begin{proof}Note that the two-sided bar construction preserves colimits in each variable.
Moreover, the colimit of~$B_\bullet(\AB,\AB,\AB)$
is $A$ after passing to $\LMod_{A^e}$,
and each $\PP\wedge_{\AB^e} \big(\AB\wedge \AB^{\wedge p}\wedge \AB\big)\simeq \PP\wedge \AB^{\wedge p}$ computes
$P\otimes A^{\otimes p}\simeq P\otimes_{A^e}\big(A \otimes A^{\otimes p}\otimes A\big)$.
Therefore, lemma follows.
\end{proof}

\begin{proof}[Proof~of~Theorem~\ref{maincontraction}]
We~use the notation in the discussion about the setting of~Theorem~\ref{maincontraction}.
We~first consider an~operadic left Kan extension
$h'\colon \dcylm \to \Mod_R^\otimes$ of~$h\colon \KS\to \Mod_R^\otimes$ over $\Gamma$
along $\KS\hookrightarrow \dcylm$, cf.~Proposition~\ref{another2}.
The open embedding $j\colon (0,1)^2\sqcup (0,1)\times S^1\stackrel{j_1\sqcup j_2}{\longrightarrow} (0,1)\times S^1$
factors as the composition of~two open embeddings
$(0,1)^2\sqcup (0,1)\times S^1\stackrel{d\sqcup \textup{id}}{\longrightarrow} (0,1)\times S^1\sqcup (0,1)\times S^1\stackrel{r \sqcup j_2}{\to} (0,1)\times S^1$.
See Definition~\ref{cylmdef} for the notation and convention.
The embedding $d\colon (0,1)^2\to (0,1)\times S^1$
and $r\colon (0,1)\times S^1\to (0,1)\times S^1$
are recti-linear embeddings such that the composite $r\circ d$ is $j_1$
(of~course, the image of~$r$ does not intersect with that of~the shrinking embedding $j_2$). Thus, $f_j$ factors as $(\langle 2\rangle,D,C_M)\to (\langle 2\rangle,C,C_M)\to (\langle 1\rangle, C_M)$ in $\dcylm$, that lie over $\langle 2\rangle=\langle2\rangle \stackrel{\rho}{\to}\langle 1\rangle$.
It~follows that $h(D)\otimes h(C_M)\to h(C_M)$ induced by~$f_j$
factors as
\begin{gather*}
h(D)\otimes h(C_M)\simeq h'(D)\otimes h'(C_M)\to h'(C)\otimes h'(C_M)\to h'(C_M)\simeq h(C_M),
\end{gather*}
where $h'(D)\to h'(C)$ and $h'(C)\otimes h'(C_M)\to h'(C_M)$
are induced by~$d$ and $r\sqcup j_2$, respectively.

Next, we describe $h'(D)\to h'(C)$ in an~explicit way.
If~$\HHC(A)$ is encoded by~the map of~an $\infty$-operads
$\etwo^\otimes \to \Mod_R^\otimes$,
we let
$i_!(\HHC(A))\colon \dcyl\to \Mod_R^\otimes$ be its operadic left Kan extension
along $i\colon \etwo^\otimes \to \dcyl$.
Let $i_!(\HHC(A))((\langle 1\rangle,D))\to i_!(\HHC(A))((\langle1\rangle,C))$
be a~morphism in $\Mod_R$, that is
determined by~$d$. For simplicity,
we set $Z:=\HHC(A)$ and we write $i_!(Z)(D)\to i_!(Z)(C)$ for
this morphism.
If~we think of~$Z$ as the underlying $\eone$-algebra
given by~the composite $\eone^\otimes \to \etwo^\otimes\to \Mod_R^\otimes$,
we denote by~$l_!(Z)\colon \mathbf{Mfld}_1\to\Mod_R^\otimes$ its ope\-ra\-dic left Kan extension along
$l\colon \eone^\otimes \hookrightarrow \mathbf{Mfld}_1$.
A rectilinear embedding $(0,1)\to S^1$ gives rise to
$l_!(Z)((0,1))\to l_!(Z)\big(S^1\big)$ in $\Mod_R$ (we abuse notation as above).
More explicitly, $l_!(Z)((0,1))\to l_!(Z)\big(S^1\big)$
is $Z=\tilde{Z}((0,1)) \to \colim_{U\to S^1\in (\Disk_1)_{/S^1}} \tilde{Z}(U)$,
where $Z$ means the underlying $R$-module,
and $\tilde{Z}\colon \Disk_1\to \Mod_R$ is the underlying functor of~the
symmetric monoidal functor $\tilde{Z}\colon \Disk_1^\otimes\to \Mod_R^\otimes$
corresponding to the composite
$\eone^\otimes \to \etwo^\otimes\to \Mod_R^\otimes$.
Since
the forgetful functor $\Alg_{\eone}(\Mod_R)\to \Mod_R$
preserves colimits over the sifted category $(\Disk_1)_{/S^1}$,
it follows from Lemma~\ref{cofinalitylemma}
that
$i_!(Z)(D)\to i_!(Z)(C)$ is naturally equivalent to
$Z=l_!(Z)((0,1))\to l_!(Z)\big(S^1\big)=\colim_{(\Disk_1)_{/S^1}}\tilde{Z}(U)$
as a~morphism in $\Mod_R$.
Therefore, we may identify $h'(D)\to h'(C)$ with $Z=l_!(Z)((0,1))\to l_!(Z)(S^1)=\colim_{(\Disk_1)_{/S^1}}\tilde{Z}(U)$.
Henceforth,
we regard $Z$ as an~associative algebra in $\Mod_R$,
which is given by~a symmetric monoidal functor $\widetilde{\assoc}^\otimes\to \Mod^\otimes_R$.
The map $\xi\colon \Lambda^{\rm op}\to \widetilde{\assoc}^{\otimes}$ (see the discussion
proceeding to Lemma~\ref{observationmap})
and $\Delta^{\rm op}\hookrightarrow \Lambda^{\rm op}_\infty\to \Lambda^{\rm op}$
induces
\begin{gather*}
\Alg_{\assoc}(\Mod_R) \simeq \Fun^\otimes\big(\widetilde{\assoc}^\otimes,\Mod_R^\otimes\big)\to \Fun\big(\Lambda^{\rm op},\Mod_R\big)
\\ \hphantom{\Alg_{\assoc}(\Mod_R)}
{}\to \Fun\big(\Lambda_\infty^{\rm op},\Mod_R\big) \to \Fun\big(\Delta^{\rm op},\Mod_R\big).
\end{gather*}
It~follows from Lemma~\ref{redcolimit}
that
the image of~$Z$ in $\Fun(\Lambda_\infty^{\rm op},\Mod_R)$
is the composite $\Lambda_\infty^{\rm op}\simeq \big(\Diskd_1\big)_{/S^1}\to \Disk_1\to \Mod_R$
whose colimit is naturally equivalent to $l_!(Z)\big(S^1\big)$.
Moreover, from the cofinality of~$\Delta^{\rm op}\to \Lambda^{\rm op}_\infty$,
$l_!(Z)\big(S^1\big)$
is naturally equivalent to the
colimit of~$c\colon \Delta^{\rm op}\to \Lambda_\infty^{\rm op}\to \Mod_R$.
Taking into account Remark~\ref{globaldiagram}, $Z\to l_!(Z)\big(S^1\big)$ can be identified with
$Z=c([0])\to \colim_{[p]\in \Delta^{\rm op}}c([p])$.

Next let us consider $h'(C)\otimes h'(C_M)\to h'(C_M)$.
According to Construction~\ref{constalg},
its under\-lying morphism in $\Mod_R$
can naturally be
identified with
$\HHH(Z)\otimes \HHH(A)\simeq \HHH(Z\otimes A)\to \HHH(A)$
induced by~$\alpha\colon Z\otimes A\to A$.
Let $\ZZZ$ and $\AB$ be cofibrant and fibrant associative ring symmetric $\RR$-module
spectra that represent $Z$ and $A$, respectively
(namely, they are objects in~$\Alg_{\assoc}\!\big(\SPS\big(\RR\big)\big)$
which are both cofibrant and fibrant with respect to
the projective model structure).
Let $\bar{\alpha}\colon \ZZZ\wedge \AB\to \AB$ be a~morphism in $\Alg_{\assoc}\!\left(\SPS(\RR)^c\right)$
which represents $\alpha$.
The com\-posite $\AB \simeq \RR\wedge \AB\to \ZZZ\wedge \AB\to \AB$
induced by~$\RR\to \ZZZ$ is equivalent to the identity morphism of~$\AB$.
Let $\HHH^\Delta(\ZZZ)\wedge \HHH^\Delta(\AB)$ denote
the bisimplicial diagram induced by~the wedge product
and $\big(\HHH^\Delta(\ZZZ)\wedge \HHH^\Delta(\AB)\big)^{\rm diag}$
the associated diagonal simplicial diagram.
The morphism $\bar{\alpha}$ induces the following morphism of~simplicial diagrams
\begin{gather*}
\big(\HHH^\Delta(\ZZZ)\wedge \HHH^\Delta(\AB)\big)^{\rm diag} \simeq \HHH^\Delta(\ZZZ\wedge \AB)\to \HHH^\Delta(\AB)
\end{gather*}
whose colimit is equivalent to
$\HHH(Z)\otimes \HHH(A)\simeq \HHH(Z\otimes A)\to \HHH(A)$
(notice that $\Delta^{\rm op}$ is sifted).

Note that
$l_!(Z)\big(S^1\big)= \colim_{[p]\in \Delta^{\rm op}}c([p])\simeq \HHH(Z)$.
If~$\ZZZ_{cst}$ denotes the constant simplicial diagram
taking the value $\ZZZ$,
$\{[0]\}\to \Delta^{\rm op}$ induces $\ZZZ_{cst}\to\HHH^\Delta(\ZZZ)$
whose colimit is $Z=c([0])\to \colim_{[p]\in \Delta^{\rm op}}c([p])\simeq \HHH(Z)$.
Consider $\ZZZ_{cst}\wedge \HHH^\Delta(\AB) \to \big(\HHH^\Delta(\ZZZ)\wedge \HHH^\Delta(\AB)\big)^{\rm diag} \to \HHH^\Delta(\AB)$.
Their colimits give $Z\otimes \HHH(A)\to \HHH(Z)\otimes \HHH(A)\to \HHH(A)$,
which is equivalent to $h'(D)\otimes h'(C_M)\to h'(C)\otimes h'(C_M)\to h'(C_M)$.
Observe that the composition $\ZZZ_{cst}\wedge \HHH^\Delta(\AB)=(\ZZZ\wedge \AB)\wedge_{\AB^e}B_\bullet(\AB,\AB,\AB)\to\AB\wedge_{\AB^e}B_\bullet(\AB,\AB,\AB)=\HHH^\Delta(\AB)$
is induced by~$\bar{\alpha}\colon \ZZZ\wedge \AB\to \AB$.
The contraction morphism
$(Z\otimes A)\otimes_{A^e}A\to A\otimes_{A^e}A$
is obtained from
$\ZZZ\wedge (\AB\wedge_{\AB^e}B_\bullet(\AB,\AB,\AB))=(\ZZZ\wedge \AB)\wedge_{\AB^e}B_\bullet(\AB,\AB,\AB)\to\AB\wedge_{\AB^e}B_\bullet(\AB,\AB,\AB)$ by~taking colimits. This completes the proof.
\end{proof}

{\bf 8.2.} Let $k$ be a~field.
We~suppose that $R$ is the Eilenberg--MacLane spectrum of~$k$.
We~write $k$ for $R$.
In~this context, we will give a~concrete model of~the contraction morphism
\begin{gather*}
\sigma\colon\ \HHC(A)\otimes \HH_\bullet(A)\to \HH_\bullet(A)
\end{gather*}
as a~morphism of~chain complexes of~$k$-vector spaces.
Let $\Comp^\otimes(k)$ be the symmetric monoidal
category of~chain complexes of~$k$-vector
spaces, whose tensor product is given by~the standard tensor
product of~chain complexes.
There is a~symmetric monoidal (projective) model structure
on $\Comp(k)$
such that a~morphism is a~weak equivalence (resp.\ a fibration)
if it is a~quasi-isomorphism (resp.\ a termwise surjective),
see, e.g., \cite[Proposition~7.1.2.11]{HA}.
Since~$k$ is a~field, every object is
both cofibrant and fibrant.
Let $\Comp(k)\big[W^{-1}\big]^\otimes$ be the symmetric monoidal $\infty$-category
obtained by~inverting quasi-isomorphisms. We~fix a~symmetric monoidal equivalence $\Mod_k^\otimes\simeq \Comp(k)\big[W^{-1}\big]^\otimes$,
see \cite[Theorem~7.1.2.13]{HA}.
Thus it gives rise to $\Alg_{\assoc}(\Comp(k))\to \Alg_{\assoc}\big(\Comp(k)\big[W^{-1}\big]\big)\simeq \Alg_{\assoc}(\Mod_k)$.
We~here regard $\Alg_{\assoc}(\Comp(k))$ as (the nerve of) the category of~differential graded $k$-algebras in an~obvious way.
Let $\AADG$ be an~associative (cofibrant) differential graded
$k$-algebra, i.e., an~object of~$\Alg_{\assoc}(\Comp(k))$
that represents $A\in \Alg_{\assoc}(\Mod_k)$.
Let $\AADG^e:=\AADG^{\rm op}\otimes \AADG$.
The natural functor $\Comp(k)^\otimes\to \Comp(k)\big[W^{-1}\big]^\otimes \simeq \Mod_k^\otimes$ induces
\begin{gather*}
\RMod_{\AADG^e}(\Comp(k)) \times \LMod_{\AADG^e}(\Comp(k))\to \RMod_{A^e}(\Mod_k)\times \LMod_{A^e}(\Mod_k)\to \Mod_k,
\end{gather*}
where the right functor is informally given by~the bar construction
(i.e., the relative tensor product) $(-)\otimes_{A^e}(-)$.
We~describe the contraction morphism by~using explicit resolutions.
Let $B^{\rm dg}_\bullet(\AADG):=B^{\rm dg}_\bullet(\AADG,\AADG,\AADG)$
be the right $\AADG^e$-module associated to
the total complex of~the simplicial diagram of~$\AADG^{e}$-modules $[p]\mapsto \AADG\otimes \AADG^{\otimes p}\otimes \AADG$ (defined as in $B_\bullet(\AB,\AB,\AB)$).
The asso\-ci\-a\-ted total complex $B^{\rm dg}_\bullet(\AADG)$ computes a~homotopy
colimit of~the simplicial
diagram, and there is a~canonical morphism of~right $\AADG^e$-modules
$B^{\rm dg}_\bullet(\AADG)\to \AADG$ which is a~quasi-isomorphism.
Let $\Hom_{\AADG^e}\big(B^{\rm dg}_\bullet(\AADG),B^{\rm dg}_\bullet(\AADG)\big)$
be the hom chain complex of~right $\AADG^e$-modules.
By Lemma~\ref{barcofibrant} below,
$\Hom_{\AADG^e}\big(B^{\rm dg}_\bullet(\AADG),B^{\rm dg}_\bullet(\AADG)\big)$ represents/computes the (derived) hom complex from
$\AADG$ to $\AADG$ in~$\RMod_{\AADG^e}(\Comp(k))$.
Let us consider an~evaluation morphism of~right $\AADG^e$-modules
\begin{gather*}
\textup{Ev}\colon\ \Hom_{\AADG^e}(B^{\rm dg}_\bullet(\AADG),B^{\rm dg}_\bullet(\AADG))\otimes B^{\rm dg}_\bullet(\AADG)\to B^{\rm dg}_\bullet(\AADG)
\end{gather*}
defined in the obvious way,
where $\Hom_{\AADG^e}\big(B^{\rm dg}_\bullet(\AADG),B^{\rm dg}_\bullet(\AADG)\big)\otimes B^{\rm dg}_\bullet(\AADG)$ comes equipped with the right $\AADG^e$-module structure
induced by~$B^{\rm dg}_\bullet(\AADG)$.
Let $\RMod_{\AADG^e}(\Comp(k))\big[W^{-1}\big]$ denote the \mbox{$\infty$-category}
obtained by~inverting quasi-isomorphisms (after restricting to
cofibrant objects).
By the universal property of~the morphism $\big(\Hom_{\AADG^e}\big(B^{\rm dg}_\bullet(\AADG),B^{\rm dg}_\bullet(\AADG)\big),\textup{Ev}\big)$
and the equivalence
\begin{gather*}
\RMod_{\AADG^e}(\Comp(k))\big[W^{-1}\big]\simeq \RMod_{A^e}(\Mod_k)
\end{gather*}
\cite[Theorem~4.3.3.17]{HA} together with Lemma~\ref{barcofibrant},
it represents a~morphism object from $A$ to $A$
in $\RMod_{A^e}$. Note that
$B^{\rm dg}_\bullet(\AADG)\otimes_{\AADG^e}\AADG$ is the chain complex associated to the simplicial chain complex given by~$[p] \mapsto \big(\AADG\otimes \AADG^{\otimes p}\otimes \AADG\big)\otimes_{\AADG^e}\AADG\simeq \AADG^{\otimes p}\otimes \AADG$ so that the associated complex
is a~model of~$A\otimes_{A^e}A=\HHH(A)$. That is,
its image in $\Mod_k$ is naturally equivalent
to $B^{\rm dg}_\bullet(\AADG)\otimes_{\AADG^e}B^{\rm dg}_\bullet(\AADG)\stackrel{\sim}{\rightarrow} \AADG\otimes_{\AADG^e}B^{\rm dg}_\bullet(\AADG)=\HHH(A)$.
By the Morita theory (cf., e.g., \cite[Section~4]{BFN}),
$\mathcal{M}{\rm or}_k(\RMod_A,\RMod_A)\simeq \RMod_{A^e}$,
where the identity functor amounts to $A\in \RMod_{A^e}$ having the diagonal
module structure. Consequently,
$\Hom_{\AADG^e}\big(B^{\rm dg}_\bullet(\AADG),B^{\rm dg}_\bullet(\AADG)\big)\in \Comp(k)$
is a~model of~$\HHC(A)\in \Mod_k$.
Moreover, $B^{\rm dg}_\bullet(\AADG)\to \AADG$ induces
a quasi-isomorphism
$\Hom_{\AADG^e}\big(B^{\rm dg}_\bullet(\AADG),B^{\rm dg}_\bullet(\AADG)\big)\to
\Hom_{\AADG^e}\big(B^{\rm dg}_\bullet(\AADG),\AADG\big)$
(perhaps, the latter is more common Hoch\-schild cochain complex).
We~conclude:

\begin{Proposition}\label{explicitdg}
The above evaluation morphism induces
\begin{gather*}
\Hom_{\AADG^e}\big(B^{\rm dg}_\bullet(\AADG),B^{\rm dg}_\bullet(\AADG)\big)\otimes B^{\rm dg}_\bullet(\AADG)\otimes_{\AADG^e}\AADG\to B^{\rm dg}_\bullet(\AADG)\otimes_{\AADG^e}\AADG,
\end{gather*}
which is equivalent to the contraction morphism
$\sigma\colon \HHC(A)\otimes \HHH(A)\to \HHH(A)$.
In~par\-ti\-cu\-lar, it is one of~(explicit) models of~$u$ $($cf.~Theorem~{\rm \ref{maincontraction})}.
\end{Proposition}

\begin{Lemma}\label{barcofibrant}
Let us consider $\RMod_{\AADG^e}(\Comp(k))$
to be a~category endowed with $\Comp(k)$-enriched projective
model structure, where a~morphism is a~weak equivalence
$($resp.\ a fibration$)$ if it is a~quasi-isomorphism $($resp.\ a termwise surjective$)$, see, e.g., {\rm \cite[Theorem~3.3]{Six}}.
Then $B^{\rm dg}_\bullet(\AADG)$ is cofibrant
with respect to this model structure.
\end{Lemma}

\begin{proof}
Since $B^{\rm dg}_\bullet(\AADG)$ is obtained from the simplicial
chain complex $[p]\mapsto \AADG\otimes \AADG^{\otimes p}\otimes \AADG$,
there is an~increasing filtration of~$\AADG^e$-submodules
$0=F_{-1}\big(B^{\rm dg}_\bullet(\AADG)\big)\hookrightarrow F_{0}\big(B^{\rm dg}_\bullet(\AADG)\big)\hookrightarrow F_{1}\big(B^{\rm dg}_\bullet(\AADG)\big)\hookrightarrow \cdots$
such that $\cup_{p\ge0} F_p\big(B^{\rm dg}_\bullet(\AADG)\big)=B^{\rm dg}_\bullet(\AADG)$.
The quotient $F_{p+1}\big(B^{\rm dg}_\bullet(\AADG)\big)/F_p\big(B^{\rm dg}_\bullet(\AADG)\big)$
is isomorphic to~$\AADG\otimes \overline{\AADG}^{\otimes p+1}\otimes \AADG$
as a~right $\AADG^e$-module, where $\overline{\AADG}$ is the cokernel of~a unit morphism $k\to \AADG$. That is,
$F_{p+1}(B^{\rm dg}_\bullet(\AADG))/F_p\big(B^{\rm dg}_\bullet(\AADG)\big)$
is a~free right $\AADG^e$-module generated by~$\overline{\AADG}^{\otimes p+1}\in \Comp(k)$.
By~\mbox{\cite[Definition 9.17, Theorem~9.20]{Six}}, the existence of~this
filtration implies that $B^{\rm dg}_\bullet(\AADG)$
is cofibrant with respect to the $r$-model structure in~\cite[Section~4]{Six}.
We~deduce from the assumption that $k$ is a~field
that the projective model structure coincides with this $r$-model
structure (the~projective model structure is the same as the $q$-model
structure in~\cite[Section~3]{Six}).
This completes the proof.
\end{proof}

\begin{Remark}
\label{otheraction}
There are other operations between
$\HHC(A)\otimes \HHH(A)$ and $\HHH(A)$,
which is induced by~the $\KS$-algebra structure
on $(\HHC(A),\HHC(A))$.
We~continue to work with the coefficient field $R=k$.
We~further assume that $k$ is of~characteristic zero.
Observe first that the Kan complex
$\mul_{\textup{KS}}(\{D,C_M\},C_M)$
is equivalent to the product of~the circles $S^1\times S^1$
in $\SSS$, where $S^1\times S^1$ is regarded as an~object of~$\SSS$.
The left factor $S^1$ is homotopy equivalent to
the space of~configuration of~one point on $(0,1)\times S^1$,
which we regard as the space of~rectilinear embeddings
from $(0,1)^2$ into $(0,1)\times S^1$.
The right factor $S^1$ can be identified with
the mapping space from~$C_M$ to itself, that is,
the space of~shrinking embeddings $(0,1)\times S^1\to (0,1)\times S^1$.
The map
of~$\infty$-operads $h\colon \KS\to \Mod_k^\otimes$ encoding
the $\KS$-algebra $(\HH^\bullet(A),\HH_\bullet(A))$
in Theorem~\ref{mainconstruction}
induces morphisms $h_{(\{D,C_M\},C_M)}\colon \mul_{\textup{KS}}(\{D,C_M\},C_M)\to \Map_{\Mod_k}(\HHC(A)\otimes \HHH(A),\HHH(A))$
and $h_{(\{C_M\},C_M)}\colon \mul_{\textup{KS}}(\{C_M\},C_M)\to \Map_{\Mod_k}(\HHH(A),\HHH(A))$ in $\SSS$.
To discuss operations induced by~these morphisms,
we adopt the differential graded setting (we~simplify the problem).
By taking the singular chain complex of~the simplicial sets
$\mul_{\textup{KS}}(-,-)$ with coefficients in $k$,
we obtain the differential graded (dg)
operad $C_\bullet(\textup{KS})$ from $\textup{KS}$,
i.e., an~operad in the symmetric monoidal category $\Comp(k)$.
We~refer to~\cite{Hin} for the relation between
dg operads and $\infty$-operads.
By a~rectification result of~Hinich \cite[Theorem~4.1.1]{Hin},
there is a~canonical equivalence between $\Alg_{\KS}(\Mod_k)$
and the $\infty$-category
of~$C_\bullet(\textup{KS})$-algebras in $\Comp(k)$ in the ``conventional" sense
so that the
$\KS$-algebra $(\HHC(A),\HHH(A))$ gives rise to
a (an essentially unique) $C_\bullet(\textup{KS})$-algebra
in $\Comp(k)$
in the ``conventional'' sense: see \cite[Section~2.2.4]{Hin} (we here use
the assumption of~characteristic zero).
We~denote this $C_\bullet(\textup{KS})$-algebra by~$(HH^\bullet(A),HH_\bullet(A))$.
Then
the $C_\bullet(\textup{KS})$-algebra $(HH^\bullet(A),HH_\bullet(A))$
induces morphisms into hom chain complexes:
\begin{gather*}
h_{(\{D,C_M\},C_M)}^{\rm dg}\colon\ C_\bullet(\mul_{\textup{KS}}(\{D,C_M\},C_M))\to \Hom_k(HH^\bullet(A)\otimes HH_\bullet(A),HH_\bullet(A))
\end{gather*}
and
\begin{gather*}
h_{(\{C_M\},C_M)}^{\rm dg}\colon\ C_\bullet(\mul_{\textup{KS}}(\{C_M\},C_M))\to \Hom_k(HH_\bullet(A),HH_\bullet(A))
\end{gather*} in $\Comp(k)\big[W^{-1}\big]$.
Given a~simplicial set $S$,
we write $H_*(S)$ for the homology $H_*(C_\bullet(S))$ of~the singular chain complex $C_\bullet(S)$ of~$S$
with coefficients in $k$.
Passing to homology, we obtain a~morphism of~graded $k$-vector spaces
\begin{gather*}
H_*\big(h_{(\{D,C_M\},C_M)}^{\rm dg}\big)\colon\ H_*(\mul_{\textup{KS}}(\{D,C_M\},C_M))
\\ \hphantom{H_*\big(h_{(\{D,C_M\},C_M)}^{\rm dg}\big)\colon}{} \
 \to H_*(\Hom_k(HH^\bullet(A)\otimes HH_\bullet(A),HH_\bullet(A)))
\end{gather*}
from $h_{(\{D,C_M\},C_M)}^{\rm dg}$.
Note that by~definition the left-hand side is connective.
This map can also be obtained
by~taking homology of~$C_\bullet(h_{(\{D,C_M\},C_M)})$ up to isomorphisms.
Since the mapping space $\mul_{\textup{KS}}(\{D,C_M\},C_M)$ is homotopy equivalent to $S^1\times S^1$,
it follows that
\begin{gather*}
H_0(\mul_{\textup{KS}}(\{D,C_M\},C_M))\simeq k,
\\
H_1(\mul_{\textup{KS}}(\{D,C_M\},C_M))\simeq k\oplus k,
\\
H_2(\mul_{\textup{KS}}(\{D,C_M\},C_M))\simeq k,
\end{gather*}
and the other parts are zero.

Let $[0,1]\subset \RRR$ be the closed interval and let
$\phi\colon [0,1]\times S^1=[0,1]\times \RRR/\ZZ \to \RRR/\ZZ=S^1$
be the continuous map given by~$(t,x\ \textup{mod}\ \ZZ) \mapsto x+t\ \textup{mod}\ \ZZ$. We~think of~$\phi$ as a~homotopy from the identity map $S^1\to S^1$
to itself. The product $\operatorname{id}_{(0,1)}\times \phi$
determines a~homotopy $\overline{\phi}\colon [0,1]\times (0,1)\times S^1\to (0,1)\times S^1$
from the identity map of~$(0,1)\times S^1$ to itself.
Let $r_1:[0,1]\times (0,1)^2\stackrel{\textup{id}\times j_1}\to [0,1]\times (0,1)\times S^1\stackrel{\overline{\phi}}{\to}(0,1)\times S^1$
be the homotopy from $j_1$ to $j_1$, see discussion before Theorem~\ref{maincontraction} for the maps $j$, $j_1$, and $j_2$.
Let $r_2\colon [0,1]\times (0,1)\times S^1 \stackrel{\textup{pr}}{\to} (0,1)\times S^1\stackrel{j_2}{\to} (0,1)\times S^1$ be the trivial
homotopy from $j_2$ and $j_2$.
Define
\begin{gather*}
l:=r_1\sqcup r_2\colon\ [0,1]\times \big((0,1)^2\sqcup (0,1)\times S^1\big) \to (0,1)\times S^1.
\end{gather*}
We~think of~$l$ as a~homotopy from $j$ to $j$.
If~we regard $\overline{\phi}$ as a~homotopy from the identity to itself
in $\mul_{\textup{KS}}(\{C_M\}, C_M)$,
it determines an~element $e_B$ of~$H_1(\mul_{\textup{KS}}(\{C_M\}, C_M))$.
If~we think of~$l$ as a~homotopy from $f_j$ to itself
in $\mul_{\textup{KS}}(\{D,C_M\},C_M)$,
it determines an~ele\-ment $e_{L}$ of~$H_1(\mul_{\textup{KS}}(\{D,C_M\},C_M))$.
The $C_\bullet(\textup{KS})$-algebra $(HH^\bullet(A),HH_\bullet(A))$
(or equivalently the $\KS$-algebra $(\HHC(A),\HHH(A))$
in Theorem~\ref{mainconstruction}) defines $L\in H_1(\Hom_k(HH^\bullet(A)\otimes_k HH_\bullet(A),HH_\bullet(A)))$ for $e_L$:
$L$ is defined to be the image of~$e_L$ under $H_*\big(h_{(\{D,C_M\},C_M)}^{\rm dg}\big)$.
Similarly, $e_B$ determines an~element
$B$ in $H_1(\Hom_k(HH_\bullet(A),HH_\bullet(A)))$.

The morphism $f_j$ arising from
$j=j_1\sqcup j_2$
is a~generator of~the $0$-th homology group.
The image of~$f_j$ under $h_{(\{D,C_M\},C_M)}^{\rm dg}$
can be identified with $u\colon \HHC(A)\otimes \HHH(A)\to \HHH(A)$ which is equivalent to the contraction morphism, cf.~Theorem~\ref{maincontraction}.
By a~comparison result of~Hoyois \mbox{\cite[Theorem~2.3]{Hoy}},
$B$ in $H_1(\Hom_k(HH^\bullet(A),HH_\bullet(A)))$
may be viewed as the Connes operator:
if $\AADG$ is an~associative
(cofibrant) differential graded algebra which represents $A$
through $\Alg_{\assoc}(\Mod_k^\otimes)\simeq \Alg_{\assoc}(\Comp(k)[W^{-1}])$,
the classical Connes operator
$b\colon B_\bullet^{\rm dg}(\AADG)\otimes_{\AADG^e}\AADG\to B_\bullet^{\rm dg}(\AADG)\otimes_{\AADG^e}\AADG$ (see, e.g.,~\cite{Hoy, L})
is equivalent to $B$, up to the multiplication by~$\pm1$,
through $B_\bullet^{\rm dg}(\AADG)\otimes_{\AADG^e}\AADG\simeq \HHH(A)\simeq HH_\bullet(A)$.

The homotopy $l$ and the composition
$\overline{\phi}\circ (\textup{id}_{[0,1]}{\times} j)\colon [0,1]{\times} \big((0,1)^2\sqcup(0,1){\times} S^1\big)\to (0,1){\times} S^1$ generate the $k$-vector space
$k\oplus k\simeq H_1(\mul_{\textup{KS}}(\{D,C_M\},C_M))$.
Let $\mathbf{1}$ denote the ele\-ment of~$H_0(\Hom_{k}(\HHC(A),\HHC(A)))$ that corresponds to the identity morphism.
By~relations of~homotopies, we see that
\begin{gather*}
L=B\circ u +u\circ (\mathbf{1}\otimes (-B))=B\circ u -u\circ (\mathbf{1}\otimes B)
\end{gather*}
in $H_1(\Hom_k(HH^\bullet(A)\otimes HH_\bullet(A),HH_\bullet(A)))$.
Here ``$\circ$'' indicates the composition. This relation is
known as Cartan homotopy/magic formula.
In~the dg setting over $k$, the shifted complex
$HH^\bullet(A)[1]$ inherits the structure of~an $L_\infty$-algebra,
i.e., an~algebra over a~(cofibrant) Lie operad in $\Comp(k)$,
from the $\etwo$-algebra structure on $\HHC(A)\simeq HH^\bullet(A)$.
The morphism
\begin{gather*}
L\colon\ HH^\bullet(A)[1]\otimes HH_\bullet(A)\to HH_\bullet(A)
\end{gather*}
induced by~$L$ appears as the Lie algebra
action morphism on $HH_\bullet(A)$
(see, e.g.,~\cite{DTT}).
Since $u$ and $B$ can be explicitly described ($B$ is equivalent to Connes' operator), thus $L$ also has an~explicit presentation.
Finally, $H_2(\mul_{\textup{KS}}(\{D,C_M\},C_M))$ is generated by~$B\circ u\circ (\mathbf{1}\otimes B)$.
\end{Remark}

\section{Equivariant context}\label{equivariantsec}

Our construction in Theorem~\ref{mainconstruction} can easily be
generalized
to an~equivariant setting:
Let $G$ be a~group object in $\SSS$
and $BG\in \SSS$ the classifying space.
Let $\CCC\in \ST_R$, that is, a~small stable $R$-lienar idempotent-complete
$\infty$-category.
Suppose that $G$ acts on $\CCC$, i.e., an~left action on $\CCC$.
Namely, $\CCC$ is an~object of~$\Fun(BG,\ST_R)$
whose image under the forgetful functor $\Fun(BG,\ST_R)\to \ST_R$ is $\CCC$.
In~this setting, we have

\begin{Theorem}
\label{mainequivariant}
The pair $(\HHC(\CCC),\HHH(\CCC))$ of~Hoch\-schild cohomology and homology $R$-module spectra has the structure of~a~$\KS$-algebra in $\Fun(BG,\Mod_R)$.
In~other words, $(\HHC(\CCC),\HHH(\CCC))$ is promoted to $\Alg_{\KS}(\Fun(BG,\Mod_R))$.\end{Theorem}

\begin{Remark}\label{equivrem}
The forgetful functor $\Alg_{\KS}(\Fun(BG,\Mod_R))\to \Alg_{\KS}(\Mod_R)$
sends the $\KS$-algebra in Theorem~\ref{mainequivariant}
to a~$\KS$-algebra equivalent to the $\KS$-algebra constructed in Theorem~\ref{mainconstruction}.
\end{Remark}

Theorem~\ref{mainequivariant} follows from the following:

\begin{Construction}\label{equivariantconstruction}
The construction is almost the same as that of~Theorem~\ref{mainconstruction}. Thus, we highlight necessary modifications.

(i) Let $\DDD:=\Ind(\CCC)$ be the Ind-category which is an~$R$-linear
compactly generated
$\infty$-category. In~particular, $\DDD$ belongs to $\PR_R$.
Since $\CCC\mapsto \Ind(\CCC)$ is functorial,
the left action of~$G$ on $\CCC$ induces a~left action on $\DDD$.
Namely, $\DDD$ is promoted to $\Fun\big(BG,\PR_R\big)$.
The functor category $\Fun\big(BG,\PR_R\big)$
inherits a~(pointwise) symmetric monoidal structure from that
of~$\PR_R$. Let $\mathcal{M}{\rm or}_R^G(\DDD,\DDD)$ be
an internal hom object in the symmetric monoidal $\infty$-category $\Fun\big(BG,\PR_R\big)$. This is explicitly described as follows:
The internal hom object
$\mathcal{M}{\rm or}_R(\DDD,\DDD)$ (Lemma~\ref{linearfunctorcategory}) in $\PR_R$
has the left action of~$G^{\rm op}\times G$ induced by~the functoriality of~the internal hom object and the action of~$G$ on $\DDD$ (here $G^{\rm op}$ denotes
the opposite group). The homomorphism $G\to G^{\rm op}\times G$ informally given by~$g\mapsto \big(g^{-1},g\big)$ determines a~left action of~$G$ on $\mathcal{M}{\rm or}_R(\DDD,\DDD)$. By the universal property of~$\mathcal{M}{\rm or}_R(\DDD,\DDD)$, $\mathcal{M}{\rm or}_R(\DDD,\DDD)$ endowed with the $G$-action is an~internal hom object from $\DDD$ to $\DDD$ in $\Fun\big(BG,\PR_R\big)$. As in Lemma~\ref{linearfunctorcategory}, $\mathcal{M}{\rm or}_R^G(\DDD,\DDD)$ is promoted to $\Alg_{\assoc}(\Fun(BG,\Mod_R))\simeq \Fun(BG,\Alg_{\assoc}(\Mod_R))$. Here the equivalence follows from
the definition of~the pointwise symmetric monoidal strucutre on
$\Fun(BG,\Mod_R)$ \cite[Remark~2.1.3.4]{HA}.

(ii) Recall the adjunction $I\colon \Alg_{\assoc}(\Mod_R)\rightleftarrows \left(\PR_R\right)_{\Mod_R/}\, {\colon}\!E$ induces an~adjunction
\begin{gather*}
\Alg_{\etwo}(\Mod_R)\simeq \Alg_{\assoc}\big(\Alg_{\assoc}\big(\Mod_R\big)\big)\rightleftarrows \Alg_{\assoc}\big(\PR_R\big)\simeq \Alg_{\eone}\big(\PR_R\big)
\end{gather*}
(see Section~\ref{cohomologysec}).
Applying $\Fun(BG,-)$ to this adjunction
we get an~adjunction
\begin{gather*}
\Fun(BG,\Alg_{\etwo}(\Mod_R))\rightleftarrows \Fun\big(BG,\Alg_{\assoc}\big(\PR_R\big)\big)\simeq \Alg_{\assoc}\big(\Fun\big(BG,\PR_R\big)\big).
\end{gather*}
We~define the Hoch\-schild cohomology $R$-module spectrum
\begin{gather*}
\HHC(\CCC):=\HHC(\DDD) \in \Fun(BG,\Alg_{\etwo}(\Mod_R))\simeq \Alg_{\etwo}(\Fun(BG,\Mod_R))
\end{gather*}
to be the image of~$\mathcal{M}{\rm or}_R^G(\DDD,\DDD)$ under the right adjoint.
For ease of~notation, we write $A$ for $\HHC(\CCC)$.
Write $\RMod_{A}^{G,\otimes}$ for the image of~$A=\HHC(\CCC)=\HHC(\DDD) \in \Fun(BG,\Alg_{\etwo}(\Mod_R))$ under the left adjoint.
Consider counit map $\RMod_{A}^{G,\otimes}\!\to \mathcal{M}{\rm or}_R^G(\DDD,\DDD)$
in $\Fun\big(BG,\Alg_{\assoc}\!\big(\PR_R\big)\big)$ which we regard as a~$G$-equivariant
associative monoidal functor.
As in the non-equivariant case, the natural action of~$ \mathcal{M}{\rm or}_R^G(\DDD,\DDD)$ on $\DDD$ gives rise to a~left action of~$\RMod_{A}^{G,\otimes}$
on $\DDD$. More precisely, $\DDD$ is promoted to an~object of~$\LMod_{\RMod_{A}^{G,\otimes}}(\Fun\left(BG,\PR_R\right))$.
Let $\RPerf^{G,\otimes}_{A}$ be the symmetric monoidal full subcategory
spanned by~compact objects.
By definition, $\RPerf^{G,\otimes}_{A}$ is the associative
monoidal small $R$-linear $\infty$-category $\RPerf_A^\otimes$
endowed with the $G$-action given informally by~the $G$-action on $A$.
As in Proposition~\ref{smallleftmodule},
the restriction exhibits
$\CCC$ as a~left $\RPerf^{G,\otimes}_{A}$-module, that is,
an object of~$\LMod_{\RPerf^{G,\otimes}_{A}}(\Fun(BG,\ST_R))$.

(iii) Let $\HHH(-)\colon \ST_R\to \Fun(BS^1,\Mod_R)$ denote the symmetric monoidal functor
which carries $\mathcal{E}$ to $\HHH(\mathcal{E})$ (see Definition~\ref{homologydef}). Applying $\Fun(BG,-)$ to it, we have a~symmetric monoidal functor
$\Fun(BG,\ST_R)\to \Fun(BG,\Fun(BS^1,\Mod_R))$.
We~define $\HHH(\CCC)$ to be the image of~$\CCC$
in $\Fun(BG,\Fun(BS^1,\Mod_R))$.
By the induced functor
\begin{gather*}
\LMod_{\RPerf^{G,\otimes}_{A}}(\Fun(BG,\ST_R))\to \LMod_{\HHH\left(\RPerf^{G,\otimes}_{A}\right)}\big(\Fun\big(BG,\Fun\big(BS^1,\Mod_R\big)\big)\big),
\end{gather*}
we regard $\HHH(\CCC)$ as a~left $\HHH\big(\RPerf^{G,\otimes}_{A}\big)$-module.
That is to say, $\HHH(\CCC)$ is an~object of~\begin{gather*}
\LMod_{\HHH\left(\RPerf^{G,\otimes}_{A}\right)}\big(\Fun\big(BG,\Fun\big(BS^1,\Mod_R\big)\big)\big)
\\ \qquad
{}\simeq \LMod_{\HHH\left(\RPerf^{G,\otimes}_{A}\right)}\big(\Fun\big(BS^1,\Fun(BG,\Mod_R)\big)\big).
\end{gather*}
According to the Morita invariance, $\HHH\big(\RPerf^{G,\otimes}_{A}\big)$
can naturally be identified with $\HHH(A)$, where $\HHH(A)$ comes equipped with
a left action of~$G$ induced by~the $G$-action on $A=\HHC(\CCC)$.
Let $i_!(A)\colon \dcyl\to \Fun(BG,\Mod_R)^\otimes$ denote the operadic left Kan extension of~the map
of~$\infty$-operads $A\colon \etwo^\otimes\to \Fun(BG,\Mod_R)^\otimes$ over $\Gamma$
along $i\colon \etwo^\otimes \to \dcyl$ where $\Fun(BG,\Mod_R)^\otimes\to \Gamma$
is the pointwise symmetric monoidal $\infty$-category induced by~the structure on $\Mod_R^\otimes$. Let $i_!(A)_C$ be the restriction to $\cyl\subset \dcyl$ which
we think of~as an~object of~$\Alg_{\cyl}(\Fun(BG,\Mod_R))$.
If~we replace $\Mod_R$ by~$\Fun(BG,\Mod_R)$
in the proof~of~Proposition~\ref{identificationprop},
the argument yields a~canonical equivalence $\HHH(A)\simeq i_!(A)_C$ in
$\Fun(BS^1,\Fun(BG,\Mod_R))$.
Then $A=\HHC(\CCC)\in \Alg_{\etwo}(\Fun(BG,\Mod_R))$, the left $\HHH\big(\RPerf^{G,\otimes}_{A}\big)$-module $\HHH(\CCC)$, and $\HHH(A)\simeq i_!(A)_C$
determine an~object of~\begin{gather*}
\Alg_{\etwo}(\Fun(BG,\Mod_{R}))\!\times_{\Alg_{\assoc}(\Fun(BS^1,\Fun(BG,\Mod_R)))}\!\LMod\!\big(\Fun\!\big(BS^1\!,\Fun(BG,\Mod_R)\big)\big),
\end{gather*}
which is naturally equivalent to $\Alg_{\KS}(\Fun(BG,\Mod_R))$
(see Corollary~\ref{algebraequivalence}).
Consequently,
we~obtain a~$\KS$-algebra in $\Fun(BG,\Mod_R)$ having the property
described in Remark~\ref{equivrem}.
\end{Construction}

\begin{Remark}\label{remarkfunctoriality}The argument in Construction~\ref{equivariantconstruction} can be applied to show
other functorialities.
For example, suppose that we are given an~equivalence $f\colon \CCC_1\stackrel{\sim}{\to} \CCC_2$ in $\ST_R$.
Then it gives rise to an~equivalence
\begin{gather*}
(\HH^\bullet(\CCC_1),\HH_\bullet(\CCC_1))\stackrel{\sim}{\longrightarrow} (\HH^\bullet(\CCC_2),\HH_\bullet(\CCC_2))
\end{gather*}
 in $\Alg_{\KS}(\Mod_R)$.
To see this, consider the $\infty$-category $I$ which consists of~two objects $\{1,2\}$
such that for any $i,j\in \{1,2\}$, the mapping space $\Map_{I}(i,j)$ is a~contractible space
(so that $I$ is (the nerve of) an~ordinary groupoid).
The equivalence $f$ amounts to a~functor $I\to \ST_R$ which carries
$1$ and $2$ to $\CCC_1$ and $\CCC_2$, respectively, and carries
the unique morphism $1\to 2$
to $f$. Namely, we have an~object of~$\Fun(I,\ST_R)$ so that
the construction of~Ind-categories gives rise to an~object $F\in \Fun\left(I,\PR_R\right)$.
Now apply the argument in Construction~\ref{equivariantconstruction} by~replacing
$BG$ with $I$:
Namely, we first
consider the internal hom object from $F$ to $F$ in $\Fun\left(I,\PR_R\right)$.
Note that by~\cite[Corollary~3.3.3.2]{HTT} and the fact that $I$ is a~Kan complex (i.e., $\infty$-groupoid), $\Fun\left(I,\PR_R\right)$ can be identified with a~limit of~the constant diagram
$I\to \wCat$ with value $\PR_R$.
There exists an~internal hom object from $F$ to $F$ that is given by~the tensor product $F\otimes F^{\vee}$ of~$F$
and its dual object $F^{\vee}$, where $F^{\vee}$ is given by~taking termwise dual objects.
According to~\cite[Corollary~4.7.1.40]{HA}, $F\otimes F^{\vee}$ is promoted to an~endomorphism
algebra $\mathcal{E}{\rm nd}(F)\in \Alg_{\assoc}\big(\Fun\big(I,\PR_R\big)\big)$ equipped with a~left action on $F$ in an~essentially unique way.
The evaluation at each $1,2\in I$ gives us
the endomorphism algebras $\mathcal{E}{\rm nd}_R(\Ind(\CCC_1))$ and $\mathcal{E}{\rm nd}_R(\Ind(\CCC_2))$, respectively.
Thus, replacing $BG$ with $I$ in the procedure in Construction~\ref{equivariantconstruction},
we obtain an~object in
$\Alg_{\KS}(\Fun(I,\Mod_R))\simeq \Fun(I,\Alg_{\KS}(\Mod_R))$, which brings us
an induced equivalence
$(\HH^\bullet(\CCC_1),\HH_\bullet(\CCC_1))\stackrel{\sim}{\rightarrow} (\HH^\bullet(\CCC_2),\HH_\bullet(\CCC_2))$.
Finally, we remark that this argument can be applied to any diagram
$K\to \ST_R$ such that $K$ is a~Kan complex (i.e., an~$\infty$-groupoid).
That is, if we are given a~diagram/functor $P\colon K\to \ST_R$ from an~$\infty$-groupoid $K$,
it gives rise to a~diagram $P_{\KS}\colon K\to \Alg_{\KS}(\Mod_R)$ of~$\KS$-algebras
such that for any $k\in K$, $P_{\KS}(k)$ is equivalent to the $\KS$-algebra
constructed from $P(k)\in \ST_R$ in~Theorem~\ref{mainconstruction}.
\end{Remark}

\subsection*{Acknowledgements}

The author would like to thank Takuo Matsuoka for valuable and inspiring
conversations related to the subject of~this paper.
He would like to thank everyone
who provided constructive feedback at
his talks about the main content of~this paper.
He would also like to thank the referees for~their useful suggestions.
This work is supported by~JSPS KAKENHI grant 17K14150.

\pdfbookmark[1]{References}{ref}
\LastPageEnding

\end{document}